\documentclass[reqno,a4paper,12pt]{amsart}
\usepackage{amsmath,amssymb,amsthm,dsfont,bbm}
\usepackage{graphicx}
\usepackage[english]{babel}
\usepackage[usenames]{color}
\usepackage{times1}
\usepackage{subfigure}
\usepackage{cancel}
\usepackage[normalem]{ulem}

\RequirePackage[colorlinks,linkcolor=black,citecolor=black,urlcolor=blue]{hyperref}

\def\MR#1{\href{http://www.ams.org/mathscinet-getitem?mr=#1}{MR#1}}

\newtheorem{theorem}{Theorem}[section]
\newtheorem{lemma}[theorem]{Lemma}

\newtheorem{corollary}[theorem]{Corollary}

\theoremstyle{definition}
\newtheorem{definition}{Definition}[section]
\newtheorem{assumption}[definition]{Assumption}

\theoremstyle{remark}
\newtheorem{remark}{Remark}[section]
\numberwithin{equation}{section}
\newcommand{\RR}{\mathbb{R}}
\newcommand{\CC}{\mathbb{C}}

\newcommand{\ZZ}{\mathbb{Z}}
\newcommand{\NN}{\mathbb{N}}
\newcommand{\PP}{\mathbb{P}}
\newcommand{\EE}{\mathbb{E}}

\newcommand{\Beta}{\mathrm{Beta}\mypp}
\newcommand{\Gammad}{\mathrm{Gamma}\mypp}
\newcommand{\GEM}{\mathrm{GEM}\myp}
\newcommand{\Psp}{\mathbb{P}_{N\myn,\myp L}}
\newcommand{\Esp}{\mathbb{E}_{N\myn,\myp L}}

\newcommand{\Ens}{\tilde{\mathbb{E}}_{N}}
\newcommand{\Pns}{\tilde{\mathbb{P}}_{N}}
\newcommand{\bfOne}{\mathbbm{1}}

\newcommand{\rmd}{{\rm d}}
\newcommand{\rme}{{\rm e}}
\newcommand{\rmi}{{\rm i}}
\newcommand{\bfx}{\boldsymbol{x}}
\newcommand{\bfs}{\boldsymbol{s}}
\newcommand{\bfk}{\boldsymbol{k}}
\newcommand{\bfzero}{\boldsymbol{0}}

\newcommand{\GN}{G_{\myn N}}
\newcommand{\PD}{\mbox{\rm PD}}
\newcommand{\SN}{\mathfrak{S}_N}
\newcommand{\Sn}{\mathfrak{S}_n}

 \DeclareMathOperator{\Li}{Li}
\newcommand{\const}{\mathrm{const}}

\newcommand{\myp}{\mbox{$\:\!$}}
\newcommand{\mypp}{\mbox{$\;\!$}}

\newcommand{\myn}{\mbox{$\;\!\!$}}
\newcommand{\mynn}{\mbox{$\:\!\!$}}
\newcommand{\mynnn}{\mbox{$\!$}}

\begin{document}
\title[Cycles in surrogate-spatial random permutations]{Asymptotic
statistics of cycles in surrogate-spatial permutations}

\author[L.\mypp{}V.\myp~Bogachev]{Leonid V.\myp\ Bogachev}
\address{Department of Statistics, School of Mathematics, University of Leeds, Leeds LS2 9JT, UK}
\email{L.V.Bogachev@leeds.ac.uk}

\author[\myp{}D.~Zeindler]{Dirk Zeindler}
\address{Department of Mathematics, University of Bielefeld, Bielefeld D-33501,
Germany} \email{zeindler@math.uni-bielefeld.de}

\begin{abstract}
We propose an extension of the Ewens measure on permutations by
choosing the cycle weights to be asymptotically proportional to the
degree of the symmetric group. This model is primarily motivated by
a natural approximation to the so-called spatial random permutations
recently studied by V.~Betz and D.~Ueltschi (hence the name
``surrogate-spatial''), but it is of substantial interest in its own
right. We show that under the suitable (thermodynamic) limit both
measures have the similar critical behaviour of the cycle statistics
characterized by the emergence of infinitely long cycles. Moreover,
using a greater analytic tractability of the surrogate-spatial
model, we obtain a number of new results about the asymptotic
distribution of the cycle lengths (both small and large) in the full
range of subcritical, critical, and supercritical domains. In
particular, in the supercritical regime there is a parametric
``phase transition'' from the Poisson--Dirichlet limiting
distribution of ordered cycles to the occurrence of a single giant
cycle. Our techniques are based on the asymptotic analysis of the
corresponding generating functions using P\'olya's Enumeration
Theorem and complex variable methods.
\end{abstract}

\maketitle


\medskip 

\tableofcontents

\section{Introduction}

\subsection{Surrogate-spatial permutations}\label{sec:1.1}
Let $\SN$ be the symmetric group of all permutations on elements
$1,\dots,N$. For any permutation $\sigma\in \SN$, denote by
$C_j=C_j(\sigma)$ the \emph{cycle counts}, that is, the number of
cycles of length $j=1,\dots,N$ in the cycle decomposition
of~$\sigma$; clearly
\begin{equation}\label{eq:sumC}
C_j\ge0 \quad(j\ge1),\qquad \sum_{j=1}^N j\mypp C_j=N.
\end{equation}
A probability measure on $\SN$ with (multiplicative) cycle weights
may now be introduced by the expression
\begin{equation}\label{eq:weights}
\PP_N(\sigma):=\frac{1}{N!\mypp H_N}\prod_{j=1}^N w_j^{C_j},\qquad
\sigma\in\SN,
\end{equation}
where $H_N$ is the normalization constant,
\begin{equation}\label{eq:weights_H}
H_N=\frac{1}{N!}\sum_{\sigma\in\SN} \prod_{j=1}^N w_j^{C_j}.
\end{equation}

The cycle weights $w_j$ \,($j=1,\dots,N$) in \eqref{eq:weights},
\eqref{eq:weights_H} may in principle also depend on the degree $N$.
In the simplest possible case one just puts $w_j\equiv 1$, resulting
in the classical \emph{uniform distribution} on permutations dating
back to Cauchy (see, e.g., \cite[\S\myp1.1]{ABT03}). Generalization
with $w_j\equiv \theta>0$ is known as the \emph{Ewens sampling
formula} $\mathrm{ESF}(\theta)$, which first emerged in the study of
population dynamics in mathematical biology~\cite{Ewens}. A class of
models with variable coefficients $w_j$ (but independent of $N$) was
recently studied in papers \cite{BeUeVe11}, \cite{ErUe12},
\cite{MaNiZe12}, \cite{NiZe13} (see also an extensive background
bibliography therein); more general models of assemblies allowing
for mild dependence of the weights on $N$ were considered earlier by
Manstavi\v{c}ius (see \cite[\S\,2, pp.~67--68]{Man02}).

The \emph{surrogate-spatial measure} $\Pns$ proposed in the present
paper is a further natural generalization specified by choosing the
cycle weights in the form
\begin{equation}\label{eq:weights-SS}
w_j(N):=\theta_j+N\kappa_j,\qquad j\in\NN\myp,
\end{equation}
where $(\theta_j)$ and $(\kappa_j)$ are given sequences with
$\theta_j\ge0$, $\kappa_j\ge0$ \,($j\in\NN$). Thus, the general
model \eqref{eq:weights}, \eqref{eq:weights_H} specializes to
\begin{gather}\label{eq:def_near_spatial}
\Pns(\sigma):= \frac{1}{ N!\mypp H_{N}} \prod_{j=1}^N (\theta_j
+ N\kappa_j)^{C_j},\qquad \sigma\in \SN,\\
\label{eq:HN} H_N=\frac{1}{N!}\sum_{\sigma\in\SN} \prod_{j=1}^N
(\theta_j + N\kappa_j)^{C_j}.
\end{gather}

\subsection{Spatial permutations}\label{sec:1.2}
Our original interest to the model \eqref{eq:def_near_spatial},
\eqref{eq:HN} (which also explains the proposed name
``surrogate-spatial'') has been generated by the so-called
\emph{spatial random permutations}, introduced and studied by Betz
and Ueltschi in a series of papers \cite{BeUe09}, \cite{BeUe10},
\cite{BeUe11}, \cite{BeUe11a} (see also a recent preprint by
Ercolani et al.\
\cite{ErJaUe14} developing a more general setting of \emph{spatial
random partitions}). Specifically, the spatial model leads to the
following family of probability measures on the symmetric group
$\mathfrak{S}_N$,
\begin{equation} \label{eq:Pspatial_with_partition}
\Psp(\sigma):=
 \frac{1}{N!\mypp H_{N\myn,\myp L}} \prod_{j=1}^N \left(\rme^{-\alpha_j} \sum_{\bfk\in \ZZ^d}
   \rme^{-j\myp\varepsilon(\bfk/L)}\right)^{C_j},\qquad \sigma\in\SN,
\end{equation}
where $(\alpha_j)$ is a real sequence, $L>0$ is an additional
``spatial'' parameter, \,$\varepsilon\colon\RR^d\to[0,\infty)$ is a
certain function, and $H_{N\myn,\myp L}$ is the corresponding
normalization factor.

The hidden spatial structure of the measure
\eqref{eq:Pspatial_with_partition} is revealed by the fact that
$\Psp$ emerges as the $\SN$-marginal of a suitable probability
measure on a bigger space $\varLambda^N\mynn\times \mathfrak{S}_N$,
where $\varLambda:=[-\frac12L,\frac12L\myp]^d\subset \RR^d$;
namely (cf.~\cite[Eq.~(3.6), p.\:1179]{BeUe11a})
\begin{gather}\label{eq:P_N,L}
\Psp(\sigma)=\frac{1}{N!\mypp H_{N\myn,\myp L}}\int_{\varLambda^N}
\,\rme^{-\mathcal{H}_N(\bfx_1,\dots,\myp\bfx_N;
\mypp\sigma)}\,\rmd{\bfx_1}\dots\rmd{\bfx_N},\qquad
\sigma\in\mathfrak{S}_N,\\
\label{eq:H-hamiltonian}
\mathcal{H}_N(\bfx_1,\dots,\bfx_N;\myp\sigma): = \sum_{j=1}^N
V(\bfx_j-\bfx_{\sigma(j)})+\sum_{j=1}^N\alpha_j\myp C_j\,,
\end{gather}
where the interaction potential $V\colon \RR^d\to(-\infty, +\infty]$
is such that the function $\rme^{-V(\bfx)}$ is continuous, has
\emph{positive Fourier transform} (which implies that
$V(-\bfx)=V(\bfx)$, \,$\bfx\in\RR^d$),~and
\begin{equation}\label{eq:density}
\int_{\RR^d} \rme^{-V(\bfx)}\,\rmd{\bfx}=1,
\end{equation}
so that the function $f(\bfx):=\rme^{-V(\bfx)}$ can be interpreted
as a probability density on $\RR^d$.

The basic physical example is the \emph{Gaussian case} with a
quadratic potential (see~\cite[p.\:1175]{BeUe11a})
\begin{equation}\label{eq:GaussianV}
V(\bfx)=\tfrac14\mypp\beta^{-1}\|\bfx\|^2+\tfrac12\myp d\myp
\log\myn(4\pi\beta), \qquad \bfx=(x_1,\dots,x_d)\in\RR^d,
\end{equation}
where $\|\bfx\|:=\left(\sum_{i=1}^d x_i^2\right)^{1/2}$ is the usual
(Euclidean) norm in $\RR^d$, $\beta>0$ is the inverse temperature,
and a constant term in \eqref{eq:GaussianV} insures the
normalization condition \eqref{eq:density}. According to formula
\eqref{eq:H-hamiltonian}, particles in a random spatial
configuration $\{\bfx_1,\dots, \bfx_N\}$ interact with one another
via the spatial potential $V$ only along cycles of an auxiliary
permutation $\sigma\in\SN$ (see
Fig.\;\ref{fig:spatial_permutations}), whereby the existence of a
cycle of length $j\in\NN$ is either promoted or penalized depending
on whether $\alpha_j<0$ or $\alpha_j>0$, respectively.
\begin{figure}[h!]
 \centering
 \includegraphics[width= 0.51\textwidth]{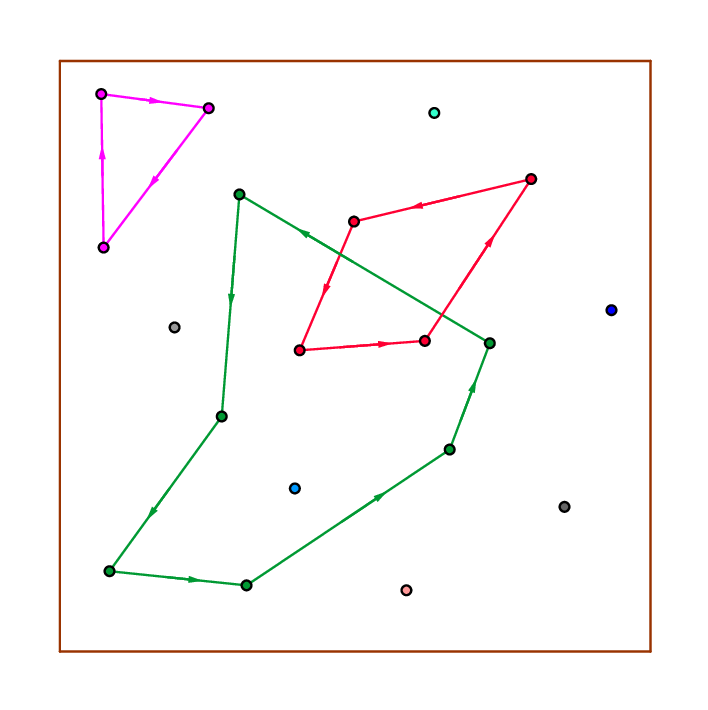}
  \put(-115,4){\mbox{\scriptsize$L$}}
  \put(-12,110){\mbox{\scriptsize$L$}}
  \put(-72,190){\mbox{\scriptsize$N$ (particles)}}
\caption{Illustration of spatial permutations.}
 \label{fig:spatial_permutations}
\end{figure}

The link between formulas \eqref{eq:Pspatial_with_partition} and
\eqref{eq:P_N,L} is provided by the function $\varepsilon(\bfs)$
defined by the Fourier transform
\begin{equation}\label{eq:eps}
\rme^{-\varepsilon(\bfs)}=\int_{\RR^d} \rme^{-2\pi\rmi\myp
(\bfx,\bfs)}\,\rme^{-V(\bfx)}\,\rmd{\bfx},\qquad \bfs\in\RR^d,
\end{equation}
where $(\bfx,\bfs)$ denotes the inner product in $\RR^d$. For
example, the Gaussian potential \eqref{eq:GaussianV} leads to a
quadratic function $\varepsilon(\bfs)=c \myp\|\bfs\|^2$ (with
$c=4\pi^2\beta$).

From the assumptions on $V(\cdot)$, it readily follows that
$\varepsilon(\bfzero)=0$, \,$\varepsilon(-\bfs)=\varepsilon(\bfs)$
\,($\bfs\in\RR^d$), \,$\varepsilon(\bfs)> 0$ \,($\bfs\ne\bfzero$)
and, by the Riemann--Lebesgue lemma,
\,$\lim_{\bfs\to\infty}\varepsilon(\bfs) = \infty$. Since the
Fourier transform \eqref{eq:eps} is positive, a simple lemma (see
\cite[Theorem~9, p.~20]{Bochner} or \cite[Lemma 7.2.1, p.\:162]{IL})
yields that
\begin{equation}\label{eq:int-eps}
\int_{\RR^d} \rme^{-\varepsilon(\bfs)}\,\rmd{\bfs} < \infty,
\end{equation}
hence the Fourier inversion formula implies the dual relation
\begin{equation*}
\rme^{-V(\bfx)}=\int_{\RR^d} \rme^{\myp2\pi\rmi\myp
(\bfx,\bfs)}\,\rme^{-\varepsilon(\bfs)}\,\rmd{\bfs},\qquad
\bfx\in\RR^d.
\end{equation*}
Finally, let us assume that the function $\varepsilon(\bfs)$ is
regular enough as $\bfs\to\infty$, so that the integrability
condition \eqref{eq:int-eps} implies the convergence (for any $L>0$)
of the series
\begin{equation}\label{eq:sum-eps}
\sum_{\bfk\in\ZZ^d} \rme^{-\varepsilon(\bfk/L)}<\infty.
\end{equation}
For the latter, it is sufficient that, for  $\bfs$ large enough,
$\varepsilon(\bfs)=\varepsilon(s_1,\dots,s_d)$ is non-decreasing in
each of the variables $s_i$, or that there is a lower bound
$\varepsilon(\bfs)\ge c_1+\gamma\log \|\bfs\|$, with some $c_1>0$
and $\gamma>d$. Importantly, the condition \eqref{eq:sum-eps}
ensures that the spatial measure \eqref{eq:Pspatial_with_partition}
is well defined.

The model \eqref{eq:P_N,L}, \eqref{eq:H-hamiltonian} is motivated by
the Feynman--Kac representation of the dilute Bose gas (at least in
the Gaussian case), and it has been proposed in connection with the
study of the \emph{Bose--Einstein condensation} (for more details
and the background, see \cite{BeUe09}, \cite{BeUe10}, \cite{BeUe11}
and further references therein). An important question in this
context, which is also interesting from the combinatorial point of
view, is the possible emergence of an infinite cycle under the
\emph{thermodynamic limit}, that is, by letting $N,L\to\infty$ while
keeping constant the \emph{particle density} $\rho= NL^{-d}$. The
(expected) fraction of points contained in infinitely long cycles
can be defined as
\begin{equation}\label{eq:nu}
\nu := \lim_{K\to \infty} \liminf_{N \to \infty}
\frac{1}{N}\,\Esp\!\mynn\left(\myn\sum_{j>K} j\mypp C_j \right).
\end{equation}
Betz and Ueltschi have shown in \cite{BeUe11a} that in the model
\eqref{eq:Pspatial_with_partition}, under certain assumptions on the
coefficients $\alpha_j$, the quantity $\nu$ is identified explicitly
as
\begin{equation}\label{eq:rho-critical}
\nu = \max\left\{0, \,1 -
\frac{\rho_{\myp\mathrm{c}}}{\rho}\right\},
\end{equation}
where $\rho_{\myp\mathrm{c}}$ is the \emph{critical density} given
by
\begin{equation}\label{eq:Riemann-integral}
\rho_{\myp\mathrm{c}} := \sum_{j=1}^\infty \rme^{-\alpha_j}
\int_{\RR^d} \rme^{-j\myp\varepsilon(\bfs)}\,\rmd{\bfs}\le +\infty.
\end{equation}
That is to say, infinite cycles emerge (in the thermodynamic limit)
when the density $\rho$ is greater than the critical density
$\rho_{\myp\mathrm{c}}$ (see further details in \cite{BeUe11a}).

However, the computations in \cite{BeUe11a} for the original spatial
measure $\Psp$ are quite complicated, and it may not be entirely
clear as to why the asymptotic behaviour of cycles is drastically
different for $\rho<\rho_{\myp\mathrm{c}}$ and
$\rho>\rho_{\myp\mathrm{c}}$ (even though intuition does suggest
such a phase transition; see a heuristic explanation in
\cite[p.\:1175]{BeUe11a}).

\subsection{Surrogate-spatial model as an approximation of the spatial model}\label{sec:1.3}
A simple observation that has motivated the present work is that the
sum in \eqref{eq:Pspatial_with_partition} can be viewed, for each
fixed $j$, as a Riemann sum (with mesh size $1/L$) for the
corresponding integral appearing in~\eqref{eq:Riemann-integral}.
Using Euler--Maclaurin's (multidimensional) summation formula (see,
e.g., \cite[\S\myp A.4]{BRR})
and recalling that $\rho= NL^{-d}$, this suggests the following
plausible approximation of the cycle weights in
formula~\eqref{eq:Pspatial_with_partition},
\begin{equation}\label{eq:approx_weigths}
\rme^{-\alpha_j} \sum_{\bfk\in \ZZ^d} \rme^{-j\myp
\varepsilon(\bfk/L)} = N \kappa_j + \theta_j + o\myp(1),\qquad
N,L\to\infty,
\end{equation}
with the coefficients
\begin{equation}\label{eq:kappa_j}
\kappa_j=\rho^{-1}\myp\rme^{-\alpha_j} \int_{\RR^d} \rme^{-j
\myp\varepsilon(\bfs)}\,\rmd{\bfs},\qquad j\in\NN\myp.
\end{equation}
(Note that, due to the condition \eqref{eq:int-eps}, the integral in
\eqref{eq:kappa_j} is finite for all $j\in\NN$\myp.) Thus,
neglecting the $o$-terms in \eqref{eq:approx_weigths}, we arrive at
the surrogate-spatial model~\eqref{eq:weights-SS}.

The ansatz \eqref{eq:approx_weigths} demands a few comments. In the
Gaussian case, with $\varepsilon(\bfs)=c \myp\|\bfs\|^2$, it can be
checked with the help of the Poisson summation formula
\cite[\S\myp3.12, p.~52]{Bru}
that the expansion \eqref{eq:approx_weigths} holds true for any
fixed $j\in\NN$ with $\kappa_j\propto \rme^{-\alpha_j} j^{-d/2}$ and
$\theta_j\equiv0$ (see Section~\ref{sec:6.1} below). In general,
however, it may not be obvious that the leading correction to the
principal term in \eqref{eq:approx_weigths} is necessarily
\emph{constant in $N$}.

More importantly, the expansion \eqref{eq:approx_weigths} with the
integral coefficients \eqref{eq:kappa_j} may fail to be adequate if
the index $j\le N$ grows fast enough with $N$. For instance, again
assuming the Gaussian case, for the sum in \eqref{eq:approx_weigths}
with $j=N$ we have in dimension $d\ge3$
\begin{equation*}
\sum_{\bfk\in \ZZ^d} \rme^{-Nc\mypp\|\bfk/L\|^2}=\sum_{\bfk\in
\ZZ^d} \exp\left\{-c\myp\rho^{2/d}N^{1-2/d}\|\bfk\|^2\right\}\to
1,\qquad N\to\infty,
\end{equation*}
whereas the integral counterpart (see \eqref{eq:kappa_j}) is
asymptotic to $O(N^{1-2/d})=o(1)$. This suggests that for $j$ close
to $N$ the relation \eqref{eq:approx_weigths} holds with the
universal (potential-free) coefficients $\kappa_j=0$, \,$\theta_j
=\rme^{-\alpha_j}$. That is to say, the (Gaussian) spatial model
$\Psp$ dynamically interpolates (up to asymptotically small
correction terms) between the surrogate-spatial model $\Pns$ with
$\kappa_j>0$, $\theta_j=0$ (for small $j$) and the one with
$\kappa_j=0$, $\theta_j>0$ (for large $j$).

We postpone a detailed comparison of the models
\eqref{eq:def_near_spatial} and \eqref{eq:Pspatial_with_partition}
until Section~\ref{sec:comparison}. For now let us just point out
that both models, after a suitable calibration, share the same
critical density $\rho_{\myp\mathrm{c}}$ and the expected fraction
of points in infinite cycles (see~\eqref{eq:Riemann-integral});
there are also qualitative similarities in the transition from the
Poisson--Dirichlet distribution of large cycles to a single giant
cycle (see Section~\ref{sec:long}).

\subsection{Synopsis and layout}\label{sec:1.4}

As will be demonstrated in this paper, the class of
surrogate-spatial measures $\Pns$ defined by
\eqref{eq:def_near_spatial} produces a rich picture of the
asymptotic statistics of permutation cycles as $N\to\infty$, being
at the same time reasonably tractable analytically. Thus, despite
lacking direct physical relevance, it can be used as an efficient
exploratory tool in the analysis of more complicated spatial models.

More specifically, the asymptotics of the surrogate-spatial model
can be characterized using the singularity analysis of the
generating functions of the sequences $(\theta_j/j)$ and
$(\kappa_j/j)$,
\begin{equation}\label{eq:def_g_theta_and_p_tau}
g_\theta(z):=\sum_{j=1}^\infty \frac{\theta_j}{j}\, z^j,\qquad
g_\kappa(z):=\sum_{j=1}^\infty \frac{\kappa_j}{j}\, z^j\qquad
(z\in\CC).
\end{equation}
In particular, the emergence of \emph{criticality} is determined by
the condition $R\mypp g'_\kappa(R\myp)<1$, where $R>0$ is the radius
of convergence of the power series $g_\kappa(z)$; this inequality
defines the \emph{supercritical} regime, with the \emph{subcritical}
counterpart represented by the opposite inequality, $R\mypp
g'_\kappa(R\myp)>1$. In the limit $N\to\infty$, the criticality is
manifested as a phase transition in the cycle statistics when the
system passes through the critical point.

\begin{remark}
The phase transition can be parameterized by scaling the
coefficients $\kappa_j\mapsto \varrho^{-1} \kappa_j$
\,($\varrho>0$). Then the critical value of the parameter $\varrho$
is given by $\varrho_{\myp\rm c}=R\mypp g'_\kappa(R\myp)$, with the
subcritical and supercritical domains corresponding to the intervals
$0<\varrho<\varrho_{\myp\rm c}$ and $\varrho>\varrho_{\myp\rm c}$,
respectively. We will argue in Section \ref{sec:6.2} below that the
scaling parameter $\varrho$ can be viewed as an analogue of the
particle density (cf.\ \eqref{eq:kappa_j}).
\end{remark}

\begin{remark}
We will routinely assume that $g_\theta(z)$ is analytic in the disk
$|z|<R$\myp, which allows us to distill the role of the leading
sequence $(\kappa_j)$ in the surrogate-spatial model
\eqref{eq:weights-SS}. However, $g_\theta(z)$ will be permitted to
have a singularity at $z=R$ (usually of a logarithmic type), which
may have significant impact on the asymptotic statistics of
\emph{long} cycles (see Section~\ref{sec:long}).
\end{remark}

The first evidence of the different asymptotic statistics of cycles
in the subcritical \emph{vs.}\ supercritical regimes is provided by
the (weak) law of large numbers for individual cycle counts,
$C_j/N\to \kappa_j\myp r_*^j/j$ as $N\to\infty$, where $r_*>0$ is
the (unique) root of the equation $r g'_\kappa(r)=1$ if $R\mypp
g'_\kappa(R\myp)\ge 1$ and $r_*\myn:=R$ otherwise (see
Theorem~\ref{thm:cycle_counts_normalized}).

This result is heuristically matched by the law of large numbers for
the \emph{total number of cycles} $T_N=\sum_{j=1}^\infty C_j$\myp,
stating that $T_N/N\to \sum_{j=1}^\infty \kappa_j\myp
r_*^j/j=g_\kappa(r_*)$. In all cases, fluctuations of $T_N$ appear
to be of order of $\sqrt{N}$, and a central limit theorem holds in
both subcritical and supercritical regimes
(Theorem~\ref{thm:asymptotic_Kon}\myp(a),~(b)). The critical case
can be studied as well (see
Theorem~\ref{thm:asymptotic_Kon}\myp(c)), but the situation there is
more complicated, being controlled by the analytic structure of the
generating functions $g_\kappa(z)$ and $g_\theta(z)$ at the
singularity point $z=R$\myp; in particular, if $g_\theta(z)$ has a
non-degenerate logarithmic singularity (e.g.,
$g_\theta(z)=-\theta^*\myn\log\myn(1-z/R\myp)$ with
$\theta^*\myn>0$, which corresponds to a geometric sequence
$\theta_j=\theta^* R^{-j}$), then it is possible that the limiting
distribution of $T_N$ includes an independent gamma-distributed
correction to a Gaussian component (see
Theorem~\ref{thm:asymptotic_Kon}\myp(c-iii)).

The criticality is further manifested by a phase transition in the
fraction of points contained in \emph{long cycles}
(cf.~\eqref{eq:nu}); namely, we will see (Theorems
\ref{thm:fractional_infinite}
and~\ref{thm:fractional_infinite_random}) that such a fraction is
asymptotically given by $\tilde{\nu}= 1-r_*\myp g'_\kappa(r_*)$ and
so is strictly positive in the supercritical regime. The quantity
$r_*\myp g'_\kappa(r_*)\le 1$ can also be interpreted (see
Theorem~\ref{thm:unordered_sub}) as the full mass of the limiting
distribution of individual cycle lengths (under a convenient
ordering called \emph{lexicographic}).

The statistics of long cycles emerging in the supercritical regime
will be investigated in Section~\ref{sec:long}. Our main result
there is Theorem~\ref{thm:large_cycles_1} stating that if the
generating function $g_\theta(z)$ has a logarithmic singularity at
$z=R$ with $\theta^*\myn>0$ (cf.\ above), then the cycle lengths,
arranged in decreasing order and normalized by $N\tilde{\nu}$,
converge
to the Poisson--Dirichlet distribution with parameter $\theta^*$.
Furthermore, if $\theta^*\myn=0$ then the limiting distribution of
the cycle order statistics is reduced to the deterministic vector
$(1,0,0,\dots)$, meaning that there is a \emph{single giant cycle}
of length $N\tilde{\nu}\mypp(1+o(1))$ (see
Theorem~\ref{thm:large_cycles_0}).

The rest of the paper is organized as follows.
Section~\ref{sec:conj_class_of_Sn} contains the necessary
preliminaries concerning certain generating functions, including a
basic identity deriving from P\'olya's Enumeration Theorem (see
Lemma~\ref{lem:cycle_index_theorem}). In
Section~\ref{sec:asymptotic_theorem}, with the help of complex
analysis we prove some basic theorems enabling us to compute the
asymptotics of (the coefficients of) the corresponding generating
functions. In Section~\ref{sec:cycles}, we apply these techniques to
study the cycle counts, the total number of cycles and also the
asymptotics of lexicographically ordered cycles. In
Section~\ref{sec:long}, we study the asymptotic statistics of long
cycles. Finally, we compare our surrogate-spatial model with the
original spatial model in Section~\ref{sec:comparison}.

\section{Preliminaries}\label{sec:conj_class_of_Sn}

\subsection{Generating functions}\label{sec:gfs}
We use the standard notation $\ZZ$ and $\NN$ for the sets of integer
and natural numbers, respectively, and also denote
$\NN_0:=\{j\in\ZZ:\,j\ge0\}=\{0\}\cup\NN$\myp.

For a sequence of complex numbers $(a_j)_{j\ge 0}$, its (ordinary)
generating function is defined as the (formal) power series
\begin{align}\label{eq:G}
g(z): = \sum_{j=0}^\infty a_j z^j,\qquad z\in\CC.
\end{align}
As usual \cite[\S\myp{}I.1, p.\:19]{FlSe09}, we define the
\emph{extraction symbol} $[z^j\myp]\, g(z):= a_j$, that is the
coefficient of $z^j$ in the power series expansion \eqref{eq:G}
of~$g(z)$.

The following simple lemma known as \emph{Prings\-heim's Theorem}
(see, e.g., \cite[Theorem IV.6, p.~240]{FlSe09}) is important in the
singularity analysis of asymptotic enumeration problems, where
generating functions with non-negative coefficients are usually
involved.
\begin{lemma}\label{lm:Pringsheim}
Assume that $a_j\ge0$ for all $j\ge0$, and let the series expansion
\eqref{eq:G} have a finite radius of convergence $R$\myp. Then the
point $z = R$ is a singularity of the function $g(z)$.
\end{lemma}

The generating functions with the coefficients $(a_j/j)$ (cf.\
$g_\theta(z)$ and $g_\kappa(z)$ introduced in
\eqref{eq:def_g_theta_and_p_tau}) are instrumental in the context of
permutation cycles due to the following important result. Recall
that the cycle counts $C_j=C_j(\sigma)$ are defined as the number of
cycles of length $j\in\NN$ in the cycle decomposition of permutation
$\sigma\in\SN$ (see the Introduction). The next well-known identity
is a special case of the general \emph{P\'olya's Enumeration
Theorem} \cite[\S\myp16, p.\:17]{Polya}; we give its proof for the
reader's convenience (cf., e.g., \cite[p.~5]{Mac95}).
\begin{lemma}\label{lem:cycle_index_theorem}
Let $(a_j)_{j\in\NN}$ be a sequence of \textup{(}real or
complex\textup{)} numbers. Then there is the following
\textup{(}formal\textup{)} power series expansion
\begin{equation}\label{eq:symm_fkt}
\exp\left(\sum_{j=1}^{\infty}\frac{a_jz^j}{j}\right)
=\sum_{n=0}^\infty\frac{z^n}{n!}\sum_{\sigma\in\Sn}\prod_{j=1}^{n}
a_{j}^{C_j},
\end{equation}
where $C_j=C_j(\sigma)$ are the cycle counts. If either of the
series in \eqref{eq:symm_fkt} is absolutely convergent then so is
the other one.
\end{lemma}
\begin{proof}
Dividing all permutations $\sigma\in\Sn$ into classes with the same
cycle type $(c_j):=(c_1,\dots,c_n)$, that is, such that
$C_j(\sigma)=c_j$ \,($j=1,\dots,n$), we have
\begin{equation}\label{eq:cycle_types}
\sum_{\sigma\in\Sn}\prod_{j=1}^{n} a_{j}^{C_j}=\sum_{(c_j)}
N_{(c_j)}\prod_{j=1}^n a_j^{c_j},
\end{equation}
where $\sum_{(c_j)}$ means summation over non-negative integer
arrays $(c_j)$ satisfying the condition $\sum_{j=1}^n j\myp c_j=n$
\,(see~\eqref{eq:sumC}), and $N_{(c_j)}$ denotes the number of
permutations with cycle type $(c_j)$.

Furthermore, allocating the elements $1,\dots,n$ to form a given
cycle type $(c_j$) and taking into account that (i) each cycle is
invariant under cyclic rotations (thus reducing the initial number
$n!$ of possible allocations by a factor of $\prod_{j=1}^n
j^{c_j}$\myp), and (ii) cycles of the same length can be permuted
among themselves (which leads to a further reduction by a factor of
$\prod_{j=1}^n c_j!$\,), it is easy to obtain Cauchy's formula (see,
e.g., \cite[\S\myp1.1]{ABT03})
\begin{equation}\label{eq:N(type)}
N_{(c_j)}=\frac{n!}{\prod_{j=1}^n j^{c_j} c_j!}\mypp.
\end{equation}
Hence, using \eqref{eq:cycle_types} and \eqref{eq:N(type)} and
noting that $z^n=\prod_{j=1}^n z^{jc_j}$, the right-hand side of
\eqref{eq:symm_fkt} is rewritten as
$$
1+\sum_{n=1}^\infty\prod_{j=1}^{n}
\frac{1}{c_j!}\left(\frac{a_{j}z^{j}}{j}\right)^{c_j}
=\prod_{j=1}^\infty \sum_{k=0}^\infty
\frac{1}{k!}\left(\frac{a_{j}z^{j}}{j}\right)^{k}=\prod_{j=1}^\infty
\exp\!\left(\frac{a_jz^{j}}{j}\right),
$$
which coincides with the left-hand side of~\eqref{eq:symm_fkt}.

The second claim of the lemma follows by the dominated convergence
theorem.
\end{proof}

Lemma~\ref{lem:cycle_index_theorem} can be used to obtain a
convenient expression for the normalization constant $H_N$ involved
in the definition of the surrogate-spatial measure $\Pns$ (see
\eqref{eq:def_near_spatial},~\eqref{eq:HN}). More generally, set
$h_0(N):=1$ and
\begin{align}\label{eq:hn(v)_def}
h_n(N) := \frac{1}{n!}\sum_{\sigma\in \Sn} \prod_{j=1}^n (\theta_j +
N \kappa_j)^{C_j},\qquad n\in\NN\myp.
\end{align}
In particular (see \eqref{eq:HN}) we have $H_N=h_N(N)$. By
Lemma~\ref{lem:cycle_index_theorem} (with $a_j=\theta_j+N\kappa_j$),
it is immediately seen that the generating function of the sequence
$(h_n(N))_{n\in\NN_0}$ is given by
\begin{equation}\label{eq:generating_h(v)}
\sum_{n=0}^\infty h_n(N)\mypp z^n = \exp\myn\left\{\sum_{j=1}^\infty
\frac{\theta_j+N\kappa_j}{j}\,z^j\right\}\equiv\rme^{\myp\GN(z)},
\end{equation}
where we define the function
\begin{equation}\label{eq:GN}
\GN(z):=g_\theta(z) + N g_\kappa(z),
\end{equation}
with the generating functions $g_\theta(z)$ and $g_\kappa(z)$
defined in \eqref{eq:def_g_theta_and_p_tau}. Hence, the coefficients
$h_n(N)$ can be represented as
\begin{equation}\label{eq:Hn_generatingN}
h_{n}(N)=[z^{n}]\,\rme^{\myp\GN(z)},\qquad n\in\NN_0.
\end{equation}
In particular, for $n=N$ formula \eqref{eq:Hn_generatingN}
specializes to
\begin{equation}\label{eq:Hn_generatingN1}
H_N\equiv  h_N(N)=[z^N]\,\rme^{\myp\GN(z)}.
\end{equation}

\subsection{Some simple properties of the cycle distribution}\label{sec:prob}
In what follows, we use the Poch\-hammer symbol $(x)_{n}$ for the
falling factorials,
\begin{equation}\label{eq:Pochhammer}
(x)_{n}:=x(x-1)\cdots(x-n+1)\quad(n\in\NN),\qquad (x)_0:=1.
\end{equation}
Denote by $\Ens$ the expectation with respect to the measure $\Pns$
(see \eqref{eq:def_near_spatial}).
\begin{lemma}\label{lem:generatin_moments_factorial_Cm}
For each $m\in \NN$ and any integers $n_1,\dots,n_m\ge0$, we have
\begin{equation}\label{eq:E(C)}
\Ens\!\mynn\left(\mynn\prod_{j=1}^m(C_j)_{n_j}\right) =
\frac{h_{N-K_m}(N)}{H_N}\prod_{j=1}^m\left( \frac{\theta_j+ N
\kappa_j}{j}\right)^{\!n_j},
\end{equation}
where $K_m:=\sum_{j=1}^m j n_j$ and $N\ge K_m$.
\end{lemma}

\begin{proof}
Using the definitions of $h_n(N)$ and $\GN(z)$ (see
\eqref{eq:hn(v)_def} and \eqref{eq:GN}, respectively), differentiate
the identity \eqref{eq:generating_h(v)} $n_j$ times with respect to
$\theta_j$ for all $j=1,\dots,m$ to obtain
\begin{equation}\label{eq:dif-theta}
\sum_{n=0}^\infty \frac{z^n}{n!}\sum_{\sigma\in\Sn}
\prod_{j=1}^m(C_j)_{n_j} \prod_{i=1}^n(\theta_i + N\kappa_i)^{C_i}
=\rme^{\myp\GN(z)}\prod_{j=1}^m (\theta_j + N\kappa_j)^{n_j}\left(
\frac{z^{j}}{j}\right)^{\!n_j}.
\end{equation}
Extracting the coefficient $[z^N](\cdot)$ on both sides of
\eqref{eq:dif-theta} and using \eqref{eq:def_near_spatial} we get
\begin{align*}
\Ens\!\mynn\left(\mynn\prod_{j=1}^m(C_j)_{n_j}\right)& =
\frac{1}{N!\mypp H_N}\sum_{\sigma\in\SN}
\prod_{j=1}^m(C_j)_{n_j} \prod_{i=1}^N(\theta_i + N\kappa_i)^{C_i}\\
&=\frac{1}{H_N}\,[z^N]\!\left(\rme^{\myp\GN(z)}\prod_{j=1}^m\left(
\frac{\theta_j +N\kappa_j}{j}\right)^{\myn n_j} z^{j n_j}
\right)\\
&=\prod_{j=1}^m\left(\frac{\theta_j +N\kappa_j}{j}\right)^{\myn n_j}
\mynn \frac{[z^{N-K_m}] \exp\myn\{\GN(z)\}}{H_N}\mypp,
\end{align*}
and formula \eqref{eq:E(C)} follows on
using~\eqref{eq:Hn_generatingN}.
\end{proof}

Next, let $T_N$ be the total number of cycles,
\begin{equation}\label{eq:T_N}
T_N:= \sum_{j=1}^N C_j.
\end{equation}
\begin{lemma}\label{lem:generating_K0n} For each  $v>0$,
the probability generating function of\/ $T_N$ is given by
\begin{equation}\label{eq:generating_K}
\Ens(v^{\myp T_N})= \frac{1}{H_N}\,
[z^N]\myp\exp\myn\{v\myp\GN(z)\},
\end{equation}
where the expansion of the function $z\mapsto \exp\myn\{v\myp
\GN(z)\}$ on the right-hand side of \eqref{eq:generating_K} is
understood as a formal power series.
\end{lemma}

\begin{proof}
By definition of the measure $\Pns$ (see
\eqref{eq:def_near_spatial}) we have
\begin{align*}
\Ens(v^{\myp T_N})&=\frac{1}{N!\mypp H_N}
\sum_{\sigma\in\SN} v^{\myp C_1+\dots+C_N}\prod_{j=1}^N(\theta_j+N\kappa_j)^{C_j}\\
&= \frac{1}{N!\mypp H_N}
\sum_{\sigma\in\SN}\prod_{j=1}^N(v\myp\theta_j+Nv\kappa_j)^{C_j}.
\end{align*}
The last sum is analogous to the expression \eqref{eq:hn(v)_def},
only with $\theta_j$ and $\kappa_j$ replaced  by $v\myp\theta_j$ and
$v\myp\kappa_j$, respectively. Hence, formula
\eqref{eq:Hn_generatingN} may be used with $v\myp\GN(z)$ in place of
$\GN(z)$, thus readily yielding~\eqref{eq:generating_K}.
\end{proof}

One convenient way to list the cycles (and their lengths) is via the
lexicographic ordering, that is, by tagging them with a suitable
increasing subsequence of elements starting from $1$.
\begin{definition}\label{def:L-lex}
For permutation $\sigma\in\Sn$ decomposed as a product of cycles,
let $L_1 = L_1(\sigma)$ be the length of the cycle containing
element~$1$, $L_2=L_2(\sigma)$ the length of the cycle containing
the smallest element not in the previous cycle, etc. The sequence
$(L_j)$ is said to be \emph{lexicographically ordered}.
\end{definition}

It is easy to compute the joint (finite-dimensional) distribution of
the lengths $L_j$.

\begin{lemma}\label{lem:distribution_of_ell_1,ell_2}
For each $m\in\NN$ and any $\ell_1,\dots,\ell_m\in\NN$
\,\textup{(}with $\ell_0:=0$\textup{)}, we have
\begin{equation}\label{eq:L=l}
\Pns\{L_1 = \ell_1,\dots, L_m = \ell_m\}= \prod_{j=1}^m
\frac{\theta_{\ell_j} + N\kappa_{\ell_j}}{N-
\ell_1-\dots-\ell_{j-1}}\cdot
\frac{h_{N-\ell_1-\dots-\ell_m}(N)}{H_N}\mypp.
\end{equation}
In particular, for $m=1$
\begin{equation}\label{eq:dist_L1}
\Pns\{L_1 = \ell\}= \frac{\theta_{\ell} + N\kappa_{\ell}}{N} \cdot
\frac{h_{N-\ell}(N)}{H_N}\mypp.
\end{equation}
\end{lemma}

\begin{proof} Note that there are $(N-1)\cdots(N-\ell+1)=(N-1)_{\ell-1}$
cycles of length $\ell$ containing the element~$1$, and the choice
of such a cycle does not influence the cycle lengths of the
remaining cycles. Using the definition \eqref{eq:hn(v)_def} of
$h_n(N)$ we get
\begin{align*}
\Pns\{L_1 = \ell\} &= \frac{  (N-1)_{\ell-1} (\theta_{\ell} +
N\kappa_{\ell})\cdot (N-\ell)! \,h_{N-\ell}(N)}{N!\mypp H_N}\\
&= \frac{\theta_{\ell} + N\kappa_{\ell}}{N} \cdot
\frac{h_{N-\ell}(N)}{H_N}\mypp,
\end{align*}
which proves the lemma for $m=1$ (see~\eqref{eq:dist_L1}).
Similarly, for $m=2$
\begin{align}
\notag
\Pns\{L_1 = \ell_1,L_2=\ell_2\} &=
\begin{aligned}[t](N-1)_{\ell_1-1} (\theta_{\ell_1} +
N\kappa_{\ell_1})&\cdot (N-\ell_1-1)_{\ell_2-1}
(\theta_{\ell_2} + N\kappa_{\ell_2})\\
& \times  (N-\ell_1-\ell_2)!
\,\,\frac{h_{N-\ell_1-\ell_2}(N)}{N!\mypp H_N}
\end{aligned}\\
&= \frac{(\theta_{\ell_1} + N\kappa_{\ell_1})(\theta_{\ell_2} +
N\kappa_{\ell_2})}{N(N-\ell_1)} \cdot
\frac{h_{N-\ell_1-\ell_2}(N)}{H_N} \label{eq:N-L1}
\end{align}
(cf.~\eqref{eq:L=l}). The general case $m\in\NN$ is handled in the
same manner.
\end{proof}

\section{Asymptotic Theorems for the Generating
Function}\label{sec:asymptotic_theorem}

In this section, we develop complex-analytic tools for computing the
asymptotics of the coefficient $h_{N}(N)=H_N$ in the power series
expansion of $\exp\myn\{\GN(z)\}$ (see \eqref{eq:Hn_generatingN}).
More generally, it is useful to consider expansions of the function
$\exp\myn\{v\myp\GN(z)\}$, with some parameter $v>0$. From Lemma
\ref{lem:generatin_moments_factorial_Cm} it is clear that the case
$v=1$ is of primary importance, but Lemma \ref{lem:generating_K0n}
suggests that information for $v\approx 1$ will also be needed for
the sake of limit theorems for cycles (see Sections \ref{sec:cycles}
and~\ref{sec:long}). General considerations in Sections
\ref{sec:3.1}\mypp--\mypp\ref{sec:critical} below are illustrated by
a few important examples (Section~\ref{sec:examples}).

\subsection{Preliminaries and motivation}\label{sec:3.1}
Let us introduce notation for the ``modified'' derivatives of a
function $z\mapsto g(z)$,
\begin{equation}\label{eq:mod_d}
g^{\{n\}}\myn(z):= z^n\mypp \frac{\rmd^n g(z)}{\rmd z^n}\mypp,\qquad
n\in\NN_0.
\end{equation}
For the generating functions $g_\theta(z)$, $g_\kappa(z)$ (see
\eqref{eq:def_g_theta_and_p_tau}) it is easy to see that, for each
$n\in\NN$\myp,
\begin{equation}\label{eq:a}
g_\theta^{\{n\}}\myn(z)=\sum_{j=1}^\infty (j-1)_{n-1}\mypp\theta_j
\mypp z^j,\qquad g_\kappa^{\{n\}}\myn(z)=\sum_{j=1}^\infty
(j-1)_{n-1}\mypp\kappa_j \mypp z^j,
\end{equation}
where $(\cdot)_{n-1}$ is the Pochhammer symbol defined
in~\eqref{eq:Pochhammer}.

Let $R>0$ (possibly $R=+\infty$) be the radius of convergence of the
power series $g_\kappa(z)$, and hence of each of its (modified)
derivatives $g_\kappa^{\{n\}}\myn(z)$. If $R<\infty$ then, according
to Prings\-heim's Theorem (see Lemma~\ref{lm:Pringsheim}), $z=R$ is
a point of singularity of $g_\kappa(z)$ (and each
$g_\kappa^{\{n\}}\myn(z)$, see~\eqref{eq:a}). We write
$g_\kappa^{\{n\}}\myn(R\myp):= \lim_{r\uparrow R}
\,g_\kappa^{\{n\}}\myn(r)$, with
$g_\kappa^{\{n\}}\myn(R\myp):=+\infty$ if this limit is divergent.
For $r\in(0,R\myp]$, let us also denote
\begin{equation}\label{eq:b12}
b_1(r):= g_\kappa^{\{1\}}\myn(r)>0,\qquad b_2(r):=
g_\kappa^{\{1\}}\myn(r)+g_\kappa^{\{2\}}\myn(r)>0.
\end{equation}
The following simple lemma will be useful.

\begin{lemma}\label{lm:g<g}
For any\/ $r\in(0,R\myp]$,
\begin{equation}\label{eq:g<g}
b_1(r)\le \sqrt{g_\kappa(r)\,b_2(r)}\,,
\end{equation}
where the inequality is in fact strict unless $\kappa_j=0$ for all\/
$j\ge2$.
\end{lemma}
\begin{proof} Using the expression \eqref{eq:a} for the modified
derivatives of $g_\kappa$ and applying the Cauchy--Schwarz
inequality, we have
\begin{align*}
g_\kappa^{\{1\}}\myn(r)=\sum_{j=1}^\infty \kappa_j\mypp
r^j&=\sum_{j=1}^\infty \left(\frac{\kappa_j\mypp
r^j}{j}\right)^{1/2}\!\bigl(j\myp
\kappa_j\mypp r^j\bigr)^{1/2}\\
&\le \left(\sum_{j=1}^\infty \frac{\kappa_j\mypp
r^j}{j}\right)^{1/2}\!\left(\sum_{j=1}^\infty j\myp \kappa_j\mypp
r^j\right)^{1/2}\\
&=\sqrt{g_\kappa(r)}\cdot\sqrt{g_\kappa^{\{1\}}\myn(r)+g_\kappa^{\{2\}}\myn(r)}\,,
\end{align*}
and the inequality \eqref{eq:g<g} follows in view of the
notation~\eqref{eq:b12}. The equality is only possible when
$\kappa_j/j=j\kappa_j$ for all $j\ge1$, which implies that
$\kappa_j\equiv0$ for $j\ge2$.
\end{proof}

To avoid trivial complications (or simplifications), let us impose
\begin{assumption}\label{as:NA}
The sequence $(\kappa_j)$ is assumed to be \emph{non-arithmetic},
that is, there is no integer $j_0>1$ such that the inequality
$\kappa_j\ne0$ would imply that $j$ is a multiple of $j_0$. We also
suppose that $(\kappa_j)$ is non-degenerate, in that $\kappa_j>0$
for some $j\ge2$.
\end{assumption}

\begin{remark}
The assumption of non-arithmeticity simplifies the analysis by
ensuring there is a unique maximum of (the real part of) the
generating function $g_\kappa(z)$ (see Lemma \ref{lm:NA} below).
However, a more general case (with multiple maxima) can be treated
as well, without too much difficulty. Note also that in the original
spatial model (see \eqref{eq:approx_weigths}, \eqref{eq:kappa_j})
all coefficients $\kappa_j$ are positive, hence such a sequence
$(\kappa_j)$ is automatically non-arithmetic.
\end{remark}

In what follows, $\Re\myp(w)$ denotes the real part of a complex
number $w\in\CC$, and $\arg\myn(w)\in(-\pi,\pi]$ is the principal
value of its argument.
\begin{lemma}\label{lm:NA} \textup{(a)} \myp{}Under Assumption
\textup{\ref{as:NA}}, for each $r\in(0,R\myp)$ there is a unique
maximum of the function\/ $t\mapsto\Re\myp(g_\kappa(r\rme^{\myp\rmi
t}))$ over $t\in[-\pi,\pi]$, attained at $t=0$ and equal to
$g_\kappa(r)$. If\/ $g_\kappa(R\myp)<\infty$ then this claim is also
true with $r=R$\myp.

\textup{(b)} \myp{}The same statements hold for the function
$t\mapsto\Re\myp(g_\kappa^{\{1\}}\myn(r\rme^{\myp\rmi t}))$,
\,$t\in[-\pi,\pi]$.
\end{lemma}
\begin{proof} (a) Since the coefficients \strut$\kappa_j$ are non-negative,
$t=0$ is always a point of global maximum of the function
$t\mapsto\Re\myp(g_\kappa(r\rme^{\myp\rmi t}))$. It is easy to see
that the uniqueness of this maximum over $t\in[-\pi,\pi]$ is
equivalent to the property that $g_\kappa(\cdot)$ cannot be written
as $g_\kappa(z)=f(z^{j_0})$ with a holomorphic function $f(\cdot)$
and some integer $j_0>1$, and the latter is true because the
sequence $(\kappa_j)$ is non-arithmetic due to
Assumption~\ref{as:NA}.

(b) The same considerations are valid for the function
$g_\kappa^{\{1\}}\myn(z)=\sum_{j=1}^\infty \kappa_j z^j$ which is
again a power series with non-negative non-arithmetic coefficients.
\end{proof}

\begin{remark}
The uniqueness part of Lemma \ref{lm:NA}\myp(b) may fail for
$g_\kappa^{\{n\}}\myn(z)$ with $n\ge2$, because the
non-arithmeticity property may cease to hold, like in the following
example: $\kappa_1>0$, \,$\kappa_{2j}>0$, \,$\kappa_{2j+1}=0$
\,($j\in\NN$).
\end{remark}

\begin{remark}
Lemma \ref{lm:NA} is akin to the ``Daffodil Lemma'' in
\cite[\S\myp{}IV.6.1, Lemma~IV.1, p.~266]{FlSe09}, but the latter
deals with the maximum of the absolute value rather than the real
part.
\end{remark}

Setting $z=r\myp\rme^{\myp\rmi t}$ with $r=|z|<R$\myp,
\,$t=\arg\myn(z)\in(-\pi,\pi]$, let us consider the Taylor expansion
of $g_\kappa(z)$ near $z_0=r$ with respect to $t$,
\begin{equation}\label{eq:admissible_expansion}
g_\kappa(r\rme^{\myp\rmi t}) = g_\kappa(r) + \rmi\myp t\mypp
b_1(r)-\tfrac{1}{2}\myp t^2 b_2(r) + o\myp(t^2),\qquad  t\to0,
\end{equation}
where $b_1(r)$, $b_2(r)$ are defined in \eqref{eq:b12}. Note that
$g_\kappa^{\{1\}}\myn(0) = 0$ and the function $r\mapsto
g_\kappa^{\{1\}}\myn(r)$ is real analytic and strictly increasing
for $r\geq 0$ (of course, provided that $g_\kappa(z)$ is not
identically zero). Thus, the inverse of $g_\kappa^{\{1\}}\myn(r)$
exists for $0<r\le R$\myp. For $v\ge 1/g_\kappa^{\{1\}}\myn(R\myp)$,
let $r_v$ be the (unique) solution of the equation
\begin{equation}\label{eq:r_v}
g_\kappa^{\{1\}}\myn(r) = v^{-1},\qquad 0<r\le R.
\end{equation}
In particular, for $v=1/g_\kappa^{\{1\}}\myn(R\myp)$ we have
$r_v=R$\myp.

\begin{definition}\label{def:sub-sup}
The cases $g_\kappa^{\{1\}}\myn(R\myp)>1$ and
$g_\kappa^{\{1\}}\myn(R\myp)<1$ are termed ``subcritical'' and
``supercritical'', respectively.
\end{definition}

This terminology will be justified in Section~\ref{sec:cycles}
below, in particular by Theorem \ref{thm:fractional_infinite}, where
we will demonstrate that the limiting fraction of points in infinite
cycles is positive if and only if $g_\kappa^{\{1\}}\myn(R\myp)<1$.
In~Section~\ref{sec:comparison}, it will also be shown that the
dichotomy in Definition \ref{def:sub-sup} corresponds to the cases
$\tilde{\rho}< \tilde{\rho}_{\myp\mathrm{c}}$ and $\tilde{\rho}>
\tilde{\rho}_{\myp\mathrm{c}}$, respectively, with $\tilde{\rho}$
standing for the system ``density'' (see Section~\ref{sec:6.2}
below).

The analytical reason for such a distinction becomes clear from the
following observation. Proceeding from the expression
\eqref{eq:Hn_generatingN1} for $H_N$, with the function $\GN(z)$
defined in~\eqref{eq:GN}, Cauchy's integral formula with the contour
\,$\gamma\mynn:=\{z = r \myp\rme^{\myp\rmi t},\;t\in[-\pi,\pi]\}$
\,($r<R$) gives
\begin{align}
\notag [z^{N}]\myp\exp\myn\{\GN(z)\}
&= \frac{1}{2 \pi\rmi} \oint_{\gamma} \frac{\myp\exp\myn\{\GN(z)\}}{z^{N+1}} \,\rmd{z}\\
\label{eq:int_HN} &= \frac{1}{2 \pi\rmi} \oint_{\gamma}
\frac{\exp\myn\{g_\theta(z)\}}{z}\,\exp\myn\bigl\{N\bigl(
g_\kappa(z)-\log z\bigr)\bigr\} \,\rmd{z}.
\end{align}
Hence, the classical saddle point method (see, e.g.,
\cite[Ch.~5]{Bru})
suggests that the asymptotics of the integral \eqref{eq:int_HN} are
determined by the maximum of the function $z\mapsto
\Re\myp(g_\kappa(z)-\log z)$. In turn, by Lemma~\ref{lm:NA}\myp(a)
this is reduced to finding the maximum of the function $r\mapsto
g_\kappa(r)-\log r$ for real $r>0$, leading to the equation
\begin{equation*}
g'_\kappa(r)-\frac{1}{r}=0\quad\Leftrightarrow\quad
g_\kappa^{\{1\}}\myn(r)=1,
\end{equation*}
which, in view of Lemma~\ref{lm:NA}\myp(b), is solvable if and only
if $g_\kappa^{\{1\}}\myn(R\myp)\ge 1$.

\subsection{The subcritical case}
\label{sec:weak_admissible} In what follows, we shall frequently use
a standard shorthand $a_N\sim b_N$ for $\lim_{N\to\infty}
a_N/b_N=1$, or equivalently $a_N=b_N(1+o(1))$. The same symbol
$\sim$ will have a similar meaning for other limits (e.g., as
$z\uparrow R$\myp).

\begin{theorem}\label{thm:aux_asypmtotic} Assume
that the generating functions $g_\theta(z)$ and $g_\kappa(z)$ both
have radius of convergence $R>0$, and suppose that\/
$1<g_\kappa^{\{1\}}\myn(R\myp)\le\infty$. Let $f(z)$ be a function
holomorphic in the open disk $|z|<R$\myp. Then, uniformly in
$v\in[v_1,v_2]$ with arbitrary constants
$v_2>1>v_1>1/g_\kappa^{\{1\}}\myn(R\myp)$, we have
\begin{equation}\label{eq:th3.2}
[z^{N}]\bigl[f(z)\exp\myn\{v\myp \GN(z)\}\bigr] \sim
\frac{f(r_v)\exp\myn\{v\myp \GN(r_v)\}}{r_v^{N}\sqrt{2\pi N v\myp
b_2(r_v)}}\mypp,\qquad N\to\infty,
\end{equation}
where $b_2(r)$ and $r_v$ are defined in \eqref{eq:b12} and
\eqref{eq:r_v}, respectively. In particular,
\begin{equation}\label{eq:HN-1}
H_N\sim \frac{\exp\myn\{g_\theta(r_1) + N
g_\kappa(r_1)\}}{r_1^{N}\sqrt{2\pi N\myp b_2(r_1)}}\mypp,\qquad
N\to\infty.
\end{equation}
\end{theorem}

\begin{proof}
First of all, according to \eqref{eq:Hn_generatingN1} formula
\eqref{eq:HN-1} readily follows from \eqref{eq:th3.2} by setting
$f(z)\equiv1$ and $v=1$, whereby $r_v|_{v=1}=r_1$. To handle the
general case, apply Cauchy's integral formula with the contour
\,$\gamma\mynn:=\{ z = r_v \myp\rme^{\myp\rmi t},\allowbreak
\;t\in[-\pi,\pi]\}$ to obtain
\begin{align}
\notag [z^{N}]&\bigl[f(z)\exp\myn\{v\myp \GN(z)\}\bigr] = \frac{1}{2
\pi\rmi} \oint_{\gamma} \frac{f(z)\exp\myn\{v\myp
\GN(z)\}}{z^{N+1}} \,\rmd{z}\\
\notag &= \frac{1}{2 \pi} \int_{-\pi}^\pi \frac{f(r_v\myp
\rme^{\myp\rmi t})\exp\myn\bigl\{v\myp g_\theta(r_v\myp
\rme^{\myp\rmi t})+\rmi\myp t\myp
n\bigr\}}{r_v^{N}}\,\exp\myn\bigl\{N\bigl( v\myp g_\kappa(r_v\myp
\rme^{\myp\rmi t}) - \rmi\myp
t\bigr) \bigr\} \,\rmd{t}\\
\label{eq:int123} &=:\frac{1}{2
\pi}\bigl(\mathcal{I}_N^{\myp1}+\mathcal{I}_N^{\mypp2}+\mathcal{I}_N^{\mypp3}\bigr),
\end{align}
where $\mathcal{I}_N^{\myp1}$, $\mathcal{I}_N^{\mypp2}$ and
$\mathcal{I}_N^{\mypp3}$ are the corresponding integrals arising
upon splitting the interval $[-\pi,\pi]$\strut{} into three parts,
with $|t|\in[0,t_N]\cup[t_N,\delta]\cup[\delta,\pi]$. Choosing
$t_N=N^{-\beta}$ with \strut{}$\frac13<\beta<\frac12$ and $\delta>0$
small enough, we estimate each of the integrals in~\eqref{eq:int123}
as follows.

\smallskip
(i) \,By Taylor's expansion \eqref{eq:admissible_expansion} at
$r=r_v$, with $b_1(r_v) = v^{-1}$ (see \eqref{eq:b12},
\eqref{eq:r_v}), we have
\begin{equation}\label{eq:I1}
\mathcal{I}_N^{\myp1}=\frac{f(r_v)\exp\myn\{v\myp
\GN(r_v)\}}{r_v^{N}} \int_{-t_N}^{t_N} \exp\myn\bigl\{-\tfrac12\myp
N\myp t^2\myp v\myp b_2(r_v) + O(N\myp t^3)\bigr\}\, \rmd{t}.
\end{equation}
On the change of variables $s= \sqrt{N}\mypp t$, the integral in
\eqref{eq:I1} becomes
\begin{align*}
\frac{1}{\sqrt{N}} &\int_{-\sqrt{N}\myp t_N}^{\sqrt{N}\myp t_N}
\exp\myn\bigl\{-\tfrac12\myp v\myp b_2(r_v)\myp s^2 +
O(N^{-1/2}s^3)\bigr\}\, \rmd{s}\\
&= \frac{\exp\myn\{O(N t_N^3)\}}{\sqrt{N}} \int_{-\sqrt{N}\myp
t_N}^{\sqrt{N}\myp t_N} \exp\myn\bigl\{-\tfrac12\myp v\myp
b_2(r_v)\myp s^2\bigr\} \,\rmd{s}\sim \sqrt{\frac{2\pi}{N v\myp
b_2(r_v)}}\,,
\end{align*}
as long as $N\myp t_N^3\to0$. Hence, returning to \eqref{eq:I1} we
get
\begin{equation}\label{eq:I1as-sub}
\mathcal{I}_N^{\myp1} \sim \frac{f(r_v)\exp\myn\{v\myp
\GN(r_v)\}}{r_v^{N}}\,\sqrt{\frac{2\pi}{N v\myp b_2(r_v)}}\,,\qquad
N\to\infty.
\end{equation}

\smallskip
(ii) \,Similarly, using the expansion
\eqref{eq:admissible_expansion} for $|t|\le \delta$ we obtain
\begin{align}
\notag |\mathcal{I}_N^{\mypp2}|&= O(1)\,\frac{\exp\myn\{v\myp
\GN(r_v)\}}{\sqrt{N}\,r_v^{N}} \int_{\sqrt{N}t_N}^{\infty}
\exp\myn\bigl\{-\tfrac12\myp
v\myp b_2(r_v)\myp s^2\bigr\} \,\rmd{s}\\
\label{eq:I2as-sub} &= o\myp(1)\,\frac{\exp\myn\{v\myp
\GN(r_v)\}}{\sqrt{N}\,r_v^{N}}\mypp,
\end{align}
as long as $\sqrt{N}\mypp t_N\to\infty$.

\smallskip
(iii) \,For $|t|\ge\delta$, by a simple absolute value estimate we
have
\begin{equation}\label{eq:I3as-sub}
|\mathcal{I}_N^{\mypp3}|= \frac{O(1)}{r_v^{N}} \int_{\delta}^{\pi}
\exp\myn\bigl\{-Nv\myp\Re \myp(g_\kappa(r_v\myp \rme^{\myp\rmi t}))
\bigr\} \,\rmd{s} =
O(1)\,\frac{\exp\myn\{-Nv\myp\mu_\kappa(\delta)\}}{r_v^{N}}\mypp,
\end{equation}
where, according to Lemma~\ref{lm:NA}\myp(a),
\begin{equation*}
\mu_\kappa(\delta):= \max_{|t|\in[\delta,\myp\pi]}
\Re\myp(g_\kappa(r_v\myp \rme^{\myp\rmi t}))< g_\kappa(r_v).
\end{equation*}
Hence, the bound \eqref{eq:I2as-sub} is exponentially small as
compared to \eqref{eq:I1as-sub}, and so the contribution from
$\mathcal{I}_N^{\mypp3}$ is negligible.

\smallskip
Substituting the estimates \eqref{eq:I1as-sub}, \eqref{eq:I2as-sub}
and \eqref{eq:I3as-sub} into \eqref{eq:int123} yields the asymptotic
formula \eqref{eq:th3.2} for a fixed $v>0$. Finally, it is easy to
see that all $O(\cdot)$ and $o\myp(\cdot)$ terms used above are
uniform in $v\in [v_1,v_2]$, as claimed in the theorem. This
completes the proof.
\end{proof}


\subsection{The supercritical case}
\label{sec:weak_admissible_2}

\subsubsection{The domain $\varDelta_0$}\label{sec:Delta0}
Recall (see Definition \ref{def:sub-sup}) that the supercritical
case occurs when $g_\kappa^{\{1\}}\myn(R\myp)<1$, whereby the
equation $g_\kappa^{\{1\}}\myn(r) =1$ is no longer solvable. To
overcome this difficulty, we have to allow the contour of
integration in Cauchy's integral formula akin to \eqref{eq:int_HN}
to go outside the disk of convergence $|z|< R$\myp. To make this
idea more precise, we give the following definition (see
Fig.~\ref{fig_delta_0}).
\begin{definition} \label{def:Delta0} For
$\eta>0$ and $\varphi\in(0,\frac{\pi}{2})$, define an open domain
$\varDelta_0= \varDelta_0(R,\eta,\varphi)$ in the complex plane by
\begin{equation*}
\varDelta_0: = \bigl\{z\in \CC :
\,|z|<|R\mypp(1+\eta\mypp\rme^{\myp\rmi\varphi})|, \;z\ne R,
\;|\myn\arg\myn(z-R\myp)|>\varphi\bigr\}.
\end{equation*}
We also introduce the notation for the radius of the outer circle,
$R^{\myp\prime}:=R\,|1+\eta\mypp\rme^{\myp\rmi\varphi}|$, and the
angle $\alpha:=\arg\myn(1+\eta\mypp\rme^{\myp\rmi\varphi})\in
(0,\myp\pi/2)$, so that $R\mypp(1+\eta\mypp \rme^{\myp\rmi\varphi})
= R^{\myp\prime}\rme^{\myp\rmi\alpha}$.
\end{definition}

\begin{figure}[ht!]
\centering
\includegraphics[width=.35\textwidth]{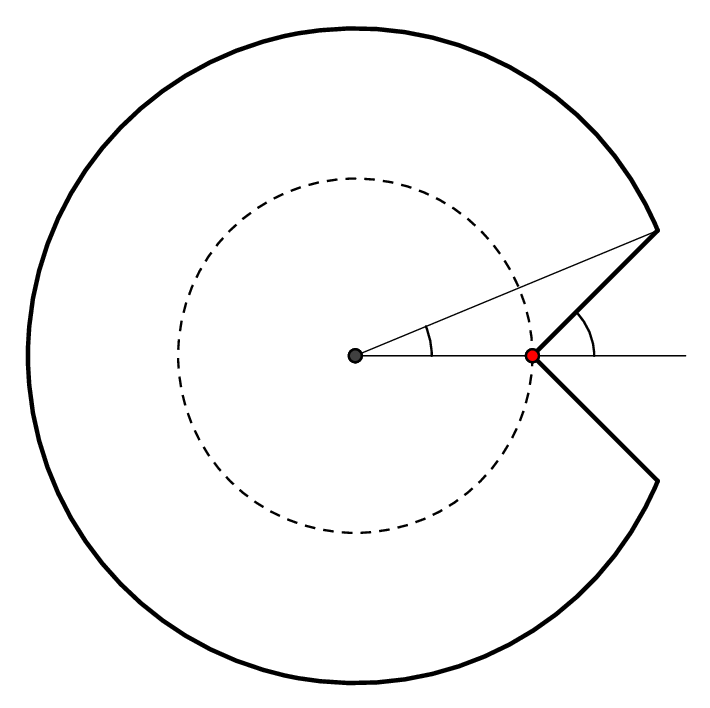}
  \put(-81,67){\mbox{\scriptsize$0$}}
  \put(-93,27){\mbox{\scriptsize$|z|\mynn=\mynn R$}}
  \put(-156,131){\mbox{\scriptsize$|z|\mynn=\mynn R^{\myp\prime}$}}
  \put(-34.9,53.0){\mbox{\scriptsize$R\myp\eta$}}
 \put(-23,82){\mbox{\scriptsize$\varphi$}}
 \put(-57.8,79){\mbox{\scriptsize$\alpha$}}
 \caption{Domain $\varDelta_0=\varDelta_0(R,\eta,\varphi)$.}
 \label{fig_delta_0}
\end{figure}

\begin{lemma}\label{lm:NA1}
Suppose that Assumption \textup{\ref{as:NA}} holds, and let
$g_\kappa(z)$ be holomorphic in a domain
$\varDelta_0=\varDelta_0(R,\eta,\varphi)$ as defined above, with
$g_\kappa^{\{1\}}\myn(R\myp)<\infty$. Then there exist constants
$\delta>0$, $\varepsilon>0$ such that, for all $z\in\varDelta_0$
with $R\le|z|\le R\myp(1+\delta)$, \,$|\myn\arg\myn(z)|\ge\delta$,
\begin{equation}\label{eq:log|z|}
\Re\myp(g_\kappa(z)) \le
g_\kappa(R\myp)+\bigl(g_\kappa^{\{1\}}\myn(R\myp)-\varepsilon\bigr)\myp\log
\frac{|z|}{R}\mypp.
\end{equation}
\end{lemma}
\begin{proof} By Lemma \ref{lm:NA}\myp(b), if $|z|=R$ and $z\ne0$\/ then
\begin{equation}\label{eq:{1}}
\Re\myp(g_\kappa^{\{1\}}\myn(z))<g_\kappa^{\{1\}}\myn(R\myp).
\end{equation}
Along with $g_\kappa(z)$, the function $z\mapsto
\Re\myp(g_\kappa^{\{1\}}\myn(z))$ is also analytic in $\varDelta_0$,
and in particular it is continuous in a vicinity of the punctured
circle\strut{} $\{|z|=R,\ z\ne R\}\subset \varDelta_0$. Hence, by a
compactness argument the inequality \eqref{eq:{1}} also holds on a
closed domain
$$
\varDelta_\delta:=\varDelta_0\cap\{z:R\le |z|\le R\myp(1+\delta),\
|\myn\arg\myn(z)|\ge\delta\},
$$
with $\delta>0$ small enough. Moreover, by the continuity of
$\Re\myp(g_\kappa^{\{1\}}\myn(z))$ the strict inequality
\eqref{eq:{1}} implies that
$$
\varepsilon:=g_\kappa^{\{1\}}
\myn(R\myp)-\sup_{z\in\varDelta_\delta}\Re\myp(g_\kappa^{\{1\}}\myn(z))>0.
$$
Therefore, for any $z=r\myp\rme^{\myp\rmi t}\in\varDelta_\delta$ we
have
\begin{equation}\label{eq:log-r/R-a}
\int_R^{r} \Re\myp\bigl(g_\kappa^{\{1\}}\myn(u\myp\rme^{\myp\rmi
t})\bigr)\, \frac{\rmd{u}}{u}\le \int_R^{r}
\bigl(g_\kappa^{\{1\}}\myn(R\myp)-\varepsilon\bigr)\,
\frac{\rmd{u}}{u}
=\bigl(g_\kappa^{\{1\}}\myn(R\myp)-\varepsilon\bigr) \myp\log
\frac{|z|}{R}\mypp.
\end{equation}
On the other hand,
\begin{align}\label{eq:log-r/R-b}
\int_R^{r} \Re\myp\bigl(g_\kappa^{\{1\}}\myn(u\myp\rme^{\myp\rmi
t})\bigr)\, \frac{\rmd{u}}{u}&=\Re\left(\int_R^{r}
\frac{\partial}{\partial u}\,g_\kappa(u\myp\rme^{\myp\rmi
t})\,\rmd{u}\right)=
\Re\myp\bigl(g_\kappa(z)-g_\kappa(R\myp\rme^{\myp\rmi t})\bigr).
\end{align}
Hence, combining \eqref{eq:log-r/R-a} and \eqref{eq:log-r/R-b}  we
obtain
$$
\Re\myp(g_\kappa(z)) \le \Re\myp (g_\kappa(R\,\rme^{\myp\rmi
t}))+\bigl(g_\kappa^{\{1\}}\myn(R\myp)-\varepsilon\bigr)\myp \log
\frac{|z|}{R}\mypp,
$$
and the inequality \eqref{eq:log|z|} readily follows in view of
Lemma~\ref{lm:NA}\myp(a).
\end{proof}

Motivated by the model choice $\theta_j=\const\ge0$ (see Section
\ref{sec:polylog} below), in the supercritical asymptotic theorems
that follow we allow the generating function $g_\theta(z)$ to have a
logarithmic singularity at $z=R$\myp, of the form
$g_\theta(z)\sim-\theta^*\log\myn(1-z/R\myp)$ as $z\to R$
\,($z\in\varDelta_0$). It turns out that there is a significant
distinction between the cases $\theta^*\myn>0$ and $\theta^*\myn=0$.

\subsubsection{Case $\theta^*\myn>0$}
We first handle the case with a non-degenerate log-singularity of
$g_\theta(z)$.
\begin{theorem}\label{thm:aux_asypmtotic_2}
Let the generating functions $g_\theta(z)$ and $g_\kappa(z)$ both
have radius of convergence $R>0$ and be holomorphic in some domain
$\varDelta_0$ as in Definition~\textup{\ref{def:Delta0}}. Assume
that $g_\kappa^{\{1\}}\myn(R\myp)< 1$ and the following asymptotic
formulas hold as $z\to R$ \,\textup{(}$z\in \varDelta_0$\textup{)},
with some $\theta^*\ge0$, $\delta>0$,
\begin{align}\label{eq:g_theta_as}
    g_\theta(z) &= -\theta^* \log\myn(1-z/R\myp)
+O\bigl((1-z/R\myp)^{\delta}\bigr),\\
    \label{eq:g_kappa_as}
    g_\kappa(z)  &= g_\kappa(R\myp) - g_\kappa^{\{1\}}\myn(R\myp)
    (1-z/R\myp) +
    O\bigl((1-z/R\myp)^{1+\delta}\bigr).
\end{align}
Finally, let $f\colon\varDelta_0 \to \CC$ be a holomorphic function
such that for some $\beta\geq0$
\begin{equation}\label{eq:f_as}
f(z) =
(1-z/R\myp)^{-\beta}\mypp\bigl\{1+O\bigl((1-z/R\myp)^{\delta}\bigr)\bigr\},
\qquad z\to R\ \ \ (z\in\varDelta_0).
\end{equation}
Then, provided that\/ $v\myp\theta^*\myn+\beta>0$, we have, as
$N\to\infty$,
\begin{equation}\label{eq:th3.3}
 [z^{N}]\bigl[f(z)\exp\myn\{v\myp \GN(z)\}\bigr]
\sim\frac{\exp\myn\{N v\myp
g_\kappa(R\myp)\}\cdot\bigl\{N\myp(1-v\myp
g_\kappa^{\{1\}}\myn(R\myp))\bigr\}^{v\theta^*\myn+\beta -1}}{R^{N}
\,\Gamma(v\myp\theta^*\myn+\beta)}\mypp,
\end{equation}
uniformly in $v\in[v_1,v_2]$ for any
$0<v_1<1<v_2<1/g_\kappa^{\{1\}}\myn(R\myp)$.  In particular, for
$\theta^*\myn>0$
\begin{equation}\label{eq:HN-2}
H_N\sim \frac{\exp\myn\{N g_\kappa(R\myp) \}\cdot\bigl\{N\myp(1-
g_\kappa^{\{1\}}\myn(R\myp))\bigr\}^{\theta^*\myn-1}}{R^{N}
\,\Gamma(\theta^*\myn)}\mypp,\qquad N\to\infty.
\end{equation}
\end{theorem}

\begin{proof}
In view of the identity \eqref{eq:Hn_generatingN1}, formula
\eqref{eq:HN-2} is obtained from \eqref{eq:th3.3} by setting
$f(z)\equiv1$ (so that $\beta=0$) and $v=1$. Let us also observe
that, according to \eqref{eq:g_theta_as} and \eqref{eq:f_as},
$$
v\myp g_\theta(z)+\log
f(z)=-(v\myp\theta^*\myn+\beta)\log\myn(1-z/R\myp)
+O\bigl((1-z/R\myp)^{\delta}\bigr).
$$
Thus, accounting for the pre-exponential factor $f(z)$ just leads to
the change\, $v\myp\theta^*\mapsto v\myp\theta^*\myn+\beta$. With
this in mind, it suffices to consider the basic case $\beta=0$ (but
now with $\theta^*\myn>0$).

Without loss of generality, we may and will assume (by slightly
reducing the original domain $\varDelta_0$ if necessary) that both
$g_\theta(z)$ and $g_\kappa(z)$ are continuous on the boundary of
$\varDelta_0$ except at $z=R$\myp. By virtue of Lemma \ref{lm:NA1}
(and again reducing $\varDelta_0$ as appropriate), we may also
assume that the inequality \eqref{eq:log|z|} is fulfilled for all
$z\in\varDelta_0$ such that $R\le |z|\le R^{\myp\prime}$ and
$|\myn\arg\myn(z)|\ge\delta_0$, with some $\delta_0>0$, where
$R^{\myp\prime}=R\,|1+\eta\mypp\rme^{\myp\rmi\varphi}|$ (see
Definition~\ref{def:Delta0}).

\begin{figure}[h]
\centering
\subfigure[\,$\gamma=\gamma_1\cup\gamma_2\cup\gamma_3\cup\gamma_4$]{\hspace{-1pc}
\hspace{2pc}\includegraphics[height=.2\textheight]{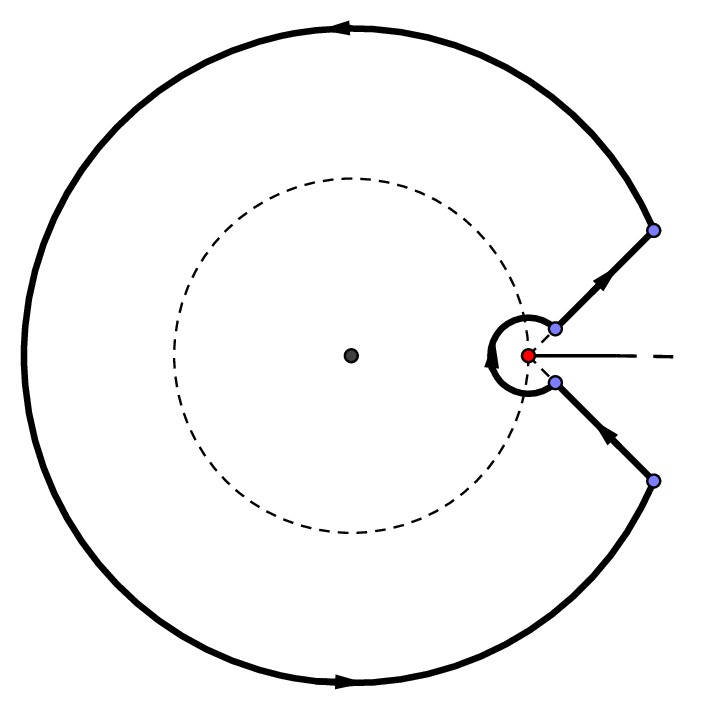}
  \label{fig:gamma}
  \put(-75,57){\mbox{\scriptsize$0$}}
  \put(-83,102){\mbox{\scriptsize$|z|\mynn=\mynn R$}}
  \put(-126,123){\mbox{\scriptsize$|z|\mynn=\mynn R^{\myp\prime}$}}
   \put(-29.5,44.5){\mbox{\scriptsize$\gamma_1$}}
   \put(-54.5,64){\mbox{\scriptsize$\gamma_2$}}
   \put(-31.5,83.5){\mbox{\scriptsize$\gamma_3$}}
   \put(-142,65){\mbox{\scriptsize$\gamma_4$}}
 }
\hspace{2.5pc}\subfigure[\,$\gamma'=\gamma'_1\cup\gamma'_2\cup\gamma'_3$]{
   \includegraphics[height=.2\textheight]{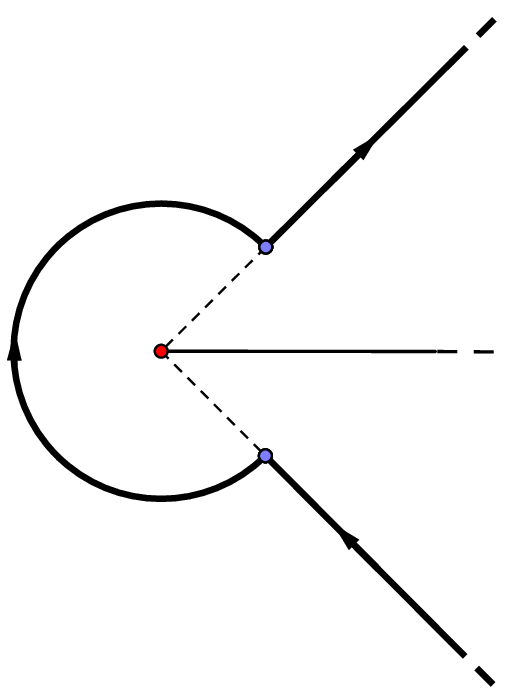}
   \label{fig:gamma'}
  \put(-75,57){\mbox{\scriptsize$0$}}
    \put(-46,24){\mbox{\scriptsize$\gamma'_1$}}
   \put(-110,64){\mbox{\scriptsize$\gamma'_2$}}
   \put(-43,106.5){\mbox{\scriptsize$\gamma'_3$}}
 }
\hspace{2.5pc}\subfigure[\,$\gamma''=\gamma''_1\cup\gamma''_2\cup\gamma''_3$]{
   \includegraphics[height=.2\textheight]{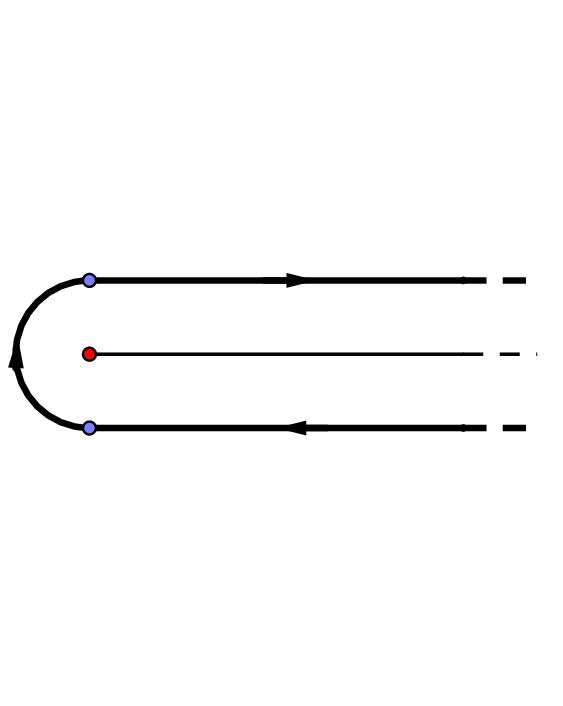}
   \label{fig:gamma''}
   \put(-98.5,59){\mbox{\scriptsize$0$}}
   \put(-56.5,41.5){\mbox{\scriptsize$\gamma''_1$}}
   \put(-118,64){\mbox{\scriptsize$\gamma''_2$}}
   \put(-59.5,86.0){\mbox{\scriptsize$\gamma''_3$}}
 }
\caption{Contours used in the proof of
Theorem~\ref{thm:aux_asypmtotic_2}.}
\end{figure}Consider a continuous closed contour
$\gamma=\gamma_1\cup\gamma_2\cup\gamma_3\cup\gamma_4$ (see
Fig.~\ref{fig:gamma}), where
\begin{equation}\label{eq:gamma1-4}
\begin{aligned}
\gamma_1:={}&\bigl\{z= R\mypp(1+(\eta-x)\,\rme^{-\rmi\varphi}),\
\,x\in[0,\myp\eta-N^{-1}]\bigr\},\\[.2pc]
\gamma_2:={}&\bigl\{ z=R\mypp(1 +N^{-1} \rme^{-\rmi t}),\ \,
t\in[\varphi,\myp 2\pi -\varphi\myp]\bigr\},\\[.2pc]
 \gamma_3:={}&\bigl\{
z=R\mypp(1 + x\myp\rme^{\myp\rmi\varphi}),\ \, x\in[N^{-1}\mynn,\myp\eta\myp]\bigr\},\\[.2pc]
\gamma_4:={}&\bigl\{z=R^{\myp\prime}\rme^{\myp\rmi t},\ \,
t\in[\alpha, \myp 2\pi -\alpha\myp]\bigr\}.
\end{aligned}
\end{equation}
According to the chosen parameterization in \eqref{eq:gamma1-4}, the
contour $\gamma$ is traversed anti-clockwise. Then Cauchy's integral
formula yields
\begin{align}
\notag [z^{N}]\exp\myn\{v\myp \GN(z)\} &=\frac{1}{2
\pi\rmi}\left(\int_{\gamma_1}\!+\cdots+\int_{\gamma_4}\right)
\frac{\exp\myn\{v\myp \GN(z)\}}{z^{N+1}}
\,\rmd{z}\\[.3pc]
&=:\frac{1}{2
\pi\rmi}\left(\mathcal{I}_N^{\myp1}+\cdots+\mathcal{I}_N^{\mypp4}\right).
\label{eq:I1...I4}
\end{align}
Now, we estimate each of the integrals in~\eqref{eq:I1...I4}. Denote
for short
\begin{equation}\label{eq:chi}
d_1:=1-v\myp g_\kappa^{\{1\}}\myn(R\myp)>0.
\end{equation}

(i) Let us first show that the integral $\mathcal{I}_N^{\mypp4}$
over the circular arc $\gamma_4$ (see~\eqref{eq:gamma1-4}) is
negligible as $N\to\infty$. Indeed, by an absolute value inequality
we have
\begin{align}
\notag |\mathcal{I}_N^{\mypp4}|&\le \int_{\gamma_4}
\frac{\exp\myn\{v\mypp\Re\myp(\GN(z))\}}
{|z|^{N+1}}\,\rmd|z|\\
\label{eq:I2<} &\le
\frac{2\pi}{(R^{\myp\prime})^{N}}\,\exp\myn\bigl\{v\myp
\mu_\theta(R^{\myp\prime})+Nv\myp \mu_\kappa(R^{\myp\prime})\bigr\},
\end{align}
where $\mu_\theta(R^{\myp\prime}):=\max_{z\in\gamma_4}
\Re\myp(g_\theta(z))<\infty$ and, by Lemma \ref{lm:NA1},
\begin{equation}\label{eq:mu-kappa}
\mu_\kappa(R^{\myp\prime}):= \max_{z\in\gamma_4}
\Re\myp(g_\kappa(z))\le
g_\kappa(R\myp)+g_\kappa^{\{1\}}\myn(R\myp)\cdot\log\frac{R^{\myp\prime}}{R}\mypp.
\end{equation}
Hence, the right-hand side of the bound \eqref{eq:I2<} is further
estimated by
\begin{align}
\notag \frac{O(1)}{(R^{\myp\prime})^{N}}\,
\exp\myn\biggl\{Nv\biggl(g_\kappa(R\myp)&+
g_\kappa^{\{1\}}\myn(R\myp)\cdot\log\frac{R^{\myp\prime}}{R}\biggr)\biggr\}\\[.3pc]
&=\frac{O(1)\,\exp\myn\{Nv\myp
g_\kappa(R\myp)\}}{R^{N}}\,\exp\myn\biggl\{-Nd_1\log\frac{R^{\myp\prime}}{R}\biggr\}
\label{eq:d1_exp}
\end{align}
(see the notation~\eqref{eq:chi}), which is exponentially small as
compared to the right-hand side of \eqref{eq:th3.3}, thanks to the
inequalities $R^{\myp\prime}/R>1$ and $d_1>0$. Thus,
\begin{equation}\label{eq:I4}
\mathcal{I}_N^{\mypp4}=o\myp(1)\,\frac{\exp\myn\{Nv\myp
g_\kappa(R\myp)\}}{R^{N}N^{1-v\theta^*}}\mypp,\qquad N\to\infty.
\end{equation}

\smallskip
(ii) For $z\in\gamma_2$ (see~\eqref{eq:gamma1-4}), we have
\begin{equation}\label{eq:w2}
z=R\mypp(1+w N^{-1}),\qquad
w\in\gamma_2^{\myp\prime}:=\bigl\{\rme^{-\rmi
t},\:t\in[\varphi,\,2\pi -\varphi\myp]\bigr\}.
\end{equation}
Then formulas \eqref{eq:g_theta_as} and \eqref{eq:g_kappa_as} yield
the uniform asymptotics, as $N\to\infty$,
\begin{align}
\label{eq:g-theta<} g_\theta(z)&=-\theta^*\log\myn(-w
N^{-1})
+O\bigl((w/N)^{\delta}\bigr),\\[.1pc]
\label{eq:g-kappa<} g_\kappa(z)&=g_\kappa(R\myp) +
g_\kappa^{\{1\}}\myn(R\myp)\mypp w N^{-1} +
O\bigl((w/N)^{1+\delta}\bigr).
\end{align}
We also have
\begin{align}
\notag \frac{1}{z^{N+1}}&=\frac{1}{R^{N+1}}\,\exp\myn\bigl\{-(N+1)
\log\myn(1+w N^{-1})\bigr\}\\
&=\frac{1}{R^{N+1}}\,\exp\myn\bigl\{-w+O(w N^{-1})\bigr\},\qquad
N\to\infty. \label{eq:1/z^}
\end{align}
Collecting \eqref{eq:g-theta<}, \eqref{eq:g-kappa<} and
\eqref{eq:1/z^}, we obtain from \eqref{eq:I1...I4} via the change of
variables~\eqref{eq:w2}
\begin{equation}\label{eq:I2}
\mathcal{I}_N^{\mypp 2} \sim \frac{\exp\myn\{N v\myp g_\kappa(R\myp)
\}}{R^{N}N^{1-v\theta^*} } \int_{\gamma_2'}
(-w)^{-v\theta^*}\rme^{-d_1 w}\,\rmd{w}.
\end{equation}

(iii) By the symmetry between the contours $\gamma_1$ and
$\gamma_3$, it is easy to see that the corresponding integrals
$-\mathcal{I}_N^{\myp 1}$ and $\mathcal{I}_N^{\mypp 3}$ are complex
conjugate to one another. Hence, it suffices to consider, say,
$\mathcal{I}_N^{\mypp 3}$. Similarly to \eqref{eq:w2}, we
reparameterize the contour $\gamma_3$ (see~\eqref{eq:gamma1-4}) as
\begin{equation}\label{eq:w3}
z=R\mypp(1+wN^{-1}),\qquad
w\in\gamma_{3,N}^{\myp\prime}:=\bigl\{x\myp\rme^{\myp\rmi
\varphi},\:x\in[1,\myp N\eta\myp]\bigr\}.
\end{equation}
Let us split the contour $\gamma^{\myp\prime}_{3,N}$ into three
parts corresponding to
$x\in[1,x_N]\cup[x_N,N\eta_0]\cup[N\eta_0,N\eta\myp]$, with
$x_N\to\infty$, $x_N=o\myp(N)$ as $N\to\infty$, and denote the
respective parts of the integral $\mathcal{I}_N^{\mypp3}$ by
$\mathcal{I}_N^{\mypp3^{\myp\prime}}$,
$\mathcal{I}_N^{\mypp3^{\myp\prime\prime}}$ and
$\mathcal{I}_N^{\mypp3^{\myp\prime\prime\prime\!}}$. Because the
substitutions \eqref{eq:w2} and \eqref{eq:w3} are formally identical
to one another, it is clear that the estimates \eqref{eq:g-theta<},
\eqref{eq:g-kappa<} and \eqref{eq:1/z^} hold true for \eqref{eq:w3}
and, moreover, are uniform in $w$ such that $|w|\le N\eta_0$, with
$\eta_0>0$ small enough. Hence, the asymptotics of the integral
$\mathcal{I}_N^{\mypp3^{\myp\prime}}$ is given by a formula
analogous to~\eqref{eq:I2},
\begin{equation}\label{eq:I3'}
\mathcal{I}_N^{\mypp3^{\myp\prime}} \sim \frac{\exp\myn\{N v\myp
g_\kappa(R\myp)\}}{R^{N}N^{1-v\theta^*}} \int_{\gamma'_3}
(-w)^{-v\theta^*} \rme^{-d_1 w}\,\rmd{w},
\end{equation}
where (cf.~\eqref{eq:w3})
\begin{equation}\label{eq:gamma'3}
\gamma^{\myp\prime}_3:=\lim_{N\to\infty}\gamma'_{3,N}
=\bigl\{w=x\myp\rme^{\myp\rmi \varphi},\:x\in[1,\infty)\bigr\}.
\end{equation}
Similarly, the integral $\mathcal{I}_N^{\mypp3^{\myp\prime\prime}}$
is asymptotically estimated as
\begin{align}
\notag
\mathcal{I}_N^{\mypp3^{\myp\prime\prime}}&=\frac{O(1)\exp\myn\{N
v\myp g_\kappa(R\myp)\}}{R^{N} N^{1-v\theta^*} }
\int_{x_N}^{N\eta_0} \frac{\exp\myn\{-d_1 x\cos\varphi
\}}{x^{v\theta^*}}\,\rmd{x}\\[.2pc]
\notag &\quad=\frac{O(1)\exp\myn\{N v\myp g_\kappa(R\myp)\}}{R^{N}
N^{1-v\theta^*}} \cdot \frac{\exp\myn\{-d_1
x_N\myn\cos\varphi\}}{\displaystyle
(x_N)^{v\theta^*}}\\[.1pc]
\label{eq:I3''}
&\qquad=o\bigl(\mathcal{I}_N^{\mypp3^{\myp\prime}}\bigr),\qquad
N\to\infty,
\end{align}
provided that $x_N\to\infty$ and $d_1>0$ (see~\eqref{eq:chi}).

By the continuity of $g_\theta(\cdot)$, for $z$ in \eqref{eq:w3}
with $|w|\in [N\eta_0,N\eta\myp]$ there is a uniform bound
\begin{equation}\label{eq:g-theta<-1}
\Re\myp(g_\theta(z))=O(1).
\end{equation}
On the other hand, the estimate \eqref{eq:log|z|} of Lemma
\ref{lm:NA1} takes the form
\begin{equation}\label{eq:Re_g_kappa<}
\Re\myp(g_\kappa(z))\le
g_\kappa(R\myp)+g_\kappa^{\{1\}}\myn(R\myp)\cdot\log
|1+wN^{-1}|\myp,
\end{equation}
whereby
\begin{equation}\label{eq:log_w<}
\log\myn\Bigl|1+\frac{w}{N}\Bigr|=\frac12\log\myn
\biggl(1+\frac{2|w|\cos\varphi}{N}+\frac{|w|^2}{N^2}\biggr)
=\frac{|w|\cos\varphi}{N} + O(|w|^2 N^{-2}).
\end{equation}
Hence, using \eqref{eq:1/z^} and
\eqref{eq:g-theta<-1}\mypp--\mypp\eqref{eq:log_w<} we can adapt the
estimation in \eqref{eq:I3''} to obtain
\begin{align}
\notag \mathcal{I}_N^{\mypp
3^{\myp\prime\prime\prime}}&=\frac{O(1)\exp\myn\{N v\myp
g_\kappa(R\myp)\}}{R^{N} N} \int_{N\eta_0}^{N\eta} \exp\myn\{-d_1
x\cos\varphi\}\,\rmd{x}\\[.2pc]
\notag &\quad=\frac{O(1)\exp\myn\{N v\myp g_\kappa(R\myp)\}}{R^{N}N}
\, \exp\myn\{-d_1N\eta_0\cos\varphi
\}\\[.1pc]
\label{eq:I3'''}
&\qquad=o\bigl(\mathcal{I}_N^{\mypp3^{\myp\prime}}\bigr),\qquad
N\to\infty,
\end{align}
because $\cos\varphi>0$ and $d_1>0$ (see~\eqref{eq:chi}).

As a result, combining the asymptotic formulas \eqref{eq:I3'},
\eqref{eq:I3''} and \eqref{eq:I3'''} we get
\begin{equation}\label{eq:I3}
\mathcal{I}_N^{\mypp3} \sim \frac{\exp\myn\{N v\myp g_\kappa(R\myp)
\}}{R^{N}N^{1-v\theta^*}} \int_{\gamma'_3} (-w)^{-v\theta^*}
\rme^{-d_1w}\,\rmd{w}.
\end{equation}

Finally, collecting the contributions from
$\mathcal{I}_N^{\myp1},\dots,\mathcal{I}_N^{\mypp4}$ (see
\eqref{eq:I4}, \eqref{eq:I2} and~\eqref{eq:I3}) and returning to
\eqref{eq:I1...I4} yields, upon the change of the integration
variable $d_1w\mapsto w$,
\begin{equation}\label{eq:Hankel}
\oint_{\gamma}\,\frac{\exp\myn\{v\myp
\GN(z)\}}{z^{N+1}}\,\rmd{z}\sim \frac{\exp\myn\{N v\myp
g_\kappa(R\myp)\}}{R^{N}N^{1-v\theta^*} } \, (d_1)^{\myp
v\theta^*\myn-1} \int_{\widetilde\gamma^{\myp\prime}}
(-w)^{-v\theta^*} \rme^{-w}\,\rmd{w},
\end{equation}
with $\widetilde\gamma^{\myp\prime}\mynn:= d_1\gamma^{\myp\prime}$,
where
$\gamma^{\myp\prime}\myn=\gamma^{\myp\prime}_1\cup\gamma^{\myp\prime}_2\cup\gamma^{\myp\prime}_3$
is defined via \eqref{eq:w2} and \eqref{eq:gamma'3} (see
Fig.~\ref{fig:gamma'}).

The integral on the right-hand side of \eqref{eq:Hankel} can be
explicitly computed. Indeed, by virtue of a simple estimate
$$
\bigl|(-w)^{-v\theta^*} \rme^{-w}\bigr|\le |w|^{-v\theta^*}
\mynn\exp\myn\{-|w|\cos \arg\myn(w)\},
$$
one can apply a standard contour transformation argument to replace
the contour $\widetilde\gamma^{\myp\prime}$ in \eqref{eq:Hankel} by
the ``loop'' contour $\gamma^{\myp\prime\prime}$ starting from
$+\infty-\rmi c$, winding clockwise about the origin and proceeding
towards $+\infty+\rmi c$ (see Fig.~\ref{fig:gamma''}), which leads
to the equality
\begin{equation}\label{eq:Hankel-based}
\int_{\gamma'} (-w)^{-v\theta^*} \rme^{-w}\,\rmd{w} =
\int_{\gamma''} (-w)^{-v\theta^*} \rme^{-w}\,\rmd{w} =\frac{2\pi
\rmi}{\Gamma(v\myp\theta^*)}\mypp,
\end{equation}
according to the well-known Hankel's loop representation of the
reciprocal gamma function (see, e.g., \cite[\S\myp B.3, Theorem~B.1,
p.~745]{FlSe09}). This completes the proof of
Theorem~\ref{thm:aux_asypmtotic_2}.
\end{proof}

\subsubsection{Case $\theta^*\myn=0$}
Note that the deceptively simple case $\theta_j\equiv0$ (leading to
$\theta^*\myn=0$ in \eqref{eq:g_theta_as}) is not covered by
Theorem~\ref{thm:aux_asypmtotic_2}, unless $\beta>0$. The reason is
that, in the lack of a logarithmic singularity of $g_\theta(z)$ at
$z=R$\myp, the main term in the asymptotic formula \eqref{eq:th3.3}
vanishes (as suggested by the formal equality $1/\Gamma(0)=0$).
Thus, in the case $\theta^*\myn=0$ the singularity of $g_\kappa(z)$
at $z=R$ should become more prominent in the asymptotics. The full
analysis of the competing contributions from the singularities of
$g_\theta(z)$ and $g_\kappa(z)$ can be complicated (however, see
Remark~\ref{rm:sing_g_kappa_and_theta} below), so for simplicity let
us focus on the role of $g_\kappa(z)$ assuming that $g_\theta(z)$ is
regular (e.g., $g_\theta(z)\equiv 0$).
\begin{theorem}\label{thm:aux_asypmtotic_2a}
Let the generating function $g_\kappa(z)$ have radius of convergence
$R>0$ and be holomorphic in a domain $\varDelta_0$ as in
Definition~\textup{\ref{def:Delta0}}. Assume that
$g_\kappa^{\{1\}}\myn(R\myp)<1$ and, furthermore, there is a
non-integer $s>1$ such that $g_\kappa^{\{n\}}\myn(R\myp)<\infty$ for
all $n<s$ while $g_\kappa^{\{n\}}\myn(R\myp)=\infty$ for\/ $n>s$,
and the following asymptotic expansion holds as $z\to R$
\,\textup{(}$z\in \varDelta_0$\textup{)},
\begin{equation}\label{eq:g_kappa_as=0}
g_\kappa(z) = \sum_{0\le n<s} \frac{(-1)^n\myp
g_\kappa^{\{n\}}\myn(R\myp)}{n!}\,(1-z/R\myp)^n
 + a_s\mypp(1-z/R\myp)^s
+O\bigl((1-z/R\myp)^{s+\delta}\bigr),
\end{equation}
with some $a_s>0$, \,$\delta>0$. As for the generating function
$g_\theta(z)$, it is assumed to be holomorphic in the domain
$\varDelta_0$ and, moreover, regular at point $z=R$\myp. Finally,
let $f\colon\varDelta_0 \to \CC$ be a holomorphic function such
that, with some $\beta>0$ and $c_\beta>0$, as $z\to R$
\,\textup{(}$z\in \varDelta_0$\textup{)},
\begin{equation}\label{eq:ftheta=0}
f(z) = 1+ \sum_{1\le n<\beta} \frac{(-1)^n
f^{\{n\}}\myn(R\myp)}{n!}\,(1-z/R\myp)^n +
c_\beta\mypp(1-z/R\myp)^\beta
+O\bigl((1-z/R\myp)^{\beta+\delta}\bigr).
\end{equation}
Then, depending on the relationship between $s$ and $\beta$, the
following asymptotics hold as $N\to\infty$, uniformly in
$v\in[v_1,v_2]$ for any $0<v_1<1<v_2<1/g_\kappa^{\{1\}}\myn(R\myp)$.

\smallskip
\textup{(i)} \myp{}If\/ $\beta>s-1$ then
\begin{equation}\label{eq:th3.3=0(i)}
[z^{N}] \bigl[f(z)\,\rme^{\myp v\myp \GN(z)}\bigr]
\sim\frac{\rme^{\myp v\myp \GN(R\myp)}\, v\myp a_s\mypp
\bigl\{N\myp(1-v\myp g_\kappa^{\{1\}}\myn(R\myp))\bigr\}^{-s}}{R^{N}
\,\Gamma(-s)\mypp \bigl(1-v\myp
g_\kappa^{\{1\}}\myn(R\myp)\bigr)}\mypp.
\end{equation}

\textup{(ii)} \myp{}If\/ $\beta$ is non-integer and $\beta<s-1$ then
\begin{equation}\label{eq:th3.3=0(ii)}
[z^{N}] \bigl[f(z)\,\rme^{\myp v\myp \GN(z)}\bigr]
\sim\frac{\rme^{\myp v\myp \GN(R\myp)}\, c_\beta\mypp
\bigl\{N\myp(1-v\myp
g_\kappa^{\{1\}}\myn(R\myp))\bigr\}^{-\beta-1}}{R^{N}
\,\Gamma(-\beta)}\mypp.
\end{equation}

\textup{(iii)} \myp{}If\/ $\beta=s-1$ then
\begin{align}
\notag [z^{N}] \bigl[f(z)\,\rme^{\myp v\myp \GN(z)}\bigr]
&\sim\frac{\rme^{\myp v\myp \GN(R\myp)}\,\bigl\{N\myp(1-v\myp
g_\kappa^{\{1\}}\myn(R\myp))\bigr\}^{-s}}{R^{N}}\\
\label{eq:th3.3=0(iii)} &\qquad\times
\left(\frac{c_\beta}{\Gamma(1-s)}+\frac{v a_s}{\Gamma(-s)\mypp
\bigl(1-v\myp g_\kappa^{\{1\}}\myn(R\myp)\bigr)}\right).
\end{align}

\smallskip
In particular,
\begin{equation}\label{eq:HN-2=0}
H_N\sim \frac{\rme^{\myp \GN(R\myp)}\, a_s\mypp
\bigl\{N\myp(1-g_\kappa^{\{1\}}\myn(R\myp))\bigr\}^{-s}}{R^{N}
\,\Gamma(-s)\mypp\bigl(1-g_\kappa^{\{1\}}\myn(R\myp)\bigr)}\mypp,\qquad
N\to\infty.
\end{equation}
\end{theorem}

\begin{proof}
The proof proceeds along the lines of the proof of Theorem
\ref{thm:aux_asypmtotic_2}, with suitable modifications indicated
below. In what follows, we may and will assume that $0<\delta<1$.

Under the change of variables $z=R\myp(1+w N^{-1})$, with
$w\in\gamma_2^{\myp\prime}$ (see \eqref{eq:w2}), by virtue of the
expansions \eqref{eq:g_kappa_as=0}, \eqref{eq:ftheta=0} and the
regularity of $g_\theta(z)$ we obtain the uniform asymptotics (cf.\
\eqref{eq:g-theta<}, \eqref{eq:g-kappa<})
\begin{align}\label{eq:g-theta<=0}
g_\theta(z)&= g_\theta(R\myp)+ \sum_{n=1}^{q}
\frac{g_\theta^{\{n\}}\myn(R\myp)}{n!\mypp N^n }\,w^n
+O(N^{-q-1}),\\
\label{eq:g-kappa<theta=0} g_\kappa(z)&= g_\kappa(R\myp) +
\sum_{n=1}^q \frac{g_\kappa^{\{n\}}\myn(R\myp)}{n!\mypp N^n }\,w^n +
    \frac{a_s\myp (-w)^s}{N^{s} }
    +O(N^{-s-\delta}),\\
\notag
f(z) &= 1+ \sum_{1\le n<\beta} \frac{
f^{\{n\}}\myn(R\myp)}{n!\mypp N^n}\,w^n  + \frac{c_\beta\mypp
(-w)^\beta}{N^\beta} +O(N^{-\beta-\delta}),
\end{align}
with $q=\lfloor s\rfloor\ge1$. We can also write
(cf.~\eqref{eq:1/z^})
\begin{equation*}
\frac{1}{z^{N+1}}=\frac{1}{R^{N+1}}\,\exp\mynn\left\{\sum_{n=1}^{q+1}
\frac{(-w)^n }{n N^{n-1}}+O(N^{-q-1})\right\}.
\end{equation*}
Substituting these expansions into the integral
$\mathcal{I}_N^{\mypp 2}$ over the contour $\gamma_2$ (see
\eqref{eq:I1...I4}), we obtain similarly to \eqref{eq:I2}
\begin{align}
\notag \mathcal{I}_N^{\mypp 2} \sim \frac{\rme^{\myp v\myp
\GN(R\myp)}}{R^{N}N}\int_{\gamma_2'} &\bigl[1+\mathcal{P}_N(w)+
c_\beta N^{-\beta}(-w)^\beta
    +O(N^{-\beta-\delta})\bigr]\\
\label{eq:I2_theta=0}  &\begin{aligned}[t] \
\times\exp\myn\bigl\{\mathcal{Q}_N(w)+
 v\myp a_s N^{1-s}\myp(-w)^s
    &+O(N^{-s-\delta})\bigr\}\\
&\times\exp\myn\{-d_1 w\}
    \,\rmd{w},
\end{aligned}\end{align}
where $d_1=1-v\myp g_\kappa^{\{1\}}\myn(R\myp)>0$
(see~\eqref{eq:chi}) and $\mathcal{P}_N(w)$, $\mathcal{Q}_N(w)$ are
polynomials in $w$,
\begin{align*}
\mathcal{P}_N(w)&=\sum_{1\le n<\beta} \frac{
f^{\{n\}}\myn(R\myp)}{n!\mypp
N^n}\,w^n,\\
\mathcal{Q}_N(w)&=v\sum_{n=1}^{q}
\frac{g_\theta^{\{n\}}\myn(R\myp)}{n!\mypp N^n}\,w^n+ v\sum_{n=2}^q
\frac{g_\kappa^{\{n\}}\myn(R\myp)}{n!\mypp N^{n-1}
}\,w^n+v\sum_{n=2}^q \frac{(-w)^n}{n N^{n-1}}\mypp.
\end{align*}
Noting that $\mathcal{Q}_N(w)=O(N^{-1})$ uniformly in
$w\in\gamma_2^{\myp\prime}$, we can Taylor expand the exponential
under the integral in \eqref{eq:I2_theta=0} keeping the terms up to
the order $N^{1-s}$. Thus, we obtain
\begin{equation*}
\mathcal{I}_N^{\mypp 2} \sim \frac{\rme^{\myp v\myp
\GN(R\myp)}}{R^{N}N}\int_{\gamma_2'} \left[\mathcal{R}_N(w)+
 v\myp a_s N^{1-s}\,(-w)^s +c_\beta N^{-\beta}(-w)^\beta
    \right]\rme^{-d_1 w}\,\rmd{w},
\end{equation*}
where $\mathcal{R}_N(w)$ is the resulting polynomial in~$w$.

A similar estimation holds for the integral $\mathcal{I}_N^{\mypp
3}$ (cf.~\eqref{eq:I3}). As a result, we get (cf.~\eqref{eq:Hankel})
\begin{equation}\label{eq:Cauchy-theta=0}
\oint_{\gamma}\,\frac{\rme^{\myp v\myp
\GN(z)}}{z^{N+1}}\,\rmd{z}\sim \frac{\rme^{\myp v\myp
\GN(R\myp)}}{R^{N}\myp N} \int_{\gamma^{\myp\prime}}
\left[\mathcal{R}_N(w)+
 v\myp a_s N^{1-s}\,(-w)^s+c_\beta N^{-\beta}(-w)^\beta
    \right]\rme^{-d_1 w}\,\rmd{w}.
\end{equation}
Since $\mathcal{R}_N(w)$ is an entire function of $w\in\CC$, its
contribution to the integral \eqref{eq:Cauchy-theta=0} vanishes,
\begin{equation}\label{eq:int-R}
\int_{\gamma^{\myp\prime}} \mathcal{R}_N(w) \,\rme^{-d_1
w}\,\rmd{w}=0.
\end{equation}
Furthermore, changing the integration variable $d_1w\mapsto w$ and
transforming the contour $\gamma^{\myp\prime}$ to the loop contour
$\gamma^{\myp\prime\prime}$ (see before
equation~\eqref{eq:Hankel-based}), we obtain
\begin{equation}\label{eq:int(-w)1}
\int_{\gamma^{\myp\prime}}  (-w)^s
    \,\rme^{-d_1 w}\,\rmd{w}= (d_1)^{-1-s}\int_{\gamma^{\myp\prime\prime}}(-w)^s
    \,\rme^{-w}\,\rmd{w}=(d_1)^{-1-s}\,\frac{2\pi\rmi}{\Gamma(-s)}\mypp,
\end{equation}
according to Hankel's identity akin to \eqref{eq:Hankel-based} (see
\cite[\S\myp B.3, Theorem~B.1, p.~745]{FlSe09}). As for the term
$(-w)^\beta$ in \eqref{eq:Cauchy-theta=0}, if $\beta\in\NN$ then its
contribution also vanishes (cf.~\eqref{eq:int-R}), but if $\beta$ is
non-integer then, similarly to \eqref{eq:int(-w)1}, we have
\begin{equation}\label{eq:int(-w)2}
\int_{\gamma^{\myp\prime}}  (-w)^\beta \,\rme^{-d_1 w}\,\rmd{w}=
(d_1)^{-1-\beta}\,\frac{2\pi\rmi}{\Gamma(-\beta)}\mypp.
\end{equation}
It remains to note that the contributions \eqref{eq:int(-w)1} and
\eqref{eq:int(-w)2} enter the asymptotic formula
\eqref{eq:Cauchy-theta=0} with the weights $N^{1-s}$ and
$N^{-\beta}$, respectively, so the relationship between the
exponents $s-1$ and $\beta$ determines which of these two terms is
the principal one in the limit as $N\to\infty$. Accordingly,
retaining one of the power terms in \eqref{eq:Cauchy-theta=0} (or
both, if $\beta=s-1$) and dividing everything by $2\pi\rmi$
(cf.~\eqref{eq:I1...I4}) , we arrive at formulas
\eqref{eq:th3.3=0(i)}, \eqref{eq:th3.3=0(ii)} and
\eqref{eq:th3.3=0(iii)}, respectively.

Finally, formula \eqref{eq:HN-2=0} follows from
\eqref{eq:th3.3=0(i)} by setting $f(z)\equiv1$, $v=1$; note that in
this case the parameter $\beta$ in \eqref{eq:ftheta=0} can be
formally chosen to be any positive (integer) number, e.g.\ bigger
than $s-1$, so formula \eqref{eq:th3.3=0(i)} applies.
\end{proof}

\begin{remark}\label{rm:sing_g_kappa_and_theta}
Analogous considerations as in the proof of Theorem
\ref{thm:aux_asypmtotic_2a} can be used to handle the case with a
\emph{power-logarithmic} term
$b_s\myp(1-z/R\myp)^s\log\myn(1-z/R\myp)$ \,($b_s\ne0$) added to the
expansion \eqref{eq:g_kappa_as=0} of $g_\kappa(z)$, where the index
$s>1$ is now allowed to be \emph{integer}; this is motivated by the
choice $\kappa_j:=\kappa^* j^{-q}$ with $q\in\NN$\myp, leading to
the polylogarithm $g_\kappa(z)=\kappa^*\Li_{q+1}(z)$ with the
asymptotics \eqref{eq:Li-q+1} (see Section \ref{sec:polylog} below).
Furthermore, the condition of regularity of $g_\theta(z)$ at $z=R$
imposed in Theorem \ref{thm:aux_asypmtotic_2a} is also not
essential, and may be extended to include a power singularity of the
form $\tilde{a}_{s_1}(1-z/R\myp)^{s_1}$ (with
$\tilde{a}_{s_1}\myn\ne0$ and some non-integer $s_1\myn>0$) and,
possibly, a power-logarithmic singularity, similarly to what was
said above about $g_\kappa(z)$ (with any $s_1\myn>0$). We leave the
details to the interested reader.
\end{remark}



\subsection{The critical case}\label{sec:critical}

\subsubsection{First theorems}

Here $g_\kappa^{\{1\}}\myn(R\myp)=1$, and the equation
$g_\kappa^{\{1\}}\myn(r)=v^{-1}$ is still solvable for all $v\ge1$,
with the unique root $r_v\le R$ (see~\eqref{eq:r_v}). Thus, the same
argumentation may be applied as in the proof of
Theorem~\ref{thm:aux_asypmtotic}, but for this to cover the case
$v=1$ one has to assume that $g_\theta(R\myp)<\infty$, because the
circle $z=r_v \mypp\rme^{\myp\rmi t}$ to be used in Cauchy's
integral formula passes through the singularity $z=R$ if $v=1$ (with
$r_1=R$). Recall that the function $\GN(z)$ is defined in
\eqref{eq:generating_h(v)}, and $b_2(r)$ is given by~\eqref{eq:b12}.
\begin{theorem}\label{thm:aux_asypmtotic_3}
Assume that both $g_\theta(z)$ and $g_\kappa(z)$ have radius of
convergence $R>0$, with $g_\theta(R\myp)<\infty$,
\,$g_\kappa(R\myp)<\infty$ and, moreover,
$g_\kappa^{\{1\}}\myn(R\myp) = 1$,
\,$0<g_\kappa^{\{2\}}\myn(R\myp)<\infty$. Let $f(z)$ be a function
holomorphic in the open disk $|z|<R$ and continuous up to the
boundary $|z|=R$\myp. Then, uniformly in $v\in[1,v_0]$ with an
arbitrary constant $v_0\ge1$, we have
\begin{equation}\label{eq:=1,sub0}
[z^{N}]\bigl[f(z)\exp\myn\{v\myp \GN(z)\}\bigr] \sim
\frac{f(r_v)\exp\myn\{v \myp \GN(r_v)\}}{r_v^{N}\sqrt{2\pi N v\myp
b_2(r_v)}}\mypp,\qquad N\to\infty.
\end{equation}
In particular,
\begin{equation}\label{eq:HN-3}
H_N\sim \frac{\exp\myn\{\GN(R\myp)\}}{R^{N}\sqrt{2\pi N\myp
b_2(R\myp)}}\mypp,\qquad N\to\infty.
\end{equation}
\end{theorem}

\begin{proof}
Note that Taylor's expansion \eqref{eq:admissible_expansion} extends
to $r=R$\myp, with $b_1(R\myp)=g_\kappa^{\{1\}}\myn(R\myp)=1$,
\,$b_2(R\myp)=1+g_\kappa^{\{2\}}\myn(R\myp)<\infty$
(see~\eqref{eq:b12}). Then, on account of the continuity of
$g_\theta(z)$ and $f(z)$ on each circle $|z|=r_v\le R$\myp, it is
easy to check that the proof of Theorem~\ref{thm:aux_asypmtotic} is
valid with any $r_v\le R$\myp, thus leading to \eqref{eq:=1,sub0}
(cf.~\eqref{eq:th3.2}). As before, formula \eqref{eq:HN-3} follows
from \eqref{eq:=1,sub0} by setting $f(z)\equiv1$ and $v=1$, making
$r_v|_{v=1}=R$\myp.
\end{proof}

If $g_\theta(z)$ has a logarithmic singularity at $z=R$\myp, we can
still use the same argumentation as in
Theorem~\ref{thm:aux_asypmtotic_2} as long as $v\ne 1$.

\begin{theorem}\label{thm:aux_asypmtotic_4}
Assume that $g_\theta(z)$ and $g_\kappa(z)$ both have radius of
convergence $R>0$ and are holomorphic in some domain $\varDelta_0$
as in Definition~\textup{\ref{def:Delta0}}. Let
$g_\kappa^{\{1\}}\myn(R\myp)=1$,
\,$0<g_\kappa^{\{2\}}\myn(R\myp)<\infty$, and suppose that for some
$\theta^*\myn\ge 0$, \,$\delta>0$, as $z\to R$ \,\textup{(}$z\in
\varDelta_0$\textup{)},
\begin{align}\label{eq:theta*}
    g_\theta(z) &= -\theta^* \log\myn(1-z/R\myp)
    +O\bigl((1-z/R\myp)^\delta\bigr), \\
    \label{eq:g-expansion}
    g_\kappa(z) &= g_\kappa(R\myp) - (1-z/R\myp) +
    \frac{g_\kappa^{\{2\}}\myn(R\myp)}{2}\mypp(1-z/R\myp)^2+O\bigl((1-z/R\myp)^
    {2+\delta}\bigr).
\end{align}
Finally, let $f\colon\varDelta_0 \to \CC$ be a holomorphic function
such that for some $\beta\ge0$
\begin{equation*}
f(z) =
(1-z/R\myp)^{-\beta}\bigl\{1+O\bigl((1-z/R\myp)^\delta\bigr)\bigr\},\qquad
z\to R\ \ \ (z\in\varDelta_0).
\end{equation*}
Then, as $N\to\infty$,

\smallskip
\textup{(a)} \myp{}uniformly in $v\in[v_1,v_2]$ with arbitrary
$v_2\ge v_1>1$,
\begin{equation}\label{eq:=1,sub}
[z^{N}]\bigl[f(z)\exp\myn\{v\myp \GN(z)\}\bigr] \sim
\frac{f(r_v)\exp\myn\{v\myp \GN(r_v)\}}{r_v^{N}\sqrt{2\pi N v\myp
b_2(r_v)}}\mypp;
\end{equation}

\textup{(b)} \myp{}uniformly in $v\in[v_1,v_2]$ with arbitrary
$0<v_1\le v_2<1$, provided that\/ $v\myp\theta^*\myn+\beta>0$,
\begin{equation}\label{eq:=1,sup}
[z^{N}]\bigl[f(z)\exp\myn\{v\myp \GN(z)\}\bigr] \sim \frac{
\exp\myn\{N v\myp g_\kappa(R\myp) \}\cdot
\{N\myp(1-v)\}^{v\theta^*\myn+\beta-1}}{R^{N}
\,\Gamma(v\myp\theta^*\myn+\beta)}\mypp.
\end{equation}
\end{theorem}
\begin{proof} (a) Like in Theorem \ref{thm:aux_asypmtotic_3},
we can use the proof of Theorem~\ref{thm:aux_asypmtotic} as long as
$r_v<R$\myp, which is guaranteed by the condition $v>1$.

\smallskip
(b) Here the proof of Theorem~\ref{thm:aux_asypmtotic_2} applies as
long as $v<1/g_\kappa^{\{1\}}\myn(R\myp)=1$ (cf.~\eqref{eq:th3.3}).
\end{proof}

Note that, unlike Theorem~\ref{thm:aux_asypmtotic_3}, both parts (a)
and (b) of Theorem~\ref{thm:aux_asypmtotic_4} do not cover the value
$v=1$ and so do not provide the asymptotics of $H_N$. Moreover, the
asymptotic expressions on the right-hand side of \eqref{eq:=1,sub}
and \eqref{eq:=1,sup} vanish as $v\to1$, which suggests that no
uniform statements are possible on intervals (even one-sided) that
contain the point $v=1$. Nevertheless, the case $v=1$ can be handled
by a more careful adaptation of the proof of
Theorem~\ref{thm:aux_asypmtotic_2}. First, we need to compute some
complex integrals emerging in the asymptotics.

\subsubsection{Two auxiliary integrals}\label{sec:J(s)}
Let
$\gamma^{\myp\prime}=\gamma_1^{\myp\prime}\cup\gamma_2^{\myp\prime}\cup\gamma_3^{\myp\prime}$
be a continuous contour composed of the parts defined as follows
(see Fig.~\ref{fig:gamma'}),
\begin{equation}\label{eq:gamma(4)-supercr}
\begin{aligned}
\gamma^{\myp\prime}_1 :={}&\bigl\{w=-y\myp\rme^{-\rmi \varphi},\
\,y\in
(-\infty, -\epsilon\myp]\bigr\},\\[.2pc]
\gamma^{\myp\prime}_2 :={}&\bigl\{
w=\epsilon\mypp\rme^{-\rmi t},\ \, t\in[\varphi,\myp2\pi-\varphi\myp]\bigr\},\\[.2pc]
\gamma^{\myp\prime}_3 :={}&\bigl\{ w= y\myp\rme^{\myp\rmi \varphi},\
\, y\in [\epsilon,+\infty)\bigr\},
\end{aligned}
\end{equation}
with $\varphi\in (\pi/4,\myp\pi/2)$ and $\epsilon>0$. For any real
parameter $\xi$, let us consider the contour integral
\begin{equation}\label{eq:J}
J_\xi(\sigma):= \int_{\gamma'} (-w)^{-\sigma}\exp\myn\{-\xi w+w^2\}
\,\rmd{w},\qquad \sigma\in\CC.
\end{equation}
The determination of $(-w)^{-\sigma}$ in \eqref{eq:J} is defined by
the principal branch when $w$ is negative real, and then extended
uniquely by continuity along the contour. An absolute value estimate
gives for $w\in\gamma^{\myp\prime}_1\cup\gamma^{\myp\prime}_3$
\begin{equation}\label{eq:Jordan}
\left| (-w)^{-\sigma}\exp\myn\{-\xi w+w^2\}\right|\le
|w|^{-\Re(\sigma)}\,\rme^{\myp|\sigma|\myp\pi}
\exp\myn\{-\xi|w|\cos\varphi+|w|^2\cos 2\varphi\},
\end{equation}
where $\cos 2\varphi <0$ according to the chosen range of~$\varphi$.
Hence, the integral \eqref{eq:J} is absolutely convergent for all
complex $s$ and, moreover, the function $\sigma\mapsto
J_\xi(\sigma)$ is holomorphic in the entire complex plane $\CC$. An
explicit analytic continuation from the domain $\Re\myp(\sigma)<0$
is furnished through the following functional equation (which can be
easily obtained from \eqref{eq:J} by integration by parts),
\begin{equation*}
2J_\xi(\sigma)=(\sigma+1)\mypp J_\xi(\sigma+2)-\xi
J_\xi(\sigma+1),\qquad \sigma\in\CC.
\end{equation*}
From the estimate \eqref{eq:Jordan} it also follows that the
integral \eqref{eq:J} does not depend on the choice of the angle
$\varphi$ and the arc radius $\epsilon>0$; in the computation below,
we will choose $\varphi=\pi/2$ and take the limit $\epsilon\to0$.

It is straightforward to calculate a few values of $J_\xi(\sigma)$,
such as
\begin{equation*}
J_0(-1)=0,\qquad J_0(1)=\rmi\myp\pi, \qquad
J_\xi(0)=\rmi\mypp\sqrt{\pi}\,\rme^{-\xi^2/4}\mypp.
\end{equation*}
A general formula for the integral \eqref{eq:J} is established in
the next lemma.

\begin{lemma}\label{lm:J}
For any $\xi\in\RR$ and all $\sigma\in\CC$, there is the identity
\begin{equation}\label{eq:J=Gamma}
J_\xi(\sigma)=
\frac{\rmi\myp\pi\,\exp\myn\{-\xi^2/4\}}{\Gamma\bigl((\sigma+1)/2\bigr)\,\Gamma(\sigma/2)}
\,\sum_{n=0}^\infty
\Gamma\mynn\left(\frac{\sigma+n}{2}\right)\frac{\,\xi^n}{n!}\mypp .
\end{equation}
In particular, for $\xi=0$ we have
\begin{equation}\label{eq:J=Gamma0}
J_0(\sigma)=\frac{\rmi\myp\pi}
{\Gamma\bigl((\sigma+1)/2\bigr)}\mypp,\qquad \sigma\in\CC.
\end{equation}
\end{lemma}

\begin{proof} It suffices to prove \eqref{eq:J=Gamma} for real $\sigma<0$
(an extension to arbitrary $\sigma\in\CC$ will follow by analytic
continuation). As already mentioned, we can choose $\varphi=\pi/2$
in \eqref{eq:gamma(4)-supercr}, so that the parts of the original
contour $\gamma^{\myp\prime}=\gamma^{\myp\prime}_1 \cup
\gamma^{\myp\prime}_2\cup\gamma^{\myp\prime}_3$ take the form
\begin{equation}\label{eq:gamma(4)'}
\begin{aligned}
\gamma^{\myp\prime}_1 &=\bigl\{ w= \rmi y,\ \, y\in [-\infty,
-\epsilon\myp]\bigr\},\\[.2pc]
\gamma^{\myp\prime}_2 &=\bigl\{
w=\epsilon\,\rme^{-\rmi t},\ \, t\in[\pi/2,\myp3\myp\pi/2\myp]\bigr\}, \\[.2pc]
\gamma^{\myp\prime}_3 &=\bigl\{ w= \rmi y,\ \, y\in
[\epsilon,+\infty]\bigr\}.
\end{aligned}
\end{equation}
Here the parameter $\epsilon>0$ is arbitrary, and the idea is to
send it to zero.

First of all, since we have assumed that $-\sigma>0$, the integral
over $\gamma^{\myp\prime}_2$ vanishes as $\epsilon\to0$. Next, using
the parameterization \eqref{eq:gamma(4)'} of the contour
$\gamma^{\myp\prime}_1$ we compute
\begin{align}
\notag \int_{\gamma^{\myp\prime}_1}  (-w)^{-\sigma} \exp\myn\{-\xi
w+w^2\} \,\rmd{w} & =
\rmi\int_{-\infty}^{-\epsilon} (-\rmi\myp y)^{-\sigma} \exp\myn\{-\rmi\myp \xi\myp y-y^2\}\, \rmd{y}\\
\notag &=
\rmi\int_{\epsilon}^{\infty} (\rmi\myp y)^{-\sigma}\exp\myn\{\rmi\myp \xi\myp y-y^2\}\, \rmd{y}\\
\notag &= \rmi\myp\bigl(\rme^{\myp\rmi \pi/2}\bigr)^{-\sigma}
\int_\epsilon^\infty y^{-\sigma}\exp\myn\{\rmi\myp \xi\myp y-y^2\}\, \rmd{y}\\
\label{eq:determin1}&\to\rmi \int_0^\infty
y^{-\sigma}\exp\mynn\left\{\rmi\left(\xi\myp y-\frac{\pi
\sigma}{2}\right)-y^2\right\}\myp\rmd{y},\qquad \epsilon \to 0.
\end{align}
A similar computation for $\gamma^{\myp\prime}_3$  gives
\begin{align}
\notag \int_{\gamma'_3}  (-w)^{-\sigma} \exp\myn\{-\xi w+w^2\}
\,\rmd{w} & =
\rmi\int_{\epsilon}^{\infty} (-\rmi\myp y)^{-\sigma}\exp\myn\{-\rmi\myp \xi\myp y-y^2\}\, \rmd{y}\\
\notag &= \rmi\myp\bigl(\rme^{-\rmi \pi/2}\bigr)^{-\sigma}
\int_\epsilon^\infty y^{-\sigma}\exp\myn\{-\rmi\myp \xi\myp y-y^2\}\, \rmd{y}\\
\label{eq:determin2} &\to \rmi \int_0^\infty
y^{-\sigma}\exp\mynn\left\{-\rmi\left(\xi\myp y-\frac{\pi
\sigma}{2}\right)-y^2\right\}\myp\rmd{y},\qquad \epsilon \to 0.
\end{align}
Combining \eqref{eq:determin1} and \eqref{eq:determin2}, by using
Euler's formula $\rme^{\myp\rmi\myp\phi}+\rme^{-\rmi\myp\phi}
=2\cos\phi$ and some elementary trigonometric identities we obtain
\begin{align}
\notag \frac12\myp (-\rmi\myp) J_\xi(\sigma) &= \int_0^\infty
y^{-\sigma}\cos\mynn\left(\xi\myp y-\frac{\pi
\sigma}{2}\right)\myp\rme^{-y^2}\, \rmd{y}\\
\label{eq:A1+A2} &=  A_1(\xi,\sigma)\mypp\sin \frac{\pi
(1-\sigma)}{2} +  A_2(\xi,\sigma)\mypp\sin \frac{\pi \sigma}{2},
\end{align}
where
\begin{equation}\label{eq:A12}
A_1(\xi,\sigma):=\int_0^\infty y^{-\sigma}\cos\myn(\xi\myp
y)\,\rme^{-y^2}\, \rmd{y},\ \ \quad A_2(\xi,\sigma):=\int_0^\infty
y^{-\sigma}\sin\myn(\xi\myp y)\,\rme^{-y^2}\, \rmd{y}.
\end{equation}
For $-\sigma>0$ the integrals \eqref{eq:A12} are given by (see
\cite[\#\myp3.952\myp(7,\,8), p.~503]{GR})
\begin{align}\label{eq:A1_F}
A_1(\xi,\sigma)&=\frac12\,\Gamma\mynn\left(\frac{1-\sigma}{2}\right)\,\rme^{-\xi^2/4}\,
{}_1F_1\mynn\left(\frac{\sigma}{2},\frac{1}{2},\frac{\,\xi^2}{4}\right),\\[.2pc]
\label{eq:A2_F}
A_2(\xi,\sigma)&=\frac12\mypp\Gamma\mynn\left(1-\frac{\sigma}{2}\right)\;\xi\,\rme^{-\xi^2/4}\,{}_1
F_1\mynn\left(\frac{\sigma+1}{2},\frac{3}{2},\frac{\,\xi^2}{4}\right),
\end{align}
where ${}_1 F_1(a,b,z)$ is the \emph{confluent hypergeometric
function} (see \cite[\#\myp9.210\myp(1), p.\:1023]{GR}),
\begin{equation}\label{eq:hypergeometric}
{}_1 F_1(a,b,z):= 1+\sum_{n=1}^\infty
\frac{\Gamma(a+n)\myp\Gamma(b)}{\Gamma(a)\myp\Gamma(b+n)}\,\frac{z^n}{n!}\mypp,
\qquad z\in\CC,
\end{equation}
with $b\notin -\NN_0$. It is easy to see (e.g., using the ratio
test) that the series \eqref{eq:hypergeometric} converges for all
$z\in\CC$, and hence ${}_1 F_1(a,b,z)$ is an entire function of~$z$.
In particular, ${}_1 F_1(0,b,z)\equiv1$.

Substituting the expressions \eqref{eq:A1_F}, \eqref{eq:A2_F} into
\eqref{eq:A1+A2} and using (twice) the well-known complement formula
for the gamma function (see, e.g., \cite[\S\myp{}B.3, pp.\
745--746]{FlSe09})
\begin{equation}\label{eq:Gamma-compl}
\Gamma(z)\,\Gamma(1-z)=\frac{\pi}{\sin\myn(\pi z)}\mypp,\qquad
z\in\CC,
\end{equation}
we obtain
\begin{equation}\label{eq:J=A1+A2}
(-\rmi\myp) J_\xi(\sigma)=\pi\,\rme^{-\xi^2/4}\left\{\frac{
{}_1F_1\mynn\bigl(\frac{\sigma}{2},\frac{1}{2},\frac{\,\xi^2}{4}\bigr)}{\Gamma\bigl(\frac{\sigma}{2}+\frac12\bigr)}
+\xi\,\frac{ {}_1F_1\mynn\bigl(\frac{\sigma}{2}+\frac12,\frac{3}{2},
\frac{\,\xi^2}{4}\bigr)}{\Gamma\bigl(\frac{\sigma}{2}\bigr)}\right\}.
\end{equation}
Furthermore, observe from \eqref{eq:hypergeometric} that
\begin{align}
\notag {}_1
F_1\mynn\left(\frac{\sigma}{2},\frac{1}{2},\frac{\,\xi^2}{4}\right)
&=\frac{1}{\Gamma\bigl(\frac{\sigma}{2}\bigr)}\sum_{k=0}^\infty
\frac{\Gamma\bigl(\frac{\sigma}{2}+k\bigr)\,\xi^{2k}}{\frac12
\cdot\frac32\cdots\bigl(\frac12+k-1\bigr)\,2^{2k}k!}\\
\label{eq:1F1_1}&=
\frac{1}{\Gamma\bigl(\frac{\sigma}{2}\bigr)}\,\sum_{k=0}^\infty
\frac{\Gamma\bigl(\frac{\sigma}{2}+k\bigr)\,\xi^{2k}}{(2k)!}\mypp,
\end{align}
and similarly
\begin{equation}\label{eq:1F1_2}
\xi\cdot
{}_1F_1\mynn\left(\frac{\sigma+1}{2},\frac{3}{2},\frac{\,\xi^2}{4}\right)
=\frac{1}{\Gamma\bigl(\tfrac{\sigma+1}{2}\bigr)}\,\sum_{k=0}^\infty
\frac{\Gamma\bigl(\tfrac{\sigma}{2}+\frac12+k\bigr)\,\xi^{2k+1}}{(2k+1)!}\mypp.
\end{equation}
Now, returning to \eqref{eq:J=A1+A2} and combining the sums in
\eqref{eq:1F1_1} and \eqref{eq:1F1_2} we arrive
at~\eqref{eq:J=Gamma}.

Finally, formula \eqref{eq:J=Gamma0} readily follows from
\eqref{eq:J=Gamma} by setting $\xi=0$.
\end{proof}

There is a simple probabilistic representation of the power series
part of the expression \eqref{eq:J=Gamma}, which will be helpful in
applications (see the proof of
Theorem~\ref{thm:asymptotic_Kon}\myp(c-iii) below).

\begin{lemma}\label{lm:Gamma} Let $X\ge 0$ be a random variable with
gamma distribution $\Gammad(\sigma/2)$ \,\textup{($\sigma>0$)}, that
is, with probability density
$f(x)=(1/\Gamma(\sigma/2))\,x^{\sigma/2-1}\mypp\rme^{-x}$, \,$x>0$.
Then the moment generating function of\/ $\sqrt{X}$ is given by
\begin{equation}\label{eq:Gamma-dist}
\EE\bigl(\rme^{\mypp\xi\myp\sqrt{X}}\myp\bigr)=\frac{1}{\Gamma(\sigma/2)}
\,\sum_{n=0}^\infty
\Gamma\mynn\left(\frac{\sigma+n}{2}\right)\frac{\,\xi^n}{n!}\mypp,\qquad
\xi\in\RR.
\end{equation}
\end{lemma}
\begin{proof}
The left-hand side of \eqref{eq:Gamma-dist} can be expanded in a
series
\begin{equation}\label{eq:Laplace-Gamma}
\EE\bigl(\rme^{\mypp\xi\myp\sqrt{X}}\myp\bigr)=\sum_{n=0}^\infty
\frac{\,\xi^n}{n!}\,\EE\bigl(X^{n/2}\bigr).
\end{equation}
The moments of $X$ in \eqref{eq:Laplace-Gamma} are easily computed,
$$
\EE\bigl(X^{n/2})=\frac{1}{\Gamma(\sigma/2)}\int_0^\infty
x^{n/2+\sigma/2-1}\mypp\rme^{-x}\,\rmd{x}=\frac{\Gamma\bigl((\sigma+n)/2\bigr)}{\Gamma(\sigma/2)}\mypp,
$$
and returning to \eqref{eq:Laplace-Gamma} we
obtain~\eqref{eq:Gamma-dist}.
\end{proof}

A similar argumentation as in Lemma \ref{lm:J} may be applied to the
integral
\begin{equation}\label{eq:J-tilde}
\tilde{J}_0(\sigma;s):=\int_{\gamma'}
(-w)^{-\sigma}\exp\myn\{(-w)^s\} \,\rmd{w},\qquad \sigma\in\CC,\ \
1<s\le2,
\end{equation}
with the contour
\strut{}$\gamma^{\myp\prime}=\gamma_1^{\myp\prime}\cup\gamma_2^{\myp\prime}\cup\gamma_3^{\myp\prime}$
as defined in~\eqref{eq:gamma(4)-supercr}. Note that for $s=2$ the
integral \eqref{eq:J-tilde} is reduced to \eqref{eq:J=Gamma0}:
$\tilde{J}_0(\sigma;2)=J_0(\sigma)$. The general case is handled in
the next lemma.
\begin{lemma}\label{lm:J-tilde}
For any  $1<s\le 2$, the following identity holds
\begin{equation}\label{eq:J=Gamma0-tilde}
\tilde{J}_0(\sigma;s)=\frac{2\pi\rmi}{\displaystyle
s\,\Gamma\mynn\left(\frac{s-1+\sigma}{s}\right)}\mypp,\qquad
\sigma\in\CC.
\end{equation}
\end{lemma}
\begin{proof}
Like in the proof of Lemma \ref{lm:J}, it suffices to consider the
case $-\sigma>0$. Then, similarly to \eqref{eq:determin1} and
\eqref{eq:determin2}, we obtain (cf.~\eqref{eq:A1+A2})
\begin{equation}\label{eq:B12}
\tfrac12\myp(-\rmi\myp)\myp \tilde{J}_0(\sigma;s) = \int_0^\infty
y^{-\sigma}\cos\mynn\mynn\left(\frac{\pi\sigma}{2}-y^s \sin\frac{\pi
s}{2}\right)\myn\exp\mynn\mynn\left\{y^s \cos\frac{\pi
s}{2}\right\}\, \rmd{y}.
\end{equation}
Note that, according to the condition $1<s\le2$, we have
$\cos\myn(\pi s/2)<0$. By the change of variable $y^s=x$ and with
the help of some standard trigonometric identities, the integral on
the right-hand side of \eqref{eq:B12} is rewritten as
\begin{align}\label{eq:B1+B2}
\frac{1}{s}\left(B_1(\sigma;s)\myp\cos \frac{\pi \sigma}{2} +
B_2(\sigma;s)\myp\sin \frac{\pi \sigma}{2}\right),
\end{align}
where
\begin{align*}
B_1(\sigma;s):={}&\int_0^\infty x^{-1+(1-\sigma)/s}
\cos\mynn\mynn\left(x\sin \frac{\pi s}{2}\right)\exp\mynn\mynn\left\{x \cos\frac{\pi s}{2}\right\}\myp \rmd{y},\\
B_2(\sigma;s):={}&\int_0^\infty
x^{-1+(1-\sigma)/s}\,\sin\mynn\mynn\left(x\sin \frac{\pi
s}{2}\right)\myn\exp\mynn\mynn\left\{x \cos\frac{\pi
s}{2}\right\}\myp \rmd{y}.
\end{align*}
These integrals can be explicitly computed as follows (see
\cite[\#\myp3.944\mypp(5,\,6), p.~498]{GR})
\begin{align*}
B_1(\sigma;s)=\Gamma\mynn\left(\frac{1-\sigma}{s}\right)\myn\cos\psi,\qquad
B_2(\sigma;s)=\Gamma\mynn\left(\frac{1-\sigma}{s}\right)\myn\sin\psi,
\end{align*}
where
$$
\psi:=\frac{\pi\myp(1-\sigma)(2-s)}{2s}\mypp.
$$
Substituting this into \eqref{eq:B12} and \eqref{eq:B1+B2} we get
\begin{align*}
\tfrac12\myp(-\rmi\myp)\myp \tilde{J}_0(\sigma;s)
&=\frac{1}{s}\,\Gamma\mynn\left(\frac{1-\sigma}{s}\right)\left(\cos\psi\mypp\cos
\frac{\pi \sigma}{2} + \sin\psi\mypp\sin \frac{\pi
\sigma}{2}\right)\\
&=\frac{1}{s}\,\Gamma\mynn\left(\frac{1-\sigma}{s}\right)\myn
\cos\left(\psi- \frac{\pi \sigma}{2}\right)\\
&=\frac{1}{s}\,\Gamma\mynn\left(\frac{1-\sigma}{s}\right)\myn\sin\frac{\pi(1-\sigma)}{s}\\
&=\frac{\pi}{\displaystyle
s\,\Gamma\mynn\left(\frac{s-1+\sigma}{s}\right)}\mypp,
\end{align*}
again using the complement formula \eqref{eq:Gamma-compl}. Hence,
the result \eqref{eq:J=Gamma0-tilde} follows.
\end{proof}

\subsubsection{Asymptotic theorems with $v\approx 1$}
We are now in a position to obtain ``dynamic'' asymptotic results
for the critical case in the neighbourhood of $v=1$.
For simplicity, we omit the pre-exponential factor $f(z)$ (which
will not be needed). First, let us consider a ``regular'' case where
the second derivative of $g_\kappa(z)$ at $z=R$ is finite.

\begin{theorem}\label{thm:aux_asypmtotic_5}
Let $g_\theta(z)$ and $g_\kappa(z)$ satisfy the conditions of\/
Theorem~\textup{\ref{thm:aux_asypmtotic_4}} with
${\theta^*\mynn\ge0}$ \textup{(}see \eqref{eq:theta*}\textup{)},
\strut{}for a suitable domain of holomorphicity $\varDelta_0$.
Assume also that $g_\kappa^{\{1\}}\myn(R\myp)=1$ and
$0<g_\kappa^{\{2\}}\myn(R\myp)<\infty$. Finally, let
$f\colon\varDelta_0 \to \CC$ be a holomorphic function such that
\begin{equation*}
f(z) = 1+O\bigl((1-z/R\myp)^{\delta}\bigr), \qquad z\to R\ \ \
(z\in\varDelta_0).
\end{equation*}
Then for any $u\ge0$, as $N\to\infty$,
\begin{align}
\notag [z^{N}]\bigl[f(z)\exp\myn\{\rme^{-u/\sqrt{N}}\GN(z)\}\bigr]
&\sim\frac{\exp\myn\bigl\{Ng_\kappa(R\myp)-\sqrt{N}\myp u\mypp
    g_\kappa(R\myp)+\tfrac{1}{2}\myp u^2 g_\kappa(R\myp)\bigr\}}{R^{N}}\\
\label{eq:th3.3-dyn} &\qquad \times\bigl(\tfrac{1}{2}\myp N\myp
b_2(R\myp)\bigr)^{(\theta^*\myn-1)/2}\,\frac{J_{\tilde{u}}(\theta^*)}{2\pi\rmi}\mypp,
\end{align}
where $\tilde{u}:=u\,\sqrt{2/b_2(R\myp)}$ and
$J_{\tilde{u}}(\theta^*)$ is given by formula~\eqref{eq:J=Gamma}. In
particular,
\begin{equation}\label{eq:HN-4}
H_N\sim  \frac{\exp\myn\{N g_\kappa(R\myp) \}\cdot
\bigl(\tfrac{1}{2}\myp N\myp b_2(R\myp)\bigr)^{(\theta^*\myn-1)/2}
}{2R^{N}\,\Gamma\bigl(\frac{1+\theta^*}{2}\bigr)}\mypp, \qquad
N\to\infty.
\end{equation}
\end{theorem}

\begin{proof}
Setting $v=\rme^{-u\sqrt{N}}$, we adapt the proof of
Theorem~\ref{thm:aux_asypmtotic_2} by making the following
modifications:
\begin{itemize}
\item[(i)] the angle $\varphi$ in the
specification \eqref{eq:gamma1-4} of the contour $\gamma$ is chosen
in the range $\varphi\in(\pi/4,\myp\pi/2)$, so that $\cos
2\varphi<0$;
\item[(ii)] the rescaling coefficient $N^{-1}$
throughout \eqref{eq:gamma1-4} is \strut{}replaced by $N^{-1/2}$;
\item[(iii)]
under the change of variables $z=R\mypp(1+w N^{-1/2})$ (cf.\
\eqref{eq:w2}), Taylor's formulas \eqref{eq:g-kappa<} and
\eqref{eq:1/z^} are extended to the corresponding second-order
expansions (in particular, using \eqref{eq:g_kappa_as=0}).
\end{itemize}

Note that due to the criticality assumption, the quantity
$d_1=1-v\myp g_\kappa^{\{1\}}\myn(R\myp)$ (see~\eqref{eq:chi}) is
reduced to zero for $v=1$, so that the estimate \eqref{eq:d1_exp} of
the asymptotic contribution of the integral $\mathcal{I}_N^{\mypp
4}$ becomes void. This can be fixed by taking advantage of an
enhanced version of the inequality \eqref{eq:mu-kappa} provided by
Lemma \ref{lm:NA1} (see \eqref{eq:log|z|}), leading to an improved
bound
$$
|\mathcal{I}_N^{\mypp 4}| =\frac{O(1)\,\exp\myn\{N
g_\kappa(R\myp)\}}{R^{N}}\,\exp\myn\biggl\{-N\myp\varepsilon\log\frac{R^{\myp\prime}}{R}\biggr\},\qquad
N\to\infty.
$$

Next, under the substitution $z=R\mypp(1+w N^{-1/2})$, using
\eqref{eq:g-expansion} and the expansion
\begin{equation*}
\rme^{-u/\sqrt{N}}=1-\frac{u}{\sqrt{N}}+\frac{\,u^2}{2N}+O(N^{-3/2}),\qquad
N\to\infty,
\end{equation*}
we have, as $N\to\infty$,
\begin{align}
\notag \rme^{-u/\sqrt{N}} Ng_\kappa(z)
    &=Ng_\kappa(R\myp)-\sqrt{N}\myp u\mypp
    g_\kappa(R\myp)+\tfrac{1}{2}\myp u^2 g_\kappa(R\myp)\\
    &\quad +\sqrt{N}\myp w -u\myp
    w+\tfrac{1}{2}\mypp g_\kappa^{\{2\}}\myn(R\myp)\mypp w^2 + O\bigl(w^{2+\delta}
    N^{-\delta/2}\bigr).
    \label{eq:uxg}
\end{align}
Combined with the expansion (cf.~\eqref{eq:1/z^})
\begin{equation}\label{eq:1/zN}
\frac{1}{z^{N+1}}=\frac{1}{R^{N+1}}\,\exp\myn\bigl\{-\sqrt{N}\myp
w+\tfrac{1}{2}\myp w^2+O(w^3 N^{-3/2})\bigr\},\qquad N\to\infty,
\end{equation}
this leads to the following asymptotics of the integral
$\mathcal{I}_N^{\mypp 2}$ as $N\to\infty$ (cf.~\eqref{eq:I2})
\begin{align}
\notag\mathcal{I}_N^{\mypp 2}
&\sim\frac{\exp\myn\bigl\{Ng_\kappa(R\myp)-\sqrt{N}\myp u\mypp
    g_\kappa(R\myp)+\tfrac{1}{2}\myp u^2 g_\kappa(R\myp)\bigr\}}{R^{N}} \\
\label{eq:I2-cr-dynamic} &\hspace{7pc}\times \bigl(\tfrac{1}{2}\myp
N\myp b_2(R\myp)\bigr)^{(\theta^*\myn-1)/2}
\int_{\widetilde{\gamma}^{\myp\prime}_2} (-w)^{-\theta^*}
\mynn\exp\myn\{-\tilde{u}\myp w+w^2\}\,\rmd{w},
\end{align}
where
$\widetilde{\gamma}^{\myp\prime}_2:=\gamma^{\myp\prime}_2\,\sqrt{b_2(R\myp)/2}$
and $\tilde{u}:=u\,\sqrt{2/b_2(R\myp)}$.

The integral $\mathcal{I}_N^{\mypp 3}$ is asymptotically evaluated
in a similar fashion (cf.~\eqref{eq:I3}), leading to the same
formula as \eqref{eq:I2-cr-dynamic} but with the contour of
integration $\gamma^{\myp\prime}_3$ defined in \eqref{eq:gamma'3}.
Note that the error terms \eqref{eq:I3''} and \eqref{eq:I3'''},
formally invalidated by $d_1=0$, can be adapted by using
Lemma~\ref{lm:NA1} as described above.

As a result, collecting the principal asymptotic terms from formula
\eqref{eq:I2-cr-dynamic} and its other counterparts
(cf.~\eqref{eq:I3'}) produces the integral
$J_{\tilde{u}}(\theta^*)$, according to the notation~\eqref{eq:J}.
Hence, we arrive at the expression \eqref{eq:th3.3-dyn}, as claimed.
Finally, \eqref{eq:HN-4} immediately follows from
\eqref{eq:th3.3-dyn} by setting $u=0$ and on using
formula~\eqref{eq:J=Gamma0}.
\end{proof}

\begin{remark}\label{rm:super->cr2}
For $\theta^*\myn=0$, formula \eqref{eq:HN-4} gives the same
asymptotics of $H_N$ as formula \eqref{eq:HN-3} of Theorem
\ref{thm:aux_asypmtotic_3} (note that $\Gamma(1/2)=\sqrt{\pi}$\,).
Moreover, formula \eqref{eq:th3.3-dyn} with $\theta^*\myn=0$
coincides with the dynamic version obtained from Theorem
\ref{thm:aux_asypmtotic_3} by setting $v=\rme^{-u/\sqrt{N}}$; this
can be shown using calculations carried out in the proof of Theorem
\ref{thm:asymptotic_Kon}\myp(a) below.
\end{remark}

Let us now study the case with an infinite second derivative of
$g_\kappa(z)$ at $z=R$\myp.

\begin{theorem}\label{thm:aux_asypmtotic_6}
Let $g_\theta(z)$ and $g_\kappa(z)$ satisfy the assumptions of\/
Theorem~\textup{\ref{thm:aux_asypmtotic_4}} with
${\theta^*\mynn\ge0}$ \textup{(}see \eqref{eq:theta*}\textup{)},
including the condition $g_\kappa^{\{1\}}\myn(R\myp)=1$, but with
the expansion \eqref{eq:g-expansion} replaced by
\begin{equation*}
g_\kappa(z) = g_\kappa(R\myp) -(1-z/R\myp) +
 a_s\myp(1-z/R\myp)^s+O\bigl((1-z/R\myp)^{s+\delta}\bigr),
\end{equation*}
with $s\in(1,2)$, $a_s>0$ and $\delta>0$. Let $f\colon\varDelta_0
\to \CC$ be a holomorphic function satisfying
\begin{equation*}
f(z) = 1+O\bigl((1-z/R\myp)^{\delta}\bigr), \qquad z\to R\ \ \
(z\in\varDelta_0).
\end{equation*}
Then for any
$u\ge0$, as $N\to\infty$,
\begin{align}
\notag [z^{N}]\bigl[f(z)\exp\myn\{\rme^{-u/\sqrt{N}}\GN(z)\}\bigr]
&\sim\frac{\exp\myn\bigl\{Ng_\kappa(R\myp)-\sqrt{N}\myp u\mypp
    g_\kappa(R\myp)+\tfrac{1}{2}\myp u^2 g_\kappa(R\myp)\bigr\}}{R^{N}}\\
\label{eq:th3.3-dyn-s} &\qquad
\times\frac{(Na_s)^{(\theta^*\myn-1)/s}}{
s\,\Gamma\mynn\left(\frac{s-1+\theta^*}{s}\right)}\mypp.
\end{align}
In particular,
\begin{equation}\label{eq:HN-5}
H_N\sim  \frac{\exp\myn\{N g_\kappa(R\myp)\}\cdot
(Na_s)^{(\theta^*\myn-1)/s}
}{sR^{N}\,\Gamma\mynn\left(\frac{s-1+\theta^*}{s}\right)}\mypp,
\qquad N\to\infty.
\end{equation}
\end{theorem}

\begin{proof}
In what follows, we can assume that $0<\delta\le 2-s$. Again setting
$v=\rme^{-u\sqrt{N}}$, we adapt the proof of
Theorem~\ref{thm:aux_asypmtotic_5} by using the change of variables
$z=R\mypp(1+w N^{-1/s})$. Hence, the expansions \eqref{eq:uxg} and
\eqref{eq:1/zN} are replaced, respectively, by
\begin{align*}
\rme^{-u/\sqrt{N}} Ng_\kappa(z)
    &=Ng_\kappa(R\myp)-\sqrt{N}\myp u\mypp
    g_\kappa(R\myp)+\tfrac{1}{2}\myp u^2 g_\kappa(R\myp)\\
    &\quad +N^{1-1/s}\, w
    +a_s\mypp(-w)^s  + O\bigl(w^{s+\delta}
    N^{-\delta/s}\bigr)
\end{align*}
and
\begin{equation*}
\frac{1}{z^{N+1}}=\frac{1}{R^{N+1}}\,\exp\myn\bigl\{-N^{1-1/s}\,w+O(w^2
N^{1-2/s})\bigr\}.
\end{equation*}
Similarly as in the proof of Theorem \ref{thm:aux_asypmtotic_6},
this leads to the asymptotics (cf.~\eqref{eq:I2-cr-dynamic})
\begin{align*}
\mathcal{I}_N^{\mypp 2}
&\sim\frac{\exp\myn\bigl\{Ng_\kappa(R\myp)-\sqrt{N}\myp u\mypp
    g_\kappa(R\myp)+\tfrac{1}{2}\myp u^2 g_\kappa(R\myp)\bigr\}}{R^{N}} \\
&\hspace{7pc}\times
(Na_s)^{(\theta^*\myn-1)/s} \int_{\widetilde{\gamma}^{\myp\prime}_2}
(-w)^{-\theta^*} \mynn\exp\myn\{(-w)^s\}\,\rmd{w},
\end{align*}
with the rescaled contour
$\widetilde{\gamma}^{\myp\prime}_2:=a_s^{1/s}\gamma^{\myp\prime}_2$.
The integral $\mathcal{I}_N^{\mypp 3}$ is estimated similarly
(cf.~\eqref{eq:I3}). As a result, recalling the notation
\eqref{eq:J-tilde} we obtain
\begin{align*}
\notag [z^{N}]\bigl[\exp\myn\{\rme^{-u/\sqrt{N}}\GN(z)\}\bigr]
&\sim\frac{\exp\myn\bigl\{Ng_\kappa(R\myp)-\sqrt{N}\myp u\mypp
    g_\kappa(R\myp)+\tfrac{1}{2}\myp u^2 g_\kappa(R\myp)\bigr\}}{R^{N}}\\
&\qquad \times(Na_s)^{(\theta^*\myn-1)/s}
\,\frac{\tilde{J}_0(\theta^*\myn;s)}{2\pi\rmi}\mypp,
\end{align*}
and \eqref{eq:th3.3-dyn-s} follows on using
formula~\eqref{eq:J=Gamma0-tilde}.

Finally, formula \eqref{eq:HN-5} follows by setting $u=0$
in~\eqref{eq:th3.3-dyn-s}.
\end{proof}

\subsection{Examples}\label{sec:examples}
Let us give a few simple examples to illustrate the conditions of
the asymptotic theorems proved in Sections
\ref{sec:weak_admissible}\mypp--\mypp\ref{sec:critical}.

\subsubsection{Constant coefficients} \label{sec:example_constant}

For all $j\in\NN$\myp, let $\theta_j=\theta^*\ge0$, \,$\kappa_j =
\kappa^*\myn>0$. Then the corresponding generating functions
specialize to
\begin{equation}\label{eq:ex1g}
g_\theta(z)=-\theta^*\myn\log\myn(1-z), \qquad
g_\kappa(z)=-\kappa^*\myn\log\myn(1-z).
\end{equation}
Recalling the expression \eqref{eq:Hn_generatingN} and using the
binomial series expansion, we have explicitly
\begin{align}
\notag h_n(N) &= [z^n]\exp\myn\{\GN(z)\}=
[z^n]\,(1-z)^{-(\theta^*\myn+N\kappa^*\myn)}\\
\label{eq:Hn_generating_in_1}&= [z^n]\,\sum_{j=0}^{\infty}
\binom{\theta^*\myn+N\kappa^*\myn+j-1}{j} \myp z^j =
\binom{\theta^*\myn+N\kappa^*\myn+n-1}{n},
\end{align}
and with Stirling's formula for the gamma function this yields, for
any $n=N+O(1)$,
\begin{equation}\label{eq:asmptotic_Hn_simple_1}
h_{n}(N) \sim \frac{1}{\sqrt{2\pi
N}}\left(\frac{\kappa^*\myn+1}{\kappa^*}\right)^{\theta^*\myn+
N\kappa^*\myn-1/2} (\kappa^*\myn+1)^{n},\qquad N\to\infty.
\end{equation}

To apply the general machinery developed in
Section~\ref{sec:weak_admissible}, from \eqref{eq:ex1g} we find
\begin{equation}\label{eq:g-kappa-ex1}
g_\kappa^{\{1\}}\myn(z) = \frac{\kappa^* z}{1-z}\mypp, \qquad
g_\kappa^{\{2\}}\myn(z) = \frac{\kappa^* z^2}{(1-z)^2}\mypp,
\end{equation}
so that the expressions \eqref{eq:b12} specialize to
\begin{equation}\label{eq:b1b2-ex1}
b_1(r) = \frac{\kappa^*\myp r}{1-r}\mypp,\qquad  b_2(r)
=\frac{\kappa^*\myp r}{1-r}+\frac{\kappa^*\myp r^2}{(1-r)^2}=
\frac{\kappa^*\myp r}{(1-r)^2}\mypp.
\end{equation}
Here $R=1$ and $g_\kappa(1) = +\infty$, $g_\kappa^{\{1\}}\myn(1) =
+\infty$, so that, according to our terminological convention in
Section~\ref{sec:3.1} (see Definition~\ref{def:sub-sup}), we are
always in the subcritical regime. In view of \eqref{eq:g-kappa-ex1},
the solution $r_v$ of the equation \eqref{eq:r_v} is explicitly
given by
\begin{equation}\label{eq:r_v_ex1}
r_v = \frac{1}{v\kappa^*\myn+1}\mypp,\qquad v>0.
\end{equation}
Setting $v=1$, from \eqref{eq:ex1g}, \eqref{eq:b1b2-ex1} and
\eqref{eq:r_v_ex1} we find
$$
g_\theta(r_1)=\theta^*\log\frac{\kappa^*\myn+1}{\kappa^*}\mypp,\qquad
g_\kappa(r_1)=\kappa^*\log\frac{\kappa^*\myn+1}{\kappa^*}\mypp,\qquad
b_2(r_1)=\frac{\kappa^*\myn+1}{\kappa^*}\mypp,
$$
and Theorem~\ref{thm:aux_asypmtotic} (with $v=1$ and $f(z)=z^{N-n}$
readily yields the same result as~\eqref{eq:asmptotic_Hn_simple_1}.

\begin{remark}
In the degenerate case with $\kappa^*=0$, \,$\theta^*>0$ (which is
nothing more than the Ewens model), a direct calculation using
\eqref{eq:Hn_generating_in_1} gives
\begin{equation}\label{eq:h_theta}
h_n(N)= \binom{\theta^*\myn+n-1}{n}=
\frac{\Gamma(\theta^*\myn+n)}{n!\,\Gamma(\theta^*\myn)}\sim
\frac{n^{\theta^*\myn-1}}{\Gamma(\theta^*\myn)}\mypp,\qquad
n\to\infty.
\end{equation}
\end{remark}

\subsubsection{Polylogarithm} \label{sec:polylog}
To illustrate the supercritical and critical regimes, we need
examples with $g_\kappa^{\{1\}}\myn(R\myp)<\infty$. To this end, let
us set $\kappa_j:=\kappa^* j^{-s}$ ($j\in\NN$) with
$\kappa^*\myn>0$, $s\in\RR$\myp, so that the corresponding
generating function is proportional to the polylogarithm (see,
e.g.,~\cite{Lewin})
\begin{equation}\label{eq:g=Li}
g_\kappa(z)=\kappa^*\myn\Li_{s+1}(z)= \kappa^* \sum_{j=1}^\infty
\frac{z^j}{j^{s+1}}\mypp.
\end{equation}
For $s=0$ the model \eqref{eq:g=Li} is reduced to the case
$\kappa_j\equiv\kappa^*$ considered in
Section~\ref{sec:example_constant}. Clearly, $\Li_{s}(z)$ has the
radius of convergence $R=1$ for any $s\in\RR$\myp, and
$\Li_{s}(1)\equiv\zeta(s)< \infty$ for $s>1$, where $\zeta(\cdot)$
is the Riemann zeta function. Differentiating \eqref{eq:g=Li} we get
\begin{equation}\label{eq:g'-ex2}
g_\kappa^{\{1\}}\myn(z) = \kappa^* z\myp
\Li^{\myp\prime}_{s+1}(z)=\kappa^*\myn \Li_{s}(z),
\end{equation}
and so the supercriticality condition reads
\begin{equation}\label{eq:Li-super}
g_\kappa^{\{1\}}\myn(1)=\kappa^*\myn
\Li_{s}(1)=\kappa^*\zeta(s)<1\quad\Leftrightarrow\quad s>1,\quad
\kappa^*\myn<1/\zeta(s).
\end{equation}
Similarly,
\begin{equation}\label{eq:g'-ex3}
g_\kappa^{\{2\}}\myn(z)=\kappa^*\bigl(\Li_{s-1}(z)-\Li_s(z)\bigr)
\end{equation}
and so
$$
g_\kappa^{\{2\}}\myn(1)=\kappa^*\bigl(\zeta(s-1)-\zeta(s)\bigr)<\infty
\quad\Leftrightarrow\quad s>2.
$$
Hence, substituting \eqref{eq:g'-ex2} and \eqref{eq:g'-ex3} into
\eqref{eq:b12} we obtain
$$
b_1(r)=\kappa^*\myn \Li_{s}(r),\qquad b_2(r)=\kappa^*\Li_{s-1}(r),
\qquad 0\le r\le 1.
$$

Furthermore, from \eqref{eq:g'-ex2} we find that the root of the
equation \eqref{eq:r_v} is given by
\begin{equation}\label{eq:r_v-ex2}
r_v=\Li_s^{-1}\bigl((\kappa^*v)^{-1}\bigr),\qquad v\ge
\frac{1}{\kappa^*\zeta(s)}\mypp,
\end{equation}
where $\Li_s^{-1}$ is the inverse of $\Li_s$. Differentiating the
identity \eqref{eq:r_v-ex2} we also obtain
$$
\frac{r'_v}{r_v}=\frac{-\kappa^* v^{-2}}{\Li_{s-1}(r_v)}\mypp.
$$
In particular,
$$
r_1=\Li_s^{-1}(1/\kappa^*),\qquad
\frac{r'_1}{r_1}=\frac{-\kappa^*}{\Li_{s-1}(r_1)}\mypp.
$$

The supercritical case $\kappa^*\myn<1/\zeta(s)$ is more
interesting. It is known (see, e.g., \cite[\S\myp{}IV.9, p.~237, and
\S\myp{}VI.8, p.~408]{FlSe09} that the polylogarithm
$\Li_{s}(\cdot)$ can be analytically continued to the complex plane
$\CC$ slit along the ray $[1,+\infty)$. The asymptotic behaviour of
$\Li_{s}(z)$ near $z_0=1$ is specified as follows (see \cite[Theorem
VI.7, p.~408, and \S\myp{}VI.20, p.~411]{FlSe09}).
\begin{lemma}\label{lm:polylog-sing}
With the notation
\begin{equation}\label{eq:Li_w}
 \varpi:=-\log z\sim
\sum_{n=1}^\infty \frac{(1-z)^n}{n}\mypp,\qquad z\to1,
\end{equation}
the polylogarithm $\Li_s(z)$ satisfies the following asymptotic
expansion as $z\to1$\textup{:}

\smallskip \textup{(a)} \myp{}if\/ $s \in \RR\setminus\NN$ then
\begin{equation}\label{eq:Li-noninteger}
\Li_{s}(z)\sim\Gamma(1-s)\, \varpi^{s-1} +\sum_{n=0}^\infty
\frac{(-1)^n\myp\zeta(s-n)}{n!}\,\varpi^n\myp;
\end{equation}

\textup{(b)} \myp{}if\/ $s=q \in \NN$ then
\begin{equation}\label{eq:Li-integer}
\Li_{q}(z)\sim\frac{(-1)^q}{(q-1)!}\,\varpi^{q-1}\,(\log
\varpi-{\mathrm H}_{q-1}) +\sum_{n\ge 0,\,n\ne q-1} \frac{(-1)^n\myp
\zeta(q-n)}{n!}\,\varpi^n,
\end{equation}
where ${\mathrm H}_{q-1}:=\sum_{n=1}^{q-1}n^{-1}$ \textup{(}with
${\mathrm H}_{0}:=0$\textup{)}.
\end{lemma}

\begin{remark}
The restriction $n\ne q-1$ in the sum \eqref{eq:Li-integer} respects
the fact that the zeta function $\zeta(z)$ has a pole at $z=1$.
\end{remark}

\begin{remark}
In the simplest case $q=1$ the expansion \eqref{eq:Li-integer}
specializes to
$$
\Li_{1}(z)\sim -\log \varpi +\sum_{n=1}^\infty
\frac{(-1)^n\,\zeta(1-n)}{n!}\,\varpi^n,\qquad \varpi\to0,
$$
which should be contrasted with the explicit expression $\Li_1(z)
=-\log\myn(1-\rme^{-\varpi})$ \,($z=\rme^{-\varpi}$). Of course,
there is no contradiction; the corresponding identity
\begin{equation}\label{eq:log_varpi}
\log\frac{\varpi}{1-\rme^{-\varpi}}=\sum_{n=1}^\infty
\frac{(-1)^n\,\zeta(1-n)}{n!}\,\varpi^n,
\end{equation}
with the left-hand side defined at $\varpi=0$ by continuity as
$\log\myn(0/0):=\log 1=0$, can be verified using that
$\zeta(0)=-1/2$, \,$\zeta(1-n)=0$ for odd $n\ge3$ and
$\zeta(1-n)=-B_{n}/n$ for even $n\ge 2$ (see, e.g.,
\cite[\S\myp B\mypp11, p.\;747]{FlSe09}), where $B_n$ are the
\emph{Bernoulli numbers} (see, e.g.,
\cite[p.~268]{FlSe09}) defined by the generating function
$\varpi/(\rme^{\myp\varpi}-1)=\sum_{n=0}^\infty B_n\mypp\varpi^n/n!$
(note that $B_0=1$, $B_1=-1/2$ and $B_n=0$ for odd $n\ge 3$).
Indeed, \eqref{eq:log_varpi} is then rewritten as
$$
\log\frac{\varpi}{1-\rme^{-\varpi}}=
\frac{\varpi}{2}-\sum_{n=1}^\infty \frac{B_{n}}{n!\,n}\,\varpi^n,
$$
and the proof is completed by differentiation of both sides with
respect to $\varpi$.
\end{remark}

Explicit asymptotic expansion of $\Li_s(z)$ in terms of $1-z$, to
any required order, can be obtained by substituting the series
\eqref{eq:Li_w} into \eqref{eq:Li-noninteger} or
\eqref{eq:Li-integer} as appropriate. In particular, for the
generating function $g_\kappa(z)$ of the form \eqref{eq:g=Li} with
$q<s<q+1$ \,($q\in\NN_0$), from formulas \eqref{eq:Li_w} and
\eqref{eq:Li-noninteger} we get the Taylor-type expansion
\begin{equation}\label{eq:Li-s+1}
g_\kappa(z)= \sum_{n=0}^q \frac{(-1)^{n}\myp
g_\kappa^{\{n\}}(1)}{n!}\,(1-z)^n +\kappa^*\myp\Gamma(-s)\myp(1-z)^s
+O\bigl((1-z)^{q+1}\bigr),
\end{equation}
where the coefficients
\begin{equation*}
g_\kappa^{\{n\}}(1) =\kappa^*
\Li_{s+1}^{\{n\}}(1)=\kappa^*\sum_{j=1}^\infty
\frac{(j)_{n}}{j^{s+1}}\mypp,\qquad n=0,\dots,q,
\end{equation*}
are expressible as linear combinations of the zeta functions
\mypp$\zeta(s+1-k)$ \mypp($k=1,\dots,n$),
\begin{gather*}
\Li_{s+1}^{\{0\}}(1)=\zeta(s+1),\quad
\Li_{s+1}^{\{1\}}(1)=\zeta(s),\quad
\Li_{s+1}^{\{2\}}(1)=\zeta(s-1)-\zeta(s),\\[.2pc]
\Li_{s+1}^{\{3\}}(1)=\zeta(s-2)-3\mypp\zeta(s-1)+2\mypp\zeta(s),\quad
\text{etc.}
\end{gather*}

Furthermore, using similar arguments it is easy to see that the
expansion \eqref{eq:Li-s+1} can be differentiated any number of
times, yielding the asymptotics
\begin{alignat}{2}\label{eq:assum_derivative_p_Upsilon-low-Li}
g_\kappa^{\{n\}}\myn(z)&=
g_\kappa^{\{n\}}\myn(1)\mypp\bigl\{1+O\bigl((1-z)^{s-q}\bigr)\bigr\}
&&(n<s),\\[.2pc]
\label{eq:assum_derivative_p_Upsilon-Li} g_\kappa^{\{n\}}\myn(z)&=
\kappa^*\myp\Gamma(n-s)\,(1-z)^{s-n}\mypp\bigl\{1+O\bigl((1-z)^{q+1-s}\bigr)\bigr\}
\qquad &&(n>s).
\end{alignat}

By virtue of formula \eqref{eq:Li-s+1}, and by choosing suitable
(non-integer) values of the parameter $s>0$ in \eqref{eq:g=Li}, one
can easily construct examples matching the assumptions of each of
the asymptotic theorems in Sections \ref{sec:weak_admissible_2} and
\ref{sec:critical}. Moreover, the asymptotic formulas
\eqref{eq:assum_derivative_p_Upsilon-low-Li} and
\eqref{eq:assum_derivative_p_Upsilon-Li} make the polylogarithm
example \eqref{eq:g=Li} suitable for the setting of
Section~\ref{sec:long} below (see the assumptions
\eqref{eq:assum_derivative_p_Upsilon-low},~\eqref{eq:assum_derivative_p_Upsilon}).

The case with integer $s=q\in\NN_0$ leads to a logarithmic term in
the singular part of the asymptotic expansion at $z=1$. Indeed, if
$q=0$ then formula \eqref{eq:g=Li} is reduced to
\[
g_\kappa(z)=\kappa^*\Li_1(z)=\kappa^*\sum_{j=1}^\infty
\frac{z^j}{j}=-\kappa^*\log\myn(1-z),
\]
while $q\in\NN$ corresponds to $g_\kappa(z)=\kappa^*\Li_{q+1}(z)$,
so that using \eqref{eq:Li_w} and \eqref{eq:Li-integer} one obtains
with an arbitrary $\delta\in(0,1)$ (cf.~\eqref{eq:Li-s+1})
\begin{align}
\notag g_\kappa(z)= \sum_{n=0}^{q-1} \frac{(-1)^{n}\myp
g_\kappa^{\{n\}}\myn(1)}{n!}\,(1-z)^n
&-\kappa^*\myp\frac{(-1)^{q}}{q!}\,(1-z)^q \log\myn(1-z)\\[-.3pc]
\label{eq:Li-q+1} &+\kappa^* A_q\myp
(1-z)^q+O\bigl((1-z)^{q+\delta}\bigr),
\end{align}
where $A_q$ is the coefficient of the term $(1-z)^q$ arising from
the expansion \eqref{eq:Li-integer} (with $q$ replaced by $q+1$)
upon the substitution \eqref{eq:Li_w}.

\subsubsection{Perturbed polylogarithm}\label{sec:pert_polylog}
A natural extension of the polylogarithm example considered in the
previous subsection
is furnished by setting $\kappa_j:=\kappa^*
j^{-s}\left(1+\xi(j)\right)$ ($j\in\NN$), where $\kappa^*\myn>0$,
$s\ge0$, $1+\xi(j)\ge 0$ for all $j\in\NN$\myp, and the perturbation
function $z\mapsto\xi(z)$ is assumed to be analytic in the
half-plane $\Re\myp(z)>\frac14$ and to satisfy there the estimate
$\xi(z)=O(z^{-\epsilon})$, with some $\epsilon>0$. It follows that
the corresponding generating function
\begin{equation}\label{eq:xi-series}
g_\kappa(z) = \kappa^*\sum_{j=1}^\infty
\frac{1+\xi(j)}{j^{s+1}}\,z^j=\kappa^*\Li_{s+1}(z) +
\kappa^*\sum_{j=1}^\infty \frac{\xi(j)}{j^{s+1}}\, z^j
\end{equation}
is analytic in the disk $|z|<1$, with singularity at $z=1$.
Furthermore (see \eqref{eq:g'-ex2}),
\begin{equation}\label{eq:xi-series-derivarive}
g_\kappa^{\{1\}}\myn(z) = \kappa^*\myn \Li_{s}(z) +
\kappa^*\sum_{j=1}^\infty \frac{\xi(j)}{j^{s}}\, z^j.
\end{equation}
As suggested by the principal term $\Li_{s}(z)$ in
\eqref{eq:xi-series-derivarive}, the criticality occurs if $s>1$
(cf.~\eqref{eq:Li-super}); indeed, with a suitable constant $c>0$ we
have for sufficiently small $\kappa^*$
\begin{equation}\label{eq:xi-series-derivarive<1}
g_\kappa^{\{1\}}\myn(1) = \kappa^*\myn \Li_{s}(1) +
\kappa^*\sum_{j=1}^\infty \frac{\xi(j)}{j^{s}}\le
\kappa^*\bigl\{\zeta(s)+c\,\zeta(s+\epsilon)\bigr\}<1.
\end{equation}

\begin{lemma}\label{lm:polylog-extension}
Under the above conditions on $\xi(z)$ \textup{(}with
$s\ge0$\textup{)}, the function $g_\kappa(z)$ can be analytically
continued to the slit complex plane $\CC\setminus [1,+\infty)$.
\end{lemma}

\begin{proof} Since the claim is valid for the polylogarithm $\Li_{s+1}(z)$
(see Section~\ref{sec:polylog}), from \eqref{eq:xi-series} we see
that it suffices to prove the lemma for the series
$\sum_{j=1}^\infty a(j)\,z^j$, where $a(w):=\xi(w)/ w^{s+1}$.
Clearly, for any $\delta\in(0,\pi)$ we have the estimate
$a(w)=O\bigl(\rme^{\myp(\pi-\delta)\myp|w|}\bigr)$. Hence, the
Lindel\"of theorem (see \cite[\S\mypp{}IV.8, p.~237]{FlSe09} yields
the integral representation
\begin{equation}\label{eq_lindelof_integral_representation}
\sum_{j=1}^\infty a(j)\mypp z^j = -\frac{1}{2\pi \rmi} \int_{(1/2)-
\rmi\infty}^{(1/2)+\rmi\infty} a(w)\, (-z)^w \frac{\pi}{\sin\pi w}
\,\rmd{w},
\end{equation}
which provides an analytic continuation of the series to the domain
$\CC\setminus\{z: |\myn\arg\myn(z)|\le \delta\}$. Since $\delta>0$
can be taken arbitrarily small, this implies the analyticity in
$\CC$ slit along the ray $[0,+\infty)$. To complete the proof, it
remains to recall that $\sum_{j} a(j)\mypp z^j$ is analytic in
$|z|<1$.
\end{proof}

The next question is the asymptotic behaviour of the function
\eqref{eq:xi-series} as $z\to1$ with $z\in\CC\setminus [1,+\infty)$.
Application of the asymptotic formulas \eqref{eq:Li-noninteger} and
\eqref{eq:Li-integer} in the particular case $\xi(z)=z^{-\epsilon}$
(leading to $g_\kappa(z)=\kappa^*\myp\{\Li_{s+1}(z)
+\Li_{s+\epsilon+1}(z)\}$) motivates and illustrates the following
expansions, which now have to be finite (up to the order of
$(1-z)^s$) due to the limited information about the perturbation
function $\xi(z)$.

\begin{lemma}\label{lm:pert_polylog}
Suppose that $\xi(z)$ satisfies the same conditions as in Lemma
\textup{\ref{lm:polylog-extension}}, and let $q\in\NN_0$ be such
that $q\le s<q+1$. Then, as $z\to 1$ so that
$z\in\CC\setminus[1,+\infty)$,

\smallskip \textup{(a)} \myp{}for $s\notin\NN_0$ \,\textup{(}i.e., $q<s<q+1$\textup{)},
\begin{equation}\label{eq:expansion_xi}
g_\kappa(z) = \sum_{n=0}^q
\frac{(-1)^ng_\kappa^{\{n\}}\myn(1)}{n!}\,(1-z)^n+\kappa^*\myp\Gamma(-s)\myp(1-z)^s
+O\bigl((1-z)^{s+\epsilon_*}\bigr),
\end{equation}
where $\epsilon_*=\epsilon$ if\/ $s+\epsilon<q+1$, $\epsilon_*=1$
if\/ $s+\epsilon>q+1$, and $\epsilon_*$ is any number in
$(0,\epsilon)$ if\/ $s+\epsilon=q+1$\myp\textup{;}

\smallskip
\textup{(b)} \myp{}for $s=q\in\NN_0$\myp,
\begin{align} \notag
g_\kappa(z) = \sum_{n=0}^{q-1}
\frac{(-1)^ng_\kappa^{\{n\}}\myn(1)}{n!}\,(1-z)^n
&-\kappa^*\myp\frac{(-1)^{q}}{q!}\,(1-z)^q \log\myn(1-z)\\
\label{eq:Li-q+1*} &+\kappa^*B_q\,(1-z)^q
+O\bigl((1-z)^{q+\epsilon_*}\bigr),
\end{align}
where $B_q$ is some constant, $\epsilon_*=\epsilon$ if\/
$\epsilon<1$ and $\epsilon_*$ is any number in $(0,\epsilon)$ if\/
$\epsilon\ge 1$.

The expansions \eqref{eq:expansion_xi} and \eqref{eq:Li-q+1*} can be
differentiated $q$ times.
\end{lemma}

\begin{proof}[Sketch of\/ proof] This result is of marginal significance for
our purposes, as it will only be used for illustration in Section
\ref{sec:comparison_1}. Its full proof is quite tedious but follows
very closely the proof of a similar result for the polylogarithm
$\Li_{s+1}(z)$ (see details in \cite[\S\myp{}VI.8]{FlSe09}). Thus,
we opt to derive the expansion \eqref{eq:expansion_xi} only for
\emph{real} $z\uparrow 1$; an extension to complex
$z\in\CC\setminus[1,+\infty)$ is based on the Lindel\"of integral
representation~\eqref{eq_lindelof_integral_representation}.

Observe that it suffices to prove \eqref{eq:expansion_xi} for
$s\in(0,1)$; the case of an arbitrary (non-integer) $s>0$ may then
be handled via a suitable ($q$-fold) integration over the interval
$[z,1]$. To this end, using the substitution $z=\rme^{-\varpi}$
(with $\varpi\downarrow 0$ as $z\uparrow1$), from
\eqref{eq:xi-series} we obtain
\begin{equation}\label{eq:1+xi_series}
g_\kappa(z)=g_\kappa(1)-\kappa^*\sum_{j=1}^\infty
\frac{1-\rme^{-j\varpi}}{j^{s+1}}-\kappa^*\sum_{j=1}^\infty
\frac{\xi(j)\myp(1-\rme^{-j\varpi})}{j^{s+1}}\mypp.
\end{equation}
The first series in \eqref{eq:1+xi_series} may be rewritten as a
Riemann integral sum
\begin{equation}\label{eq:varpi^s}
\varpi^s\sum_{j=1}^\infty
\frac{1-\rme^{-j\varpi}}{(j\varpi)^{s+1}}\,\varpi=\varpi^s
\int_0^\infty \frac{1-\rme^{-x}}{x^{s+1}}\,\rmd{x}+O(\varpi),\qquad
\varpi\downarrow0,
\end{equation}
where the asymptotics on the right-hand side can be obtained using
Euler--Maclaurin's summation formula. The integral in
\eqref{eq:varpi^s} is easily computed via integration by parts,
\begin{equation}\label{eq:varpi^s*int}
\frac{1}{(-s)}\int_0^\infty
(1-\rme^{-x})\,\rmd\myp({x^{-s}})=\frac{1}{s}\int_0^\infty
x^{-s}\mypp\rme^{-x}\,\rmd{x}=\frac{\Gamma(1-s)}{s}=-\Gamma(-s).
\end{equation}
Next, using the estimate $\xi(j)=O(j^{-\epsilon})$ and assuming that
$s+\epsilon<1$, the second series in \eqref{eq:1+xi_series} is
estimated, similarly to \eqref{eq:varpi^s} and
\eqref{eq:varpi^s*int}, by
\begin{equation}\label{eq:varpi^s-correction}
O(1)\,\sum_{j=1}^\infty
\frac{1-\rme^{-j\varpi}}{j^{s+\epsilon+1}}=O(\varpi^{s+\epsilon}),\qquad
\varpi\downarrow0.
\end{equation}

Finally, collecting \eqref{eq:1+xi_series}, \eqref{eq:varpi^s},
\eqref{eq:varpi^s*int} and \eqref{eq:varpi^s-correction}, we obtain
$$
g_\kappa(z)=g_\kappa(1)+\kappa^*\Gamma(-s)\,\varpi^s +
O(\varpi^{s+\epsilon}),
$$
and the formula \eqref{eq:expansion_xi} (with $0<s<1$ and real
$z\uparrow1$) immediately follows by the substitution $\varpi =
-\log z = 1-z + O\bigl((1-z)^2\bigr)$.
\end{proof}

\section{Asymptotic Statistics of Cycles}\label{sec:cycles}

Throughout this section, we assume that the generating functions
$g_\theta(z)$ and $g_\kappa(z)$ satisfy the hypotheses of a suitable
asymptotic theorem in Section~\ref{sec:asymptotic_theorem}
--- namely, Theorem~\ref{thm:aux_asypmtotic} for the subcritical
case ($g_\kappa^{\{1\}}\myn(R\myp)>1$), Theorems
\ref{thm:aux_asypmtotic_2} or \ref{thm:aux_asypmtotic_2a} for the
supercritical case ($g_\kappa^{\{1\}}\myn(R\myp)<1$), and
Theorems~\ref{thm:aux_asypmtotic_3}, \ref{thm:aux_asypmtotic_5}
or~\ref{thm:aux_asypmtotic_6} for the critical case
($g_\kappa^{\{1\}}\myn(R\myp)=1)$.

\begin{remark}
The analytic conditions on the generating functions $g_\theta(z)$
and $g_\kappa(z)$ employed in Section~\ref{sec:asymptotic_theorem}
are difficult to convert into general sufficient conditions on the
underlying coefficients $(\theta_j)$ and $(\kappa_j)$, respectively.
For a better orientation in the asymptotic results below, it should
be helpful for the reader to bear in mind the examples (especially
the polylogarithm) considered in Section~\ref{sec:examples}.
\end{remark}

Let us define the quantity $r_{*}>0$ as
\begin{equation}\label{eq:r*}
r_{*}:=\left\{\begin{array}{ll} r_1,&\ \ g_\kappa^{\{1\}}\myn(R\myp)
\ge 1,\\[.2pc]
R\myp,&\ \ g_\kappa^{\{1\}}\myn(R\myp) \le 1,
\end{array}
\right.
\end{equation}
where  $r_1= r_v|_{v=1}$ is the (unique) root of the equation
\eqref{eq:r_v} with $v=1$, that is, $g_\kappa^{\{1\}}\myn(r_1)=1$.
Thus, we have
\begin{equation}\label{eq:r*1}
\begin{aligned}
g_\kappa^{\{1\}}\myn(r_{*}) &=1 \quad \text{if}\ \ \
g_\kappa^{\{1\}}\myn(R\myp)\ge 1,
\\[.1pc]
g_\kappa^{\{1\}}\myn(r_{*})&<1 \quad \text{if}\ \ \
g_\kappa^{\{1\}}\myn(R\myp)< 1.
\end{aligned}
\end{equation}

\subsection{Cycle counts}\label{sec:cyclecounts}

Our first result treats the asymptotics of the cycle counts $C_j$
(i.e., the numbers of cycles of length $j\in\NN$\myp, respectively,
in a random permutation $\sigma\in\SN$).

\begin{theorem}\label{thm:cycle_counts_normalized}
Let $r_{*}$ be as defined in~\eqref{eq:r*}.

\smallskip
\textup{(a)} \myp{}For each $m\in\NN$ and any integers
$n_1,\dots,n_m\ge0$, we have
\begin{equation}\label{eq:factorial_moments_Cm}
\lim_{N\to\infty}N^{-(n_1+\dots+n_m)}\:\Ens\!\mynn\left(\mynn\prod_{j=1}^m
(C_j)_{n_j}\right) = \prod_{j=1}^m\biggl(\frac{\kappa_j\mypp
r_{*}^j}{j}\biggr)^{\myn n_j}.
\end{equation}
In particular, the random variables $(C_j/N)$ are asymptotically
independent and, for each $j\in\NN$\myp, there is the convergence in
probability
\begin{equation}\label{eq:limits_Cm}
\frac{C_j}{N}\stackrel{p}{\longrightarrow}\frac{\kappa_j\mypp
r_{*}^j}{j}\mypp,\qquad N\to\infty.
\end{equation}

\smallskip
\textup{(b)} \myp{}If, for some $j\in\NN$\myp, $\kappa_j=0$ but
$\theta_j>0$ then for any integer $n\ge0$
\begin{equation}\label{eq:factorial_moments_Cm1}
\lim_{N\to\infty}\Ens\myn\bigl[(C_j)_{n}\bigr] =\biggl(
\frac{\theta_j\mypp r_{*}^j}{j }\biggr)^{\myn n}.
\end{equation}
Hence, $C_j$ converges weakly to a Poisson law with parameter
$\theta_j\mypp r_{*}^j/j$\mypp. The asymptotic independence of\/
$C_j$ with other cycle counts \textup{(}normalized or not, as
appropriate\textup{)} is preserved.
\end{theorem}

\begin{remark}
If both $\kappa_j=0$ and $\theta_j=0$ then, by the definition
\eqref{eq:def_near_spatial} of the measure $\Pns$, $C_j=0$ almost
surely (a.s.).
\end{remark}

\begin{proof}[Proof\/ of\/ Theorem
\textup{\ref{thm:cycle_counts_normalized}}] If all $\kappa_j>0$ then
Lemma \ref{lem:generatin_moments_factorial_Cm} implies that for any
$n_1,\dots,n_m\in\NN_0$
\begin{equation}\label{eq:prod-C}
N^{-(n_1+\dots+n_m)}\,\Ens\!\mynn\left(\mynn\prod_{j=1}^m(C_j)_{n_j}\right)\sim
\frac{h_{N-K_m}(N)}{H_N}\prod_{j=1}^m\left(
\frac{\kappa_j}{j}\right)^{\myn n_j},\qquad N\to\infty,
\end{equation}
where $K_m=\sum_{j=1}^m j\myp n_j$. Furthermore, note that
\begin{equation}\label{eq:N-Km}
h_{N-K_m}(N)=[z^{N-K_m}]\,\rme^{\myp\GN(z)}=[z^{N}]\bigl[z^{K_m}\mypp\rme^{\myp\GN(z)}\bigr].
\end{equation}
Hence, applying one of Theorems \ref{thm:aux_asypmtotic},
\ref{thm:aux_asypmtotic_2}, \ref{thm:aux_asypmtotic_2a},
\ref{thm:aux_asypmtotic_3}, \ref{thm:aux_asypmtotic_5} or
\ref{thm:aux_asypmtotic_6} as appropriate (each one with the
pre-exponential function $f(z)=z^{K_m}$), we get
\begin{equation}\label{eq:prod-h/H}
\frac{h_{N-K_m}(N)}{H_N}\to r_{*}^{K_m},\qquad N\to\infty.
\end{equation}
Combining \eqref{eq:prod-C} and \eqref{eq:prod-h/H} we obtain
formula \eqref{eq:factorial_moments_Cm}, which also entails the
asymptotic independence. Finally, the convergence
\eqref{eq:limits_Cm} follows from \eqref{eq:factorial_moments_Cm} by
the method of moments.

Similarly, for $\kappa_j=0$, \,$\theta_j>0$ we have
\begin{align*}
\Ens\myn\bigl[(C_j)_n\bigr] \to \biggl(\frac{\theta_j \mypp
r_{*}^j}{j}\biggr)^{\myn n},\qquad N\to\infty,
\end{align*}
where the limit is the $n$-th factorial moment of the corresponding
Poisson distribution.
\end{proof}

\begin{remark}
If $\kappa_j=0$ for \emph{all}  $j\in\NN$, then $r_{*}$ is
ill-defined (see \eqref{eq:r*}) and our result does not apply. In
this case, it has been proved by Nikeghbali and Zeindler
\cite[Corollary~3.2]{NiZe13} that, under suitable analytic
conditions on the generating function $g_\theta(z)$ in the spirit of
those used in Section \ref{sec:asymptotic_theorem}, the cycle counts
$C_j$ converge to mutually independent Poisson random variables with
parameter $\theta_j/j$, respectively. The latter is easy to see in a
simple particular case with $\theta_j=\theta^*=\const$ (the Ewens
model), where by the asymptotics \eqref{eq:h_theta} and Lemma
\ref{lem:generatin_moments_factorial_Cm}
\begin{equation*}
\Ens\!\mynn\left(\mynn\prod_{j=1}^m(C_j)_{n_j}\right)=
\frac{h_{N-K_m}(N)}{H_N}\prod_{j=1}^m\left(
\frac{\theta_j}{j}\right)^{\myn n_j} \to \prod_{j=1}^m\left(
\frac{\theta^*}{j}\right)^{\myn n_j}\!,\qquad N\to\infty.
\end{equation*}
See also Ercolani and Ueltschi \cite[Theorem~6.1]{ErUe12}, where the
asymptotically Ewens case $\lim_{j\to\infty}\theta_j=\theta^*\myn>0$
has been studied.
\end{remark}

\begin{remark} Formally, the limiting result
\eqref{eq:limits_Cm} suggests that the total proportion of points
contained in \emph{finite} cycles, i.e., $N^{-1}\sum_{j=1}^\infty
j\mypp C_j$, is asymptotically given by (see~\eqref{eq:r*1})
$$
\sum_{j=1}^\infty \kappa_j\mypp
r_{*}^j=g_\kappa^{\{1\}}\myn(r_{*})\,\left\{\begin{array}{ll}
=1&\text{if} \ \ \ g_\kappa^{\{1\}}\myn(R)\ge1,\\[.3pc]
<1&\text{if} \ \ \ g_\kappa^{\{1\}}\myn(R)<1,\end{array}\right.
$$
which indicates the emergence of an ``infinite'' cycle in the
supercritical case (i.e., $g_\kappa^{\{1\}}\myn(R)<1$, see
Definition~\ref{def:sub-sup}). This observation is elaborated below
(see Theorems \ref{thm:fractional_infinite}
and~\ref{thm:fractional_infinite_random}).
\end{remark}

\subsection{Fraction of points in long cycles}
\label{sec:cyclefraction}

By analogy with the spatial case (see \eqref{eq:nu}), let us define
the similar quantities in the surrogate-spatial model to capture the
expected fraction of points in long cycles,
\begin{equation}\label{eq:def_v_near_spatial}
\tilde{\nu}_K:= \liminf_{N \to \infty} \frac{1}{N}
\:\Ens\!\mynn\left(\myn\sum_{j>K} j\mypp C_j\right),\qquad
\tilde{\nu} := \lim_{K\to \infty} \tilde{\nu}_K.
\end{equation}

\begin{theorem}\label{thm:fractional_infinite}
The quantity $\tilde{\nu}_K$ \textup{($K\in\NN$)} defined in
\eqref{eq:def_v_near_spatial} exists as a limit and is explicitly
given by
\begin{equation}\label{eq:tilde-nu-M}
\tilde{\nu}_K= \lim_{N \to \infty} \frac{1}{N}
\:\Ens\!\mynn\left(\myn\sum_{j>K} j\mypp C_j\right)=1 -
\sum_{j=1}^{K} \kappa_j\mypp r_{*}^j\,,
\end{equation}
where $r_{*}$ is defined in~\eqref{eq:r*}. Moreover,
\begin{equation}\label{eq:tilde-nu}
\tilde\nu=\lim_{K\to\infty} \tilde{\nu}_K=\left\{\begin{array}{cl}
0,&\quad g_\kappa^{\{1\}}\myn(R\myp)\ge 1,\\[.3pc]
1- g_\kappa^{\{1\}}\myn(R\myp)> 0,&\quad
g_\kappa^{\{1\}}\myn(R\myp)<1.
\end{array}
\right.
\end{equation}
\end{theorem}

\begin{proof} Noting that $\sum_{j=1}^\infty j\mypp C_j=N$ and applying
Theorem~\ref{thm:cycle_counts_normalized}\myp(a), we get
\begin{equation*}
\tilde{\nu}_K= 1 - \lim_{N\to\infty}\frac{1}{N}\:
\Ens\!\mynn\left(\myn\sum_{j\le K} j\mypp C_j \right) = 1 -
\sum_{j=1}^{K} \kappa_j\mypp r_{*}^j,
\end{equation*}
which proves \eqref{eq:tilde-nu-M}. Hence, using \eqref{eq:r*1} we
obtain
\begin{equation*}
\tilde{\nu} =\lim_{K\to\infty}\tilde{\nu}_K= 1 - \sum_{j=1}^\infty
\kappa_j\mypp r_{*}^j = 1-g_\kappa^{\{1\}}\myn(r_{*}),
\end{equation*}
which is reduced to \eqref{eq:tilde-nu} thanks to~\eqref{eq:r*1}.
\end{proof}

Theorem \ref{thm:fractional_infinite} can be complemented by a
similar statement about the convergence of the (random) proportion
of points in long cycles, rather than its expected value.

\begin{theorem}\label{thm:fractional_infinite_random}
Under the sequence of probability measures $\Pns$, for any finite
$K\in\NN$ there is the convergence in probability
\begin{equation}\label{eq:tilde-nu-M-random}
\frac{1}{N}\sum_{j>K} j\mypp
C_j\stackrel{p}{\longrightarrow}\tilde{\nu}_K\myp,\qquad N\to\infty,
\end{equation}
where $\tilde{\nu}_K$ is identified in~\eqref{eq:tilde-nu-M}.
\end{theorem}

\begin{proof} Similarly to the proof of
Theorem~\ref{thm:fractional_infinite}, by the probability
convergence part of Theorem~\ref{thm:cycle_counts_normalized}\myp(a)
and according to~\eqref{eq:tilde-nu-M} we have
$$
\frac{1}{N}\sum_{j>K} j\mypp C_j=1-\frac{1}{N}\sum_{j=1}^{K} j\mypp
C_j\stackrel{p}{\longrightarrow} 1-\sum_{j=1}^{K} \kappa_j\mypp
r_{*}^j=\tilde{\nu}_K\myp,
$$
and the limit \eqref{eq:tilde-nu-M-random} is proved.
\end{proof}

\subsection{Total number of cycles}\label{sec:cycletotal}
The next result is a series of weak limit theorems (in the
subcritical, supercritical and critical cases, respectively) for
fluctuations of the total number of cycles $T_N$
(see~\eqref{eq:T_N}). As stipulated at the beginning of Section
\ref{sec:cycles}, we work under the conditions of suitable
asymptotic theorems from Section \ref{sec:asymptotic_theorem}, which
will be applied without the pre-exponential factor (i.e., with
$f(z)\equiv 1$). Note that in all but one case the limiting
distribution is normal, whereas in the critical case with
$\theta^*\myn>0$ (part (c-iii)) the answer is more complicated (and
more interesting).

In what follows, the notation $\stackrel{d\,}{\rightarrow}$
indicates convergence in distribution (with respect to the sequence
of measures $\Pns$), and $\mathcal{N}(0,1)$ denotes the standard
normal law (i.e., with mean $0$ and variance~$1$).

\begin{theorem}\label{thm:asymptotic_Kon}
\textup{(a)} \myp{}Let\/ $1<g_\kappa^{\{1\}}\myn(R\myp)\le\infty$.
Then, under the conditions of\/
Theorem~\textup{\ref{thm:aux_asypmtotic}},
\begin{equation}\label{eq:subCLTforT}
 \frac{T_N - N g_\kappa(r_1)}  {\sqrt{N \bigl[g_\kappa(r_1) -1/b_2(r_1)\bigr]}}
\stackrel{d}{\longrightarrow} \mathcal{N}(0,1),\qquad N\to\infty,
\end{equation}
where $r_1$ is the root of \eqref{eq:r_v} with $v=1$, $b_2(r)$ is
defined in \eqref{eq:b12}, and $g_\kappa(r_1) -1/b_2(r_1)>0$.

\smallskip \textup{(b)} \myp{}If\/
$g_\kappa^{\{1\}}\myn(R\myp)<1$ then, under the conditions of either
of\/ Theorems~\textup{\ref{thm:aux_asypmtotic_2}}
or~\textup{\ref{thm:aux_asypmtotic_2a}},
\begin{equation}\label{eq:supCLTforT}
\frac{T_N - N g_\kappa(R\myp)}{\sqrt{N g_\kappa(R\myp)}}
\stackrel{d}{\longrightarrow} \mathcal{N}(0,1),\qquad N\to\infty.
\end{equation}

\smallskip
\textup{(c)} \myp{}Let\/ $g_\kappa^{\{1\}}\myn(R\myp)=1$.

\smallskip
\textup{(c-i)} \myp{}Under the conditions of\/
Theorem~\textup{\ref{thm:aux_asypmtotic_3}},
\begin{equation*}
 \frac{T_N - N g_\kappa(R\myp)} {\sqrt{N \bigl[g_\kappa(R\myp) -1/b_2(R\myp)\bigr]}}
\stackrel{d}{\longrightarrow} \mathcal{N}(0,1),\qquad N\to\infty,
\end{equation*}
where $g_\kappa(R\myp) -1/b_2(R\myp)>0$.

\smallskip
\textup{(c-ii)} \myp{}Under the conditions of\/
Theorem~\textup{\ref{thm:aux_asypmtotic_6}},
\begin{equation}\label{eq:crCLTforT2}
\frac{T_N - N g_\kappa(R\myp)}{\sqrt{N g_\kappa(R\myp)}}
\stackrel{d}{\longrightarrow} \mathcal{N}(0,1),\qquad N\to\infty.
\end{equation}

\smallskip \textup{(c-iii)} \myp{}Under the conditions of\/
Theorem~\textup{\ref{thm:aux_asypmtotic_5}} with $\theta^*\myn>0$,
\begin{equation}\label{eq:crCLTforT3}
 \frac{T_N - N g_\kappa(R\myp)}{\sqrt{N \bigl[g_\kappa(R\myp) -1/b_2(R\myp)\bigr]}}
\stackrel{d}{\longrightarrow} Z -\sqrt{\frac{2
X}{g_\kappa(R\myp)\,b_2(R\myp)-1}}\,,\qquad N\to\infty,
\end{equation}
where $Z$ is a standard normal random variable and $X$ is an
independent random variable with gamma distribution
$\Gammad(\theta^*\myn/2)$.
\end{theorem}

\begin{proof} (a) Using Lemma~\ref{lem:generating_K0n} and applying
Theorem~\ref{thm:aux_asypmtotic} (see \eqref{eq:th3.2}) we obtain as
$N\to\infty$, uniformly in $v$ in a neighbourhood of $v_0=1$,
\begin{equation}\label{eq:exp/exp}
\Ens(v^{\myp T_N}) = \frac{[z^N]\exp\myn\{v\myp \GN(z)\}}
{[z^N]\exp\myn\{\GN(z)\}}\sim
\frac{\varPhi(v)\exp\myn\{N\myp\varPsi(v)\}}{\varPhi(1)
\exp\myn\{N\myp\varPsi(1)\}}\mypp,
\end{equation}
where we set for short
\begin{equation}\label{eq:FG}
\varPhi(v):=\frac{\exp\myn\{v\myp g_\theta(r_v)\}}{\sqrt{v\myp
b_2(r_v)}}\mypp,\qquad \varPsi(v):=v\myp g_\kappa(r_v) - \log r_v.
\end{equation}
Using the definition of $r_v$ (see \eqref{eq:r_v}), from
\eqref{eq:FG} we find
\begin{align}
\notag \varPsi'(v)&= g_\kappa(r_v) +v\myp g'_\kappa(r_v)\mypp
r'_v-\frac{r'_v}{r_v}\\
\label{eq:d/dv} & = g_\kappa(r_v)+\frac{r'_v}{r_v}\,\bigl(v\myp
g_\kappa^{\{1\}}\myn(r_v)-1\bigr)=g_\kappa(r_v).
\end{align}
Differentiating \eqref{eq:d/dv} once more gives
\begin{equation}\label{eq:d2/dv2}
\varPsi''(v) = g'_\kappa(r_v) \mypp r'_v
=g_\kappa^{\{1\}}\myn(r_v)\,\frac{r'_v}{r_v}=\frac{r'_v}{v\myp
r_v}\mypp.
\end{equation}
On the other hand, differentiating the identity $r_v\myp
g'_\kappa(r_v)=v^{-1}$ (see equation \eqref{eq:r_v}), we get
\[
r'_v\mypp g'_\kappa(r_v)+r_v\mypp g''_\kappa(r_v)\mypp r'_v=-v^{-2},
\]
which yields, on account of \eqref{eq:mod_d} and \eqref{eq:b12},
\begin{equation}\label{eq:r'/r}
\frac{r_v'}{r_v}=-\frac{1}{v^{2}\myp b_2(r_v)}\mypp.
\end{equation}

Hence, using \eqref{eq:d/dv}, \eqref{eq:d2/dv2} and \eqref{eq:r'/r}
we obtain the expansion of $\varPsi(v)$ around $v_0=1$,
\begin{equation}\label{eq:G(u)near1}
\varPsi(v)=\varPsi(1)+ g_\kappa(r_1)(v-1)-\frac{1}{2\myp
b_2(r_1)}\,(v-1)^2+o\bigl((v-1)^2\bigr).
\end{equation}
Substituting $v=\rme^{\myp u/\sqrt{N}}=1+u\myp N^{-1/2}+\frac12\myp
u^2 N^{-1}+O(N^{-3/2})$ into \eqref{eq:G(u)near1} gives
\begin{equation*}
N\myp \varPsi\bigl(\rme^{\myp u/\sqrt{N}}\bigr)=N\mypp \varPsi(1) +
\sqrt{N}\myp u\mypp g_\kappa(r_1) + \tfrac12\myp
u^2\bigl(g_\kappa(r_1) - 1/b_2(r_v)\bigr) + O(N^{-1/2}).
\end{equation*}
Besides, for the function $\varPhi(\cdot)$ from \eqref{eq:FG} we
have, for any $u$,
\[
\varPhi\bigl(\rme^{\myp u/\sqrt{N}}\bigr)\to \varPhi(1),\qquad
N\to\infty.
\]
Therefore, returning to \eqref{eq:exp/exp} we get, as $N\to\infty$,
\begin{align*}
\Ens\myn\bigl(\rme^{\myp u\myp T_N/\sqrt{N}\myp}\bigr)
&\sim \exp\myn\bigl\{N\bigl[\varPsi(\rme^{\myp u/\sqrt{N}}\myp)-\varPsi(1)\bigr]\bigr\}\\
&\sim \exp\myn\bigl\{\sqrt{N}\myp u\mypp g_\kappa(r_1) +
\tfrac12\myp u^2\bigl(g_\kappa(r_1) -1/b_2(r_1)\bigr)\bigr\}.
\end{align*}
The statement \eqref{eq:subCLTforT} now follows by a standard
convergence theorem based on convergence of moment generating
functions and the well-known fact that if a random variable $Z$ is
standard normal then its moment generating function is given by $\EE
\myn\bigl[\exp\myn\{u\myp Z\}\bigr]=\exp\myn\{u^2/2\}$.

It remains to check that the limit variance $g_\kappa(r_1)
-1/b_2(r_1)$ is positive. But this follows by Lemma \ref{lm:g<g},
yielding
$$
g_\kappa(r_1)\mypp
b_2(r_1)>\bigl(g_\kappa^{\{1\}}\myn(r_1)\bigr)^2\equiv 1,
$$
according to the definition of $r_1$ as the root of the equation
\eqref{eq:r_v} with $v=1$.

\smallskip
(b) \,If $\theta^*\myn>0$ then, similarly as in \eqref{eq:exp/exp},
\eqref{eq:FG}, we use Lemma~\ref{lem:generating_K0n} and
Theorem~\ref{thm:aux_asypmtotic_2} to obtain as $N\to\infty$,
uniformly in $v$ in a neighbourhood of\/ $v_0=1$,
\begin{equation*}
\Ens(v^{\myp T_N})\sim N^{\theta^*\!(v-1)} \exp\myn\bigl\{N(v -1)\,
g_\kappa(R\myp) \bigr\}\, \frac{\Gamma(\theta^*)(1-v\myp
g_\kappa^{\{1\}}\myn(R\myp))^{v\theta^*\myn-1}}{\Gamma(v\myp\theta^*)
(1-g_\kappa^{\{1\}}\myn(R\myp))^{\theta^*\myn-1}}\mypp.
\end{equation*}
Again substituting $v=\rme^{\myp u/\sqrt{N}}$, the rest of the proof
proceeds similarly as in part~(a), giving
\begin{align}
\notag \Ens\myn\bigl(\rme^{\myp u\myp T_N/\sqrt{N}\myp}\bigr) &\sim
\exp\myn\bigl\{N\bigl(\rme^{\myp u/\sqrt{N}}-1\bigr)\mypp
g_\kappa(R\myp)\bigr\}\\
\label{eq:superCLT1} &\sim \exp\myn\bigl\{\sqrt{N}\myp u\mypp
g_\kappa(R\myp) + \tfrac12\myp u^2 g_\kappa(R\myp)\bigr\},
\end{align}
and \eqref{eq:supCLTforT} follows.

Likewise, if $\theta^*\myn=0$ then by
Theorem~\ref{thm:aux_asypmtotic_2a} we have from
\eqref{eq:th3.3=0(i)} and~\eqref{eq:HN-2=0}
\begin{equation}\label{eq:CLT_total_cycles_super-}
\Ens(v^{\myp T_N})\sim v \exp\myn\bigl\{(v -1)\,\GN(R\myp)
\bigr\}\left(\frac{1-v\myp
g_\kappa^{\{1\}}\myn(R\myp)}{1-g_\kappa^{\{1\}}\myn(R\myp)}\right)^{-s-1},\qquad
N\to\infty,
\end{equation}
and the substitution $v=\exp\myn\bigl\{u/\sqrt{N}\mypp\bigr\}$ in
\eqref{eq:CLT_total_cycles_super-} yields the same
asymptotics~\eqref{eq:superCLT1}.

\smallskip
(c-i) From \eqref{eq:=1,sub0}, similarly as in part (a), we obtain
the asymptotic relation \eqref{eq:exp/exp} that holds uniformly in a
right neighbourhood of $v_0=1$. The rest of the proof is an exact
copy of that in part (a), except that we must use the substitution
$v=\rme^{\myp u/\sqrt{N}}$ with $u\ge0$.

\smallskip
(c-ii) We use  Lemma~\ref{lem:generating_K0n} (with
$v=\rme^{-u/\sqrt{N}}$, \,$u\ge0$) and the dynamic asymptotics
\eqref{eq:th3.3-dyn-s} to obtain
\begin{equation*}
\Ens\myn\bigl(\rme^{-u\myp T_N/\sqrt{N}\myp}\bigr) \sim
\exp\myn\bigl\{-\sqrt{N}\myp u\mypp g_\kappa(R\myp) + \tfrac12\myp
u^2 g_\kappa(R\myp)\bigr\},
\end{equation*}
which immediately implies~\eqref{eq:crCLTforT2}.

\smallskip
(c-iii) Again applying Lemma~\ref{lem:generating_K0n} with
$v=\rme^{-u/\sqrt{N}}$ and $u\ge0$, from the asymptotic relation
\eqref{eq:th3.3-dyn} and with the help of Lemma~\ref{lm:Gamma} we
get
\begin{align*}
\Ens\myn\bigl(\rme^{-u\myp T_N/\sqrt{N}\myp}\bigr) &\sim
\exp\myn\bigl\{-\sqrt{N}\myp u\mypp g_\kappa(R\myp) + \tfrac12\myp
u^2\bigl(g_\kappa(R\myp)
-1/b_2(R\myp)\bigr)\bigr\}\\
&\qquad \times
\EE\myp\bigl[\exp\myn\bigl\{-u\sqrt{2/b_2(R\myp)}\,(-\sqrt{X}\mypp)\bigr\}\bigr],
\end{align*}
where $X$ has the distribution $\Gammad(\theta^*\mynn/2)$. Hence,
the result \eqref{eq:crCLTforT3} follows.
\end{proof}

\begin{remark}
To summarize the content of Theorem~\ref{thm:asymptotic_Kon}, the
sequence of ``phase transitions'' manifested by the limit
distribution of $T_N^*:=N^{-1/2}\bigl(T_N-Ng_\kappa(R\myp)\bigr)$ is
as follows. In the subcritical domain
($g_\kappa^{\{1\}}\myn(R\myp)>1$), $T_N^*$ is asymptotically normal
with asymptotic variance $g_\kappa(r_1) -1/b_2(r_1)$
(Theorem~\ref{thm:asymptotic_Kon}\myp(a)). This is consistent with
the critical case ($g_\kappa^{\{1\}}\myn(R\myp)=1$) with finite
$g_\theta(R\myp)$ and $g_\kappa^{\{2\}}\myn(R\myp)$, whereby $r_1=R$
(Theorem~\ref{thm:asymptotic_Kon}\myp(c-i)). Quite surprisingly, if
$g_\theta(z)$ acquires logarithmic singularity with exponent
$\theta^*\myn>0$, the central limit theorem breaks down as the
existing normal component of the limit is reduced by the square root
of an independent gamma-distributed random variable
(Theorem~\ref{thm:asymptotic_Kon}\myp(c-iii)). In particular, the
limit distribution of $T_N^*$ gets negatively skewed; for instance,
its expected value equals
$-\Gamma((\theta^*\myn+1)/2)/\Gamma(\theta^*\myn/2)<0$ whilst
$\EE\myp(T_N^*)\equiv0$. Thus, although we prove weak convergence
using the moment generating functions, the limit in this case cannot
be established by the plain method of moments. However, when the
function $g_\kappa(z)$ becomes more singular at $z=R$\myp, with
$g_\kappa^{\{2\}}\myn(R\myp)=\infty$, the limit of $T_N^*$ reverts
to a normal distribution but with a bigger variance,
$g_\kappa(R\myp)$ (Theorem~\ref{thm:asymptotic_Kon}\myp(c-ii)),
which continues to hold in the supercritical regime
$g_\kappa^{\{1\}}\myn(R\myp)<1$
(Theorem~\ref{thm:asymptotic_Kon}\myp(b)).
\end{remark}

Despite a complicated structure of the weak limit, the total number
of\/ cycles $T_N$ in all cases satisfies the following simple law of
large numbers, as $N\to\infty$.
\begin{corollary}\label{cor:T-LLN-sub}
Under the sequence of probability measures $\Pns$, in each of the
cases in Theorem \textup{\ref{thm:asymptotic_Kon}} there is the
convergence in probability
\begin{equation}\label{eq:T-LLN-sub}
\frac{T_N}{N}\stackrel{p}{\longrightarrow} g_\kappa(r_{*}),\qquad
N\to\infty,
\end{equation}
where $r_{*}$\/ is given by\/ \eqref{eq:r*}.
\end{corollary}

\begin{remark}
The result \eqref{eq:T-LLN-sub} for $T_N=\sum_j C_j$ is formally
consistent with Theorem \ref{thm:cycle_counts_normalized}(a), since
the limits \eqref{eq:limits_Cm} of the normalized cycle counts
$C_j/N$ sum up exactly to the right-hand side of
\eqref{eq:T-LLN-sub} (cf.\ \eqref{eq:def_g_theta_and_p_tau}):
$$
\sum_{j}\frac{\kappa_j\mypp r_{*}^j}{j}=g_\kappa(r_{*}).
$$
\end{remark}

\begin{remark}
The uniform asymptotic formula \eqref{eq:exp/exp} can also be used
to derive large deviation results for $T_N$ (cf.\
\cite[\S\myp4]{MaNiZe12}), but we do not enter into further details
here.
\end{remark}


\subsection{Lexicographic ordering of cycles}
\label{sec:unordered_cycles_sub}

We can now find the asymptotic (finite-dimensional) distribution of
the cycle lengths $L_j$ (see Definition~\ref{def:L-lex}). Note that
no normalization is needed.

\begin{theorem}\label{thm:unordered_sub}
For each $m\in\NN$\myp, the random variables $L_1,\dots,L_m$ are
asymptotically independent as $N\to\infty$ and, moreover, for any
$\ell_1,\dots,\ell_m\in\NN$
\begin{equation}\label{eq:Li}
\lim_{N\to\infty} \Pns\{L_1= \ell_1,\dots,L_m= \ell_m\} =
\prod_{j=1}^m \kappa_{\ell_j}\mypp r_{*}^{\ell_j},
\end{equation}
where $r_{*}$ is defined in~\eqref{eq:r*}.
\end{theorem}

\begin{proof} By Lemma \ref{lem:distribution_of_ell_1,ell_2} we have, as
$N\to\infty$,
\begin{equation*}
\Pns\{L_1= \ell_1,\dots,L_m= \ell_m\}\sim \prod_{j=1}^m
\kappa_j\cdot \frac{h_{N-\ell_1-\dots-\ell_m}(N)}{H_N}\mypp.
\end{equation*}
Similarly to \eqref{eq:N-Km}, with $K_m:=\ell_1+\cdots+\ell_m$ we
have
\begin{equation*}
h_{N-K_m}(N)=[z^{N}]\bigl[z^{K_m}\myp\rme^{\myp\GN(z)}\bigr].
\end{equation*}
Hence, by a suitable asymptotic theorem from Section
\ref{sec:asymptotic_theorem} (i.e., one of
Theorems~\ref{thm:aux_asypmtotic}, \ref{thm:aux_asypmtotic_2},
\ref{thm:aux_asypmtotic_2a}, \ref{thm:aux_asypmtotic_3},
\ref{thm:aux_asypmtotic_4} or \ref{thm:aux_asypmtotic_5} as
appropriate, each with $f(z)=z^{K_m}$), we get
\[
\frac{h_{N-K_m}(N)}{H_N} \sim r_{*}^{\myp K_m},\qquad N\to\infty.
\]
Substituting this into \eqref{eq:N-Km} we obtain the
limit~\eqref{eq:Li}.
\end{proof}

Note from \eqref{eq:Li} that the limiting distribution of each $L_j$
has the probability generating function
\begin{equation}\label{eq:pgf_L}
\lim_{N\to\infty}\Ens(v^{L_j})=\sum_{\ell=1}^\infty \kappa_\ell\,
r_{*}^{\ell}\mypp v^\ell=g_\kappa^{\{1\}}\myn(r_{*}\myp v),\qquad
0<v\le 1.
\end{equation}
The right-hand side of \eqref{eq:pgf_L} defines a proper probability
distribution if $g_\kappa^{\{1\}}\myn(R\myp)\ge 1$ (i.e., in the
subcritical and critical cases), because then its total mass is
given by $g_\kappa^{\{1\}}\myn(r_{*}\myp
v)|_{v=1}=g_\kappa^{\{1\}}\myn(r_{*})=1$ (see~\eqref{eq:r*1}). But
in the supercritical regime this distribution is \emph{deficient},
since $g_\kappa^{\{1\}}\myn(r_{*})=g_\kappa^{\{1\}}\myn(R\myp)<1$.
Clearly, the reason for this is the emergence of infinite cycles as
$N\to\infty$; see Theorems \ref{thm:fractional_infinite} and
\ref{thm:fractional_infinite_random}, where the defect
$\tilde{\nu}=1-g_\kappa^{\{1\}}\myn(R\myp)>0$ is identified
precisely as the limiting fraction of points contained in infinite
cycles.

However, conditioning on $\{L_j\le K\}$ and passing to the limit as
$K\to\infty$ gives
\begin{align*}
\lim_{K\to\infty} \lim_{N\to\infty} &\Pns\{L_j= \ell\,|\,L_j\le
K\}=\lim_{K\to\infty} \lim_{N\to\infty}\frac{\Pns\{L_j=
\ell\}}{\Pns\{L_j\le K\}}=\frac{\kappa_{\ell}
R^{\ell}}{g_\kappa^{\{1\}}\myn(R\myp)},
\end{align*}
which defines a proper distribution.

\section{Long Cycles}\label{sec:long}

The ultimate goal of Section \ref{sec:long} is to characterize the
supercritical asymptotics of longest cycles (i.e., containing the
fraction of points $\tilde{\nu}=1-g_\kappa^{\{1\}}\myn(R\myp)>0$,
see Section~\ref{sec:cyclefraction}). Following the classical
approach of Kingman \cite{Kingman} and Vershik and Shmidt
\cite{VeSh77}, we study first the asymptotic extreme value
statistics of cycles under \emph{lexicographic ordering} (see
Definition~\ref{def:L-lex}) and then deduce from this the limit
distribution for cycles arranged in the decreasing order of their
length.

More specifically, we will show (Theorem~\ref{thm:large_cycles_2})
that if the generating function $g_\theta(z)$ has a non-vanishing
logarithmic singularity at $z=R$ \,(i.e., with $\theta^*\myn>0$ in
formula \eqref{eq:g_theta_as}), then the lengths $(L_j)$ of
lexicographically ordered cycles normalized by $N\tilde{\nu}$,
converge (in the sense of finite-dimensional weak convergence) to
the \emph{modified GEM distribution} with parameter $\theta^*\myn$
(cf.\ \cite[p.\:107]{ABT03}). The latter is constructed via the
usual stick-breaking process but applied only to the breakable part
$[0,\tilde{\nu}\myp]$ of the unit stick $[0,1]$, whereas the
remaining part $[\tilde{\nu},1]$ causes delays that contribute atoms
at zero to the distribution of the output random sequence (see more
details in Section~\ref{sec:stick}). Using the well-known link
between the descending extreme values of the $\GEM$ and the
Poisson--Dirichlet distribution (see \cite[\S\myp5.7]{ABT03}), and
noting that delays in the stick-breaking process do not affect the
upper order statistics of cycle lengths,
Theorem~\ref{thm:large_cycles_2} will readily imply that the ordered
cycle lengths $L^{(1)}\ge L^{(2)}\ge \cdots$ weakly converge
to the corresponding Poisson--Dirichlet distribution
$\PD(\theta^*\myn)$ (Theorem~\ref{thm:large_cycles_1}).

In the case $\theta^*\myn=0$, a similar argumentation works as well
(see Theorem~\ref{th:law_ell_j_theta0}) but here the stick-breaking
process of Section~\ref{sec:stick} is reduced to removal of the
entire breakable part $[0,\tilde{\nu}\myp]$ at first success in a
chain of Bernoulli trials with success probability $\tilde{\nu}$.
Translated back to the language of descending order statistics
$(L^{(j)})$, this means that there is a \emph{single giant cycle},
with length $N\tilde{\nu}\mypp(1+o(1))$, emerging in the limit
$N\to\infty$ (Theorem~\ref{thm:large_cycles_0}).

\begin{remark}
Let us point out that although the supercritical regime is defined
in terms of the generating function $g_\kappa(z)$ of the sequence
$(\kappa_j)$, the limit distribution of long cycles (and in
particular the distinction between the Poisson--Dirichlet
distribution $\PD(\theta^*\myn)$ of Theorem~\ref{thm:large_cycles_1}
and the degenerate distribution of Theorem~\ref{thm:large_cycles_0})
is determined entirely by the generating function $g_\theta(x)$ of
the sequence $(\theta_j)$.
\end{remark}

\subsection{Analytic conditions on the generating
functions}\label{sec:5.1} For the most part (until Section
\ref{sec:theta*=0}), the generating functions $g_\theta(z)$ and
$g_\kappa(z)$ will be assumed to satisfy the hypotheses of
Theorem~\ref{thm:aux_asypmtotic_2}, including the asymptotic
formulas \eqref{eq:g_theta_as} and \eqref{eq:g_kappa_as} (with some
$\theta^*\myn>0$). In particular, $g_\kappa^{\{1\}}\myn(R\myp)<1$ so
we are in the supercritical regime. Furthermore, let us assume that
there is a non-integer $s>1$ such that
$g_\kappa^{\{n\}}\myn(R\myp)<\infty$ for all $n<s$ while
$g_\kappa^{\{n\}}\myn(R\myp)=\infty$ for $n>s$\textup{;} moreover,
there exist a constant $a_s>0$ and a sequence $\delta_n>0$ such
that, as $z \to R$\myp, $z\in \varDelta_0$ (see
Definition~\ref{def:Delta0}),
\begin{alignat}{2}\label{eq:assum_derivative_p_Upsilon-low}
g_\kappa^{\{n\}}\myn(z)&=
g_\kappa^{\{n\}}\myn(R\myp)\mypp\bigl\{1+O\bigl((1-z/R\myp)^{\delta_n}\bigr)\bigr\}
&(n< s),\\[.2pc]
\label{eq:assum_derivative_p_Upsilon}
g_\kappa^{\{n\}}\myn(z)&=\frac{\Gamma(-s+n)\,a_s}{\Gamma(-s)}
\,(1-z/R\myp)^{s-n}\mypp\bigl\{1+O\bigl((1-z/R\myp)^{\delta_n}\bigr)\bigr\}
\qquad &(n> s).
\end{alignat}
In addition to \eqref{eq:g_theta_as}, we also assume that
\begin{equation}\label{eq:assum_derivative_g_Theta}
g_\theta^{\{n\}}\myn(z)= \frac{\theta^*(n-1)!}{(1-z/R\myp)^n}
\mypp\bigl\{1+O\bigl((1-z/R\myp)^{\delta_n}\bigr)\bigr\},\qquad
n\in\NN\myp.
\end{equation}
In Section~\ref{sec:theta*=0} we will also consider the case
$\theta^*\myn=0$, using the results of
Theorem~\ref{thm:aux_asypmtotic_2a}.

\begin{remark} The principal term in the asymptotic expression
\eqref{eq:assum_derivative_g_Theta} may be formally obtained by
differentiating the condition \eqref{eq:g_theta_as}. Similarly, the
asymptotic formulas \eqref{eq:assum_derivative_p_Upsilon-low} and
\eqref{eq:assum_derivative_p_Upsilon} are underpinned by the
expansion \eqref{eq:g_kappa_as=0}.
\end{remark}

\subsection{Reminder: Poisson--Dirichlet distribution}\label{sec:PD}

Our aim is to show that, under the assumptions stated in Section
\ref{sec:5.1} (most importantly, with $\theta^*\myn>0$), the
descending order statistics of the cycle lengths converge to the
so-called \emph{Poisson--Dirichlet distribution} $\PD(\theta^*)$.

Let us recall that the Poisson--Dirichlet distribution $\PD(\theta)$
with parameter $\theta>0$ was introduced by Kingman
\cite[\S\myp5]{Kingman1975} as the weak limit of order statistics of
a symmetric Dirichlet distribution (with parameter $\alpha$) on
$N$-dimensional simplex $\Delta_N$ as $N\to\infty$, \,$\alpha
N\to\theta>0$. Such a limit can be identified as the distribution
law $\mathcal{L}$ of the normalized points of a Poisson process
$\sigma_1>\sigma_2>\cdots$ on $(0,\infty)$ with rate $\theta\mypp
x^{-1}\rme^{-x}$,
$$
\PD(\theta)=\mathcal{L}(\sigma_1/\sigma,\,\sigma_2/\sigma,{}\dots\myp),
$$
where $\sigma=\sigma_1+\sigma_2+\cdots<\infty$ (a.s.); note that the
random variable $\sigma$ has the gamma distribution
$\Gammad(\theta)$ and is independent of the sequence
$(\sigma_j/\sigma)$ (see, e.g., \cite[\S\myp9.3]{Kingman-Poisson} or
\cite[\S\myp5.7]{ABT03} for more details).

An explicit formula for the finite-dimensional probability density
of $\PD(\theta)$, first obtained by Watterson
\cite[\S\myp2]{Watterson} (see also \cite[p.\:113]{ABT03}), is quite
involved but is of no particular interest to us. There is, however,
an equivalent descriptive definition of the Poisson-Dirichlet
distribution $\PD(\theta)$ through the so-called $\GEM(\theta)$
distribution (named after Griffiths, Engen and McCloskey; see, e.g.,
\cite[p.\:107]{ABT03}), which is the joint distribution of the
random variables
\begin{equation}\label{eq:GEM}
Y_1:=B_1,\qquad Y_n:=B_n\prod_{j=1}^{n}(1-B_j)\ \ \quad (n\ge2),
\end{equation}
where $(B_n)$ is a sequence of independent identically distributed
(i.i.d.)\ random variables with the beta distribution
$\Beta(1,\theta)$ (i.e., with the probability density
$\theta(1-x)^{\theta-1}$, \,$0\le x\le 1$). From the definition
\eqref{eq:GEM}, it is straightforward to obtain the
finite-dimensional probability density of $\GEM(\theta)$ (see
\cite[p.\:107]{ABT03}),
$$
f_\theta(x_1,\dots,x_m)=\frac{\theta^{\myp m}\myp
(1-x_1-\dots-x_m)^{\theta-1}}{\prod_{j=1}^{m-1}(1-x_1-\dots-x_j)}\mypp,\
\ \quad x_1,\dots,x_m\ge 0,\quad x_1+\dots+x_m<1.
$$
\begin{remark}\label{rm:stick}
The sequence \eqref{eq:GEM} has a geometric interpretation as a
\emph{stick-breaking process}, whereby at each step an i.i.d.\
$\Beta(1,\theta)$-distributed fraction is removed from the remaining
``stick'' (see, e.g., \cite[\S\myp3]{IsZa} for the background and
further references).
\end{remark}
The remarkable link of the $\GEM(\theta)$ with the $\PD(\theta)$,
discovered by Tavar\'e \cite[Theorems 4 and~6]{Tavare} (see also
\cite[\S\myp5.7]{ABT03}), is as follows.
\begin{lemma}\label{lm:GEM}
The descending order statistics $Y^{(1)}\mynn\ge Y^{(2)}\mynn\ge
\cdots$ of the $\GEM(\theta)$ sequence \eqref{eq:GEM} have precisely
the $\PD(\theta)$ distribution.
\end{lemma}
Note that since the joint distribution of $Y_n$'s is continuous, the
order statistics $Y^{(n)}$ are in fact distinct (a.s.).

\subsection{Modified stick-breaking process}\label{sec:stick}
Let $(U_n)$, $(B_n)$ be two sequences of i.i.d.\ random variables
each, also independent of one another, where $U_n$'s have uniform
distribution on $[0,1]$ and $B_n$'s have beta distribution
$\Beta(1,\theta^*)$ with parameter $\theta^*\myn>0$. Setting
$\eta_0:=1$, let us define inductively the random variables
\begin{equation}\label{eq:xi-D-eta}
\xi_{n}:=\bfOne_{\{U_n\le u^*(\eta_{n-1})\}}, \ \ \quad
D_n:=\xi_nB_n,\ \ \quad \eta_n:=\prod_{j=1}^{n}(1-D_j)\ \ \  \quad
(n\in\NN),
\end{equation}
where $u^*(x):=\tilde{\nu}x/(1 -\tilde{\nu} +\tilde{\nu}x)$
\,($x\in[0,1]$) and $\bfOne_{A}$ denotes the indicator of event $A$.
Note that $\xi_n$'s are Bernoulli random variables (with values $0$
and $1$), adapted to the filtration
$\mathcal{F}_{n}:=\sigma\{(U_j,B_j),\,1\le j\le n\}$ and with the
conditional distribution
\begin{equation}\label{eq:def_A_j}
\PP\{\xi_{n} = 1\mypp|\,\mathcal{F}_{n-1}\} =
u^*(\eta_{n-1})\equiv \frac{\tilde{\nu}\mypp \eta_{n-1}}{1
-\tilde{\nu} +\tilde{\nu}\mypp\eta_{n-1}}\qquad (n\in\NN),
\end{equation}
where $\mathcal{F}_0$ is a trivial $\sigma$-algebra. In particular,
$\PP\{\xi_1=1\}=\tilde{\nu}$, \,$\PP\{\xi_1=0\}=1-\tilde{\nu}$.

Finally, let us consider the random variables
\begin{equation}\label{eq:X}
X_n:=\eta_{n-1}D_{n}\qquad (n\in\NN).
\end{equation}
Noting that, for all $n\in\NN$\myp, we have $0<B_n<1$ (a.s.), from
\eqref{eq:xi-D-eta} it is clear that
\begin{equation}\label{eq:D<1}
\forall\myp n\in\NN\myp,\qquad 0\le D_n<1\quad\text{and}\quad
0<\eta_n\le 1\qquad (\text{a.s.)},
\end{equation}
which also implies, due to \eqref{eq:X}, that
\begin{equation*}
\forall\myp n\in\NN\myp,\qquad 0\le X_n<1\qquad (\text{a.s.)}.
\end{equation*}

\begin{remark}\label{rm:stick_mod}
The random sequence $(X_n)$ may be interpreted as a \emph{modified
stick-breaking process with delays} (cf.\ Remark~\ref{rm:stick}),
whereby the original interval $[0,1]$ (``stick'') is divided into
two parts: (i) \,$[0,\tilde{\nu}\myp]$ which is subject to a
subsequent breaking, and (ii) \,$[\tilde\nu,1]$ which stays intact.
At each step, the breaking is only enabled if an independent point
chosen at random in the remaining stick falls in the breakable part,
otherwise the process is idle; the breaking, when it occurs, acts as
the removal of a fraction of the current breakable part,
independently sampled from $\Beta(1,\theta^*)$.
\end{remark}

\begin{lemma}\label{lm:sum-Xn}
The random sequence $(X_n)_{n\in\NN}$ satisfies the a.s.-identities
\begin{equation}\label{eq:X_id}
\sum_{j=1}^n X_j=1-\eta_n \quad (n\in\NN),\qquad
\sum_{j=1}^\infty X_j=1.
\end{equation}
\end{lemma}

\begin{proof} Recalling \eqref{eq:X}, we have
\begin{equation}\label{eq:pi_n}
\sum_{j=1}^n X_j = \sum_{j=1}^n(\eta_{j-1}-\eta_j)=1-\eta_n,
\end{equation}
which proves the first formula in \eqref{eq:X_id}. To establish the
second one, in view of \eqref{eq:pi_n} we only need to check that
$\lim_{n\to\infty}\eta_n=0$ a.s. To this end, note that $0\le D_j\le
1$, so that the sequence $(\eta_n)$ is non-increasing and therefore
converges to a (possibly random) limit $\eta\ge0$. To show that
$\eta=0$ a.s., consider
\begin{equation}\label{eq:E|Fn}
\EE\mypp[\eta_{n}\myp|\,\mathcal{F}_{n-1}]=\eta_{n-1}
\bigl(1-\EE\mypp[D_{n}\myp|\,\mathcal{F}_{n-1}]\bigr).
\end{equation}
Recalling the definition of $(D_{n})$ (see \eqref{eq:xi-D-eta}) and
noting that $\xi_{n}$ and $B_{n}$ are mutually independent, with
$B_{n}$ also independent of $\mathcal{F}_{n-1}$, we obtain, using
\eqref{eq:def_A_j} and replacing $\eta_{n-1}$ in the denominator
with its upper bound $1$,
\begin{align*}
\EE\mypp[D_{n}\,|\,\mathcal{F}_{n-1}]=\EE\mypp[\xi_{n}B_{n}\mypp|\,\mathcal{F}_{n-1}]
&=\EE\mypp[\xi_{n}\mypp|\,\mathcal{F}_{n-1}]\myn\cdot
\EE\mypp[B_{n}]\\
&=\frac{\tilde{\nu}\mypp \eta_{n-1}}{1 -\tilde{\nu}+\tilde{\nu}\mypp
\eta_{n-1}}\, \EE\mypp[B_{1}]\ge\tilde{\nu}\mypp\eta_{n-1}
\mypp\EE\mypp[B_{1}].
\end{align*}
Returning to \eqref{eq:E|Fn} and taking the expectation, we obtain
\begin{align*}
\EE\mypp[\eta_{n}]&\le
\EE\mypp[\eta_{n-1}]-\tilde{\nu}\,\EE\mypp[B_{1}]\cdot\EE\mypp[\eta_{n-1}^2].
\end{align*}
Passing to the limit as $n\to\infty$ and using the monotone
convergence theorem, we deduce that $\EE\myp[\eta^2]\le 0$, hence
$\eta=0$ a.s. This completes the proof of the lemma.
\end{proof}

\begin{remark} In terms of the modified stick-breaking process (see
Remark~\ref{rm:stick}), the second equality in \eqref{eq:X_id} means
that, with probability~$1$, the total fraction removed from the
breakable part of the stick has full measure. This is a
generalization of the similar property of the standard
stick-breaking process~\eqref{eq:GEM}.
\end{remark}

\begin{lemma}\label{lm:D>0i.o.}
With probability $1$, infinitely many $X_n$'s are non-zero.
\end{lemma}

\begin{proof} By Lemma \ref{lm:sum-Xn}, $\eta_n\to0$ (a.s.) as
$n\to\infty$. In view of the last formula in \eqref{eq:xi-D-eta},
this implies that $\sum_{n}\log\myn(1-D_n)=-\infty$ (a.s.), which is
equivalent to $\sum_{n}D_n=+\infty$ (a.s.). Hence, $D_n>0$
infinitely often (a.s.), and the claim of the lemma now readily
follows from \eqref{eq:X} since $\eta_n>0$ a.s.\
(see~\eqref{eq:D<1}).

The fact that $D_n>0$ infinitely often (a.s.) can be established
more directly as follows. Put $\tau_0:=0$ and define inductively the
successive hitting times
\begin{equation}\label{eq:def_stoptime_T_j}
\tau_n:= \min\myn\{j>\tau_{n-1}\mynn:\ \xi_j = 1\},\qquad
n\in\NN\myp,
\end{equation}
with the usual convention $\min \emptyset:=+\infty$. Since $B_n>0$
(a.s.), from \eqref{eq:xi-D-eta} we see that $D_{\tau_n}\mynn>0$
(a.s.), while $D_j=0$ for $\tau_{n-1}<j<\tau_n$, so it remains to
verify that $(\tau_n)$ is an a.s.-infinite sequence. Indeed, from
the definition of $\eta_n$ it follows that
\begin{equation*}
\eta_j\equiv \eta_{\tau_{n-1}} \ \ \ (\tau_{n-1}\le j<\tau_n),\qquad
\eta_{\tau_n}< \eta_{\tau_{n-1}}\myn.
\end{equation*}
Together with formulas \eqref{eq:xi-D-eta} and \eqref{eq:def_A_j},
this implies that, conditionally on $\mathcal{F}_{\tau_{n-1}}$, the
random variable $\tau_{n}-\tau_{n-1}$ is time to first success in a
sequence of independent Bernoulli trials $\xi_{j}$ with success
probability $u^*(\eta_{\tau_{n-1}})>0$ (see~\eqref{eq:def_A_j}), so
it has the geometric distribution
\begin{align*}
\PP\{\tau_{n} -\tau_{n-1}=k \,|\,\mathcal{F}_{\tau_{n-1}}\}&=
\left(1- u^*(\eta_{\tau_{n-1}})\right)^{k-1} u^*(\eta_{\tau_{n-1}}),
\qquad k=1,2,\dots,
\end{align*}
and in particular $\tau_{n}-\tau_{n-1}<\infty$ a.s. From the product
structure of the filtration $(\mathcal{F}_n)$ it is also clear that
the waiting times $\tau_{n} -\tau_{n-1}$ are mutually independent
($n\in\NN$). Hence, it follows that all $\tau_{n}$ are a.s.-finite,
as required.
\end{proof}

Let $X^{(1)}\mynn\ge X^{(2)}\mynn\ge \cdots$ be the order statistics
built from the sequence $(X_n)$ (see~\eqref{eq:X}) by arranging the
entries in decreasing order. Recall that $X_n\ge0$; moreover, by
Lemma~\ref{lm:D>0i.o.} infinitely many entries are positive (a.s.)\
and, in addition, $\sum_n X_n=1$ according to Lemma~\ref{lm:sum-Xn}.
It follows that, with probability~$1$, there is an infinite sequence
of order statistics $X^{(n)}\mynn>0$, each well defined up to
possible ties of at most finite multiplicities (however, it will be
clear from the next lemma and the continuity of the
Poisson--Dirichlet distribution that positive order statistics
$X^{(n)}$ are in fact a.s.-distinct).

In particular, the sequence $X^{(1)}\mynn\ge
X^{(2)}\mynn\ge\cdots>0$ is not affected by any zero entries among
$X_n$'s, which therefore can be removed from $(X_n)$ prior to
ordering. But, according to the definition of the random times
$\tau_{n}$ (see~\eqref{eq:def_stoptime_T_j}) and by formulas
\eqref{eq:xi-D-eta} and \eqref{eq:X}, successive non-zero entries
among $X_1,X_2,\dots$ are precisely given by
\begin{equation}\label{eq:GEM1}
X_{\tau_1}\myn=B_{\tau_1}\myn,\qquad
X_{\tau_n}=B_{\tau_n}\prod_{j=1}^{n-1}(1-B_{\tau_j})\ \ \quad
(n\ge2),
\end{equation}
where $(B_{\tau_n})$ are i.i.d.\ random variables with beta
distribution $\Beta(1,\theta^*)$. The latter claim can be easily
verified using the total probability formula and mutual independence
of the waiting times $\tau_{n} -\tau_{n-1}$ pointed out in the
alternative proof of Lemma~\ref{lm:D>0i.o.}; for instance,
\begin{align*}
\PP\{B_{\tau_1}\mynnn>x_1,B_{\tau_2}\mynn>x_2\}&=\sum_{k_1=1}^\infty\sum_{k_2>k_1}
\PP\{\tau_1=k_1,\,\tau_{2}-\tau_{1}=k_2-k_1\}\,\PP\{B_{k_1}\mynnn>x_1,B_{k_2}\mynn>x_2\}\\
&=(1-x_1)^{\theta^*}(1-x_2)^{\theta^*}\sum_{k_1=1}^\infty
\PP\{\tau_1=k_1\}\sum_{\ell=1}^\infty\PP\{\tau_2-\tau_1=\ell\}\\
&=(1-x_1)^{\theta^*}(1-x_2)^{\theta^*}.
\end{align*}

Now, comparing \eqref{eq:GEM1} with \eqref{eq:GEM} and using
Lemma~\ref{lm:GEM}, we arrive at the following result.

\begin{lemma}\label{lm:large_cycles_1}
The sequence of\/ the descending order statistics $(X^{(n)})$ has
the Poisson--Dirichlet distribution $\PD(\theta^*)$ with parameter
$\theta^*\myn>0$.
\end{lemma}

In conclusion of this subsection, let us prove some moment
identities for the random variables $X_j$.
\begin{lemma}\label{lm:D}
For each $n\in\NN$\myp, as $N\to\infty$,
\begin{equation}\label{eq:ED1}
\EE\myp[X_1^n]=\tilde{\nu}\,\frac{B(n+1,\theta^*)}{B(1,\theta^*)}
\equiv\frac{\tilde{\nu}\,n!\,\Gamma(\theta^*\myn+1)}{\Gamma(\theta^*\myn+n+1)}\mypp.
\end{equation}
Furthermore, for all $n_1,n_2\in\NN$
\begin{align}
\notag \EE\myp\bigl[X_1^{n_1}\myn X_2^{n_2}\myp
(1-\tilde{\nu}X_1)\bigr] &=\frac{\tilde{\nu}^{\mypp 2}\myp
B(n_1+1,\theta^*\myn+n_2+1)\,B(n_2+1,\theta^*)}{(B(1,\theta^*))^2}\\
\label{eq:ED2} &\equiv \frac{\theta^*\mypp\tilde{\nu}^{\mypp2}\mypp
n_1\myn!\,n_2!\,\Gamma(\theta^*\myn+1)}{\Gamma(\theta^*\myn+n_1+n_2+2)}\mypp,
\end{align}
while for $n_1=0$ and any $n_2\in\NN$
\begin{align}
\notag \EE\myp\bigl[X_2^{n_2}\myp (1-\tilde{\nu}X_1)\bigr]
&=\left(1-\tilde{\nu}+\tilde{\nu}\,\frac{B(1,n_2+1+\theta^*)}{B(1,\theta^*)}\right)
\frac{\tilde{\nu}\mypp B(n_2+1,\theta^*)}{B(1,\theta^*)}\\
\label{eq:ED3} &\equiv \bigl\{\theta^*\myn+(n_2
+1)(1-\tilde{\nu})\bigr\}\,\frac{\tilde{\nu}\mypp
n_2!\,\Gamma(\theta^*\myn+1)}{\Gamma(\theta^*\myn+n_2+2)}\mypp.
\end{align}
\end{lemma}

\begin{proof} Using the definitions \eqref{eq:xi-D-eta} and \eqref{eq:X},
we obtain
\begin{align*}
\EE\myp[X_1^n]=\EE\myp[\xi_1B_1^n]
&=\EE\myp[\xi_1]\cdot\EE\myp[B_1^n]\\
&=\frac{\tilde{\nu}}{B(1,\theta^*)}\int_0^1
x^n(1-x)^{\theta^*\myn-1}\,\rmd{x}=\frac{\tilde{\nu}\mypp
B(n+1,\theta^*)}{B(1,\theta^*)}\mypp,
\end{align*}
which proves the first part of formula \eqref{eq:ED1}. The second
part follows by substituting the well-known representation
$B(a,b)=\Gamma(a)\myp \Gamma(b)/\Gamma(a+b)$.

To prove \eqref{eq:ED2}, let us first compute the conditional
expectation
\begin{equation}\label{eq:condED}
\EE\myp\bigl[X_1^{n_1}\myn X_2^{n_2}\myp (1-\tilde{\nu}
X_1)\mypp|\,\mathcal{F}_1\bigr]=D_1^{n_1}(1-D_1)^{n_2}\myp
(1-\tilde{\nu}
D_1)\,\EE\myp\bigl[D_2^{n_2}\myp|\,\mathcal{F}_1\bigr],
\end{equation}
where, according to \eqref{eq:xi-D-eta} and \eqref{eq:def_A_j},
\begin{align*}
\EE\myp\bigl[D_2^{n_2}\myp|\,\mathcal{F}_1\bigr]=\EE\myp\bigl[\xi_2
B_2^{n_2}\myp|\,\mathcal{F}_1\bigr]&=\frac{\tilde{\nu}\mypp
\eta_{1}}{1 -\tilde{\nu} +\tilde{\nu}\mypp\eta_{1}}\cdot
\EE\myp[B_2^{n_2}]\\
&=\frac{\tilde{\nu}\mypp (1-D_1)}{1 -\tilde{\nu}D_1}\cdot
\frac{B(n_2+1,\theta^*)}{B(1,\theta^*)}\mypp.
\end{align*}
Hence, on taking the expectation of \eqref{eq:condED}, we obtain
\begin{align}
\notag \EE\myp\bigl[X_1^{n_1}\myn X_2^{n_2}\myp
(1-\tilde{\nu}X_1)\bigr] &=\frac{\tilde{\nu}\mypp
B(n_2+1,\theta^*)}{B(1,\theta^*)}\,
\EE\myp\bigl[D_1^{n_1}(1-D_1)^{n_2+1}\bigr]\\
&=\frac{\tilde{\nu}\mypp B(n_2+1,\theta^*)}{B(1,\theta^*)}\,
\EE\myp\bigl[\xi_1^{n_1}\myn B_1^{n_1}(1-\xi_1B_1)^{n_2+1}\bigr].
\label{eq:ED2a}
\end{align}
For $n_1\ge 1$ we have $\xi_1^{n_1}=\xi_1$, leading to
\begin{align}
\notag \EE\myp\bigl[\xi_1^{n_1}\myn
B_1^{n_1}(1-\xi_1B_1)^{n_2+1}\bigr]&=\PP\{\xi_1=1\}
\,\EE\myp\bigl[B_1^{n_1}(1-B_1)^{n_2+1}\bigr]\\
&=\frac{\tilde{\nu}}{B(1,\theta^*)}\int_0^1
x^{n_1}(1-x)^{\theta^*\myn+n_2}\,\rmd{x}=\frac{\tilde{\nu}\mypp
B(n_1+1,\theta^*\myn+n_2+1)}{B(1,\theta^*)}\mypp, \label{eq:ED2b}
\end{align}
and the substitution of \eqref{eq:ED2b} into \eqref{eq:ED2a} gives
the first line of formula~\eqref{eq:ED2}. The second line can again
be obtained by expressing the beta function through the gamma
function.

For $n_1=0$, formula \eqref{eq:ED2a} is reduced to
\begin{equation}\label{eq:ED2c}
\EE\myp\bigl[X_2^{n_2}\myp (1-\tilde{\nu}
X_1)\bigr]=\frac{\tilde{\nu}\mypp
B(n_2+1,\theta^*)}{B(1,\theta^*)}\,
\EE\myp\bigl[(1-\xi_1B_1)^{n_2+1}\bigr],
\end{equation}
and similarly as before we find
\begin{align*}
\EE\myp\bigl[(1-\xi_1B_1)^{n_2+1}\bigr]&=\PP\{\xi_1=0\}
+\PP\{\xi_1=1\}\,\EE\myp\bigl[(1-B_1)^{n_2+1}\bigr]\\
&=(1-\tilde{\nu})+\tilde{\nu}\int_0^1
(1-x)^{\theta^*\myn+n_2}\,\rmd{x}\\
&=1-\tilde{\nu}+\frac{\tilde{\nu}\mypp
B(1,\theta^*\myn+n_2+1)}{B(1,\theta^*)}\mypp.
\end{align*}
Together with \eqref{eq:ED2c} this proves the first line
of~\eqref{eq:ED3}, while the second line then follows by usual
manipulations with the beta function.
\end{proof}

\subsection{Cycles under lexicographic ordering}
The next theorem characterizes the asymptotic (finite-dimensional)
distributions of normalized cycle lengths $(L_j)$ under the
lexicographic ordering introduced in Definition~\ref{def:L-lex}.
Owing to the normalization proportional to $N$, only long cycles
(i.e., of length comparable to $N$) survive in the limit as
$N\to\infty$. This result should be contrasted with
Theorem~\ref{thm:unordered_sub} that deals with non-normalized cycle
lengths, thus revealing an asymptotic loss of mass in the
supercritical regime due to the emergence of long cycles (see a
comment after the proof of Theorem~\ref{thm:unordered_sub}).
\begin{theorem}\label{thm:large_cycles_2}
For each $m\in\NN$\myp,
\begin{equation}\label{eq:weak_conv_ell_j}
\frac{1}{N\tilde{\nu}}\,(L_1,\dots,L_m)
\stackrel{d}{\longrightarrow} (X_1,\dots,X_m),\qquad N\to\infty,
\end{equation}
where $\tilde{\nu}>0$ is given by \eqref{eq:tilde-nu} and the random
variables $X_j$ are as defined in~\eqref{eq:X}. In particular,
$L_1/(N\tilde{\nu})$ converges in distribution to a random variable
$X_1$ with atom $1-\tilde{\nu}$ at zero and an absolutely continuous
component on $(0,1)$ with density \,$\tilde{\nu}
\mypp\theta^*(1-x)^{\theta^*\myn-1}$.
\end{theorem}

For the proof of the theorem, we first need to establish the
following lemma.

\begin{lemma}\label{lem:ell_1_ell_2_super}
For each $n\in\NN$\myp,
\begin{equation}\label{eq:moments_ell_1}
\lim_{N\to\infty}\frac{1}{(N\tilde{\nu})^{n}}\, \Ens(L_1^n) =
\frac{\tilde{\nu}\, n!\:
\Gamma(\theta^*\myn+1)}{\Gamma(\theta^*\myn+n+1)}\mypp.
\end{equation}
Furthermore, for any $n_1,n_2\in\NN$
\begin{equation}\label{eq:moments_ell_1(n-ell_1)l_2}
\lim_{N\to\infty} \frac{1}{(N\tilde{\nu})^{n_1 +
n_2}}\,\Ens\bigl[L_1^{n_1} L_2^{n_2}\myp (1-L_1/N)\bigr] =
 \frac{\theta^*\myp\tilde{\nu}^{\mypp2}\mypp
n_1\myn!\,n_2!\:\Gamma(\theta^*\myn+1)}{\Gamma(\theta^*\myn +n_1+n_2
+2)}\mypp,
\end{equation}
while for $n_1=0$ and any $n_2\in\NN$
\begin{equation}\label{eq:moments_(n-ell_1)l_2}
\lim_{N\to\infty}
\frac{1}{(N\tilde{\nu})^{n_2}}\,\Ens\myn\bigl[L_2^{n_2}\myp
(1-L_1/N)\bigr] = \frac{\tilde{\nu}\mypp\{\theta^*\myn+(n_2
+1)(1-\tilde{\nu})\}\,n_2!\: \Gamma(\theta^*\myn+1)}
{\Gamma(\theta^*\myn + n_2 +2)}\mypp.
\end{equation}
\end{lemma}

\smallskip
\begin{remark}
The subtlety of Lemma \ref{lem:ell_1_ell_2_super} is hidden in the
fact that the asymptotics \eqref{eq:moments_ell_1},
\eqref{eq:moments_ell_1(n-ell_1)l_2},
\eqref{eq:moments_(n-ell_1)l_2} are invalidated if either of $n$,
$n_1$, $n_2$ takes the value~$0$; for instance, formula
\eqref{eq:moments_(n-ell_1)l_2} cannot be readily deduced from
\eqref{eq:moments_ell_1(n-ell_1)l_2} by setting $n_1=0$. An
explanation lies in the asymptotic separation of ``short'' and
``long'' cycles, leading to the emergence of an atom of mass
$1-\tilde{\nu}$ at zero in the limiting distribution of $(L_1/N)^n$
for $n>0$.
\end{remark}

\begin{proof}[Proof\/ of\/ Lemma~\textup{\ref{lem:ell_1_ell_2_super}}]
Let us first consider $\Ens\myn\bigl[(L_1-1)_{n}\bigr]$, where
$(\cdot)_{n}$ is the Pochhammer symbol defined
in~\eqref{eq:Pochhammer}. Using the distribution of $L_1$ obtained
in Lemma~\ref{lem:distribution_of_ell_1,ell_2} (see
\eqref{eq:dist_L1}) and recalling formulas \eqref{eq:Hn_generatingN}
and \eqref{eq:a}, we get for each $n\in\NN$
\begin{align}
\notag \Ens\myn\bigl[(L_1-1)_n\bigr] &= \frac{1}{N
H_N}\sum_{\ell=1}^\infty (\ell-1)_n\mypp (\theta_{\ell}
+N\kappa_{\ell})\mypp h_{N-\ell}(N)\\
\notag &= \frac{1}{N H_N}\sum_{\ell=1}^\infty
(\ell-1)_n\mypp(\theta_{\ell} +N\kappa_{\ell})\mypp
[z^{N-\ell\myp}]\bigl[\rme^{\myp\GN(z)}\bigr]\\
\notag &= \frac{1}{N H_N}\,[z^{N}]\left[\sum_{\ell=1}^\infty
(\ell-1)_n \mypp(\theta_{\ell} +N\kappa_{\ell})\mypp
z^\ell\,\rme^{\myp\GN(z)}\right]\\
\label{eq:moment_ell_1_with_gen}&=\frac{1}{N
H_N}\,[z^N]\bigl[\GN^{\{n+1\}}\myn(z) \,\rme^{\myp\GN(z)}\bigr].
\end{align}
In view of the formula \,$\tilde{\nu}=1-
g_\kappa^{\{1\}}\myn(R\myp)$ (see~\eqref{eq:tilde-nu}),
Theorem~\ref{thm:aux_asypmtotic_2} with function
$f(z)=g_\theta^{\{n+1\}}\myn(z)$ and $\beta=n+1$ (see
\eqref{eq:assum_derivative_g_Theta}) gives
\begin{equation}\label{eq:theta*1}
\frac{1}{N H_N}\,[z^N]\bigl[g_\theta^{\{n+1\}}\myn(z)
\,\rme^{\myp\GN(z)}\bigr]\sim (N\tilde{\nu})^{\myp n}
\frac{\theta^*n!\,\Gamma(\theta^*)}{\Gamma(\theta^*\myn+n+1)} \,\myp
\tilde{\nu}\myp.
\end{equation}
Furthermore, using the asymptotic expansions
\eqref{eq:assum_derivative_p_Upsilon-low},
\eqref{eq:assum_derivative_p_Upsilon}  and applying
Theorem~\ref{thm:aux_asypmtotic_2} with
$f(z)=g_\kappa^{\{n+1\}}\myn(z)$ and $\beta=\max\myn\{0,n+1-s\}$, we
obtain
\begin{equation}\label{eq:kappa*1<s}
\frac{1}{N H_N}\,[z^N]\bigl[Ng_\kappa^{\{n+1\}}\myn(z)
\,\rme^{\myp\GN(z)}\}\bigr]= O(1)+O(N^{n+1-s})=o\myp(N^{n}),
\end{equation}
due to the conditions $s>1$, $n>0$. Thus, substituting
\eqref{eq:theta*1} and \eqref{eq:kappa*1<s} into
\eqref{eq:moment_ell_1_with_gen} yields
\begin{equation*}
\Ens\myn\bigl[(L_1-1)_{n}\bigr] \sim (N\tilde{\nu})^n\myp\frac{n!\,
\Gamma(\theta^*\myn+1)}{\Gamma(\theta^*\myn+n+1)}\,\tilde{\nu},
\end{equation*}
which implies~\eqref{eq:moments_ell_1}.

We now turn to \eqref{eq:moments_(n-ell_1)l_2} and
\eqref{eq:moments_ell_1(n-ell_1)l_2}. Using the joint distribution
of $L_1$ and $L_2$ (see \eqref{eq:L=l}), it follows by a similar
computation as in \eqref{eq:moment_ell_1_with_gen} that, for any
integers $n_1\ge0$, $n_2>0$,
\begin{equation}\label{eq:GGexpG}
\Ens\bigl[(L_1-1)_{n_1} (L_2-1)_{n_2}\myp(1-L_1/N)\bigr] =
\frac{1}{N^2H_N}\,
[z^N]\bigl[\GN^{\{n_1+1\}}\myn(z)\,\GN^{\{n_2+1\}}\myn(z)
\,\rme^{\myp\GN(z)}\bigr].
\end{equation}
As before, we can work out the asymptotics of \eqref{eq:GGexpG} by
using Theorem~\ref{thm:aux_asypmtotic_2} with the function
\begin{align}
\notag
f(z)&=\GN^{\{n_1+1\}}\myn(z)\,\GN^{\{n_2+1\}}\myn(z)\\
\label{eq:Gexpansion} &\begin{aligned} & \hspace{-.16pc}
=g_\theta^{\{n_1+1\}}\myn(z)\,g_\theta^{\{n_2+1\}}\myn(z)
+N g_\kappa^{\{n_1+1\}}\myn(z)\,g_\theta^{\{n_2+1\}}\myn(z)\\
&\quad+N
g_\theta^{\{n_1+1\}}\myn(z)\,g_\kappa^{\{n_2+1\}}\myn(z)+N^2
g_\kappa^{\{n_1+1\}}\myn(z)\,g_\kappa^{\{n_2+1\}}\myn(z).
\end{aligned}
\end{align}
The singularity of each term in \eqref{eq:Gexpansion} (and the
respective index $\beta$, see \eqref{eq:f_as}) is specified from
formulas \eqref{eq:assum_derivative_p_Upsilon-low},
\eqref{eq:assum_derivative_p_Upsilon} and
\eqref{eq:assum_derivative_g_Theta}. First of all, using
\eqref{eq:assum_derivative_g_Theta} we have
\begin{equation*}
g_\theta^{\{n_1+1\}}\myn(z)\,g_\theta^{\{n_2+1\}}\myn(z)\sim\frac{(\theta^*)^2\mypp
n_1\myn!\,n_2!}{(1-z/R\myp)^{n_1+n_2+2}}
\end{equation*}
(i.e., $\beta=n_1+n_2+2$), so formulas \eqref{eq:th3.3},
\eqref{eq:HN-2} of Theorem~\ref{thm:aux_asypmtotic_2} give
\begin{equation}\label{eq:theta*1+2}
\frac{1}{N^2H_N}\,[z^N]
\bigl[g_\theta^{\{n_1+1\}}\myn(z)\,g_\theta^{\{n_2+1\}}\myn(z)
\,\rme^{\myp\GN(z)}\}\bigr]\sim
\frac{\theta^*\,\Gamma(\theta^*\myn+1)\,n_1\myn!\,n_2!}{
\Gamma(\theta^*\myn+n_1+n_2+2)} \,(N\tilde{\nu})^{n_1+n_2},
\end{equation}
which coincides with the
asymptotics~\eqref{eq:moments_ell_1(n-ell_1)l_2}.

For $n_1\ge1$, contributions from other terms in
\eqref{eq:Gexpansion} are negligible as compared to $N^{n_1+n_2}$.
Indeed, from \eqref{eq:assum_derivative_p_Upsilon-low} and
\eqref{eq:assum_derivative_p_Upsilon} we get
\begin{equation*}
g_\kappa^{\{n_1+1\}}\myn(z)\,g_\kappa^{\{n_2+1\}}\myn(z)=O\bigl((1-z/R\myp)^{-\beta}\bigr)
\end{equation*}
with
$$
\beta=\max\myn\{0,n_1+1-s,n_2+1-s,n_1+n_2+2-2s\}<n_1+n_2,
$$
thanks to the condition $s>1$. Hence, by
Theorem~\ref{thm:aux_asypmtotic_2} we have
\begin{equation}\label{eq:kappa+2}
\frac{1}{N^2H_N}
[z^N]\bigl[N^2g_\kappa^{\{n_1+1\}}\myn(z)\,g_\kappa^{\{n_2+1\}}\myn(z)
\,\rme^{\myp\GN(z)}\bigr]=o\myp(N^{n_1+n_2}).
\end{equation}

Similarly, using \eqref{eq:assum_derivative_p_Upsilon-low},
\eqref{eq:assum_derivative_p_Upsilon} and
\eqref{eq:assum_derivative_g_Theta} we get
\begin{equation*}
g_\theta^{\{n_1+1\}}\myn(z)\,g_\kappa^{\{n_2+1\}}\myn(z)=\frac{\theta^*
n_1\myn!}{(1-z/R\myp)^{n_1+1}}\left\{O(1)
+O\bigl((1-z/R\myp)^{s-n_2-1}\bigr)\right\},
\end{equation*}
and Theorem~\ref{thm:aux_asypmtotic_2} with
$$
\beta=\max\myn\{n_1+1,n_1+n_2+2-s\}<n_1+n_2+1
$$ again gives
\begin{equation}\label{eq:theta*2+2}
\frac{1}{N^2H_N} [z^N]\bigl[N
g_\theta^{\{n_1+1\}}\myn(z)\,g_\kappa^{\{n_2+1\}}\myn(z)
\,\rme^{\myp\GN(z)}\bigr]=o\myp(N^{n_1+n_2}).
\end{equation}
By symmetry, the same estimate \eqref{eq:theta*2+2} holds for the
term $g_\kappa^{\{n_1+1\}}\myn(z)\,g_\theta^{\{n_2+1\}}\myn(z)$ with
$n_1\ge1$, $n_2\ge 1$. Thus, substituting \eqref{eq:theta*1+2},
\eqref{eq:kappa+2} and \eqref{eq:theta*2+2} into \eqref{eq:GGexpG},
we obtain~\eqref{eq:moments_ell_1(n-ell_1)l_2}.

The case $n_1=0$ requires more care; here we have (see
\eqref{eq:assum_derivative_p_Upsilon-low} and
\eqref{eq:assum_derivative_g_Theta})
\begin{equation*}
g_\kappa^{\{1\}}\myn(z)\,g_\theta^{\{n_2+1\}}\myn(z)
\sim\frac{g_\kappa^{\{1\}}\myn(R\myp)\,\theta^*
n_2!}{(1-z/R\myp)^{n_2+1}}\mypp,
\end{equation*}
and so Theorem~\ref{thm:aux_asypmtotic_2} with $\beta=n_2+1$ yields
\begin{equation}\label{eq:theta*1+3}
\frac{1}{N^2H_N} [z^N]\bigl[N
g_\kappa^{\{1\}}\myn(z)\,g_\theta^{\{n_2+1\}}\myn(z)
\,\rme^{\myp\GN(z)}\bigr]\sim
\frac{g_\kappa^{\{1\}}\myn(R\myp)\,\Gamma(\theta^*\myn+1)\,n_2!}{
\Gamma(\theta^*\myn+n_2+1)} \,(N\tilde{\nu})^{n_2},
\end{equation}
which is of the same order as the right-hand side of
\eqref{eq:theta*1+2} (with $n_1=0$). Hence, adding up the
contributions \eqref{eq:theta*1+2} and \eqref{eq:theta*1+3} and
recalling that $g_\kappa^{\{1\}}\myn(R\myp)=1-\tilde{\nu}$, we
obtain~\eqref{eq:moments_(n-ell_1)l_2}.
\end{proof}

\begin{remark}\label{rm:N-L1}
As should be clear from the proof, the factor
$1-L_1/N=N^{-1}(N-L_1)$ is included in
\eqref{eq:moments_ell_1(n-ell_1)l_2} and
\eqref{eq:moments_(n-ell_1)l_2} in order to cancel the denominator
$N-\ell_1$ in formula \eqref{eq:N-L1} of the two-dimensional
distribution of $(L_1,L_2)$. As suggested by the general formula
\eqref{eq:L=l} of Lemma~\ref{lem:distribution_of_ell_1,ell_2}, an
extension of Lemma~\ref{lem:ell_1_ell_2_super} to the
$m$-dimensional case $L_1,\dots,L_m$ requires the inclusion of the
product $\prod_{j=1}^{m-1} (N-L_1-\cdots-L_{j})$. The corresponding
calculations are tedious but straightforward, and follow the same
pattern as for $m=2$. A suitable extension is also possible for
Lemma~\ref{lm:D}.
\end{remark}

\begin{proof}[Proof\/ of\/ Theorem~\textup{\ref{thm:large_cycles_2}}]
For the sake of clarity, we consider only the case $m=2$ (i.e.,
involving the joint distribution of $L_1,L_2$); computations in the
general case require an extension of Lemmas~\ref{lm:D}
and~\ref{lem:ell_1_ell_2_super} (see Remark~\ref{rm:N-L1}) and can
be carried out along the same lines.

By the continuous mapping theorem and according to the definitions
\eqref{eq:xi-D-eta} and \eqref{eq:X}, the convergence
\eqref{eq:weak_conv_ell_j} with $m=2$ is equivalent to
\begin{equation*}
\frac{1}{N\tilde{\nu}}\bigl(L_1,L_2(1-L_1/N)\bigr)
\stackrel{d}{\longrightarrow} \bigl(X_1, X_2\myp (1-\tilde{\nu} X_1)
\bigr),\qquad N\to\infty.
\end{equation*}
By the method of moments, it suffices to show that for any
$n_1,n_2\in\NN_0$, as $N\to\infty$,
\begin{equation}\label{eq:comaring_moments0}
\Ens\myn\bigl[L_1^{n_1} L_2^{n_2}\myp(1-L_1/N)^{n_2}\bigr] \sim
(N\tilde{\nu})^{n_1+n_2}\,\EE\bigl[X_1^{n_1}\myn X_2^{n_2}\myp
(1-\tilde{\nu} X_1)^{n_2}\bigr].
\end{equation}
First, for $n_1=n_2=0$ the relation \eqref{eq:comaring_moments0} is
trivial, since both sides are reduced to $1$. If $n_2=0$ and $n_1>0$
then \eqref{eq:comaring_moments0} readily follows from the relations
\eqref{eq:ED1} (Lemma~\ref{lm:D}) and \eqref{eq:moments_ell_1}
(Lemma~\ref{lem:ell_1_ell_2_super}).

To cover the case $n_2\ge1$, let us prove a more general asymptotic
relation
\begin{equation}\label{eq:comaring_moments}
\Ens\myn\bigl[L_1^{n_1} L_2^{n_2}\myp(1-L_1/N)^{q}\bigr] \sim
(N\tilde{\nu})^{n_1+n_2}\,\EE\bigl[X_1^{n_1}\myn X_2^{n_2}\myp
(1-\tilde{\nu} X_1)^{q}\bigr],
\end{equation}
valid for all $n_1\in\NN_0$ and $n_2,q\in\NN$\myp. We argue by
induction on~$q$. For $q=1$, the relation
\eqref{eq:comaring_moments} is verified by comparing formulas
\eqref{eq:ED2} and \eqref{eq:moments_ell_1(n-ell_1)l_2} (for
$n_1\ge1$) or \eqref{eq:ED3} and \eqref{eq:moments_(n-ell_1)l_2}
(for $n_1=0$). Now suppose that \eqref{eq:comaring_moments} is true
for some $q\ge1$. Expanding
$$
(1-L_1/N)^{q+1}=(1-L_1/N)^{q}-N^{-1}L_1(1-L_1/N)^{q}
$$
and using the induction hypothesis \eqref{eq:comaring_moments}, we
get
\begin{align*}
\Ens\myn\bigl[L_1^{n_1} L_2^{n_2}\myp&(1-L_1/N)^{q+1}\bigr]\\
&=\Ens\myn\bigl[L_1^{n_1} L_2^{n_2}\myp(1-L_1/N)^{q}\bigr]-
N^{-1}\,\Ens\myn\bigl[L_1^{n_1+1}
L_2^{n_2}\myp(1-L_1/N)^{q}\bigr]\\
&\sim (N\tilde{\nu})^{n_1+n_2}\begin{aligned}[t]
\Bigl\{\EE&\bigl[X_1^{n_1}\myn X_2^{n_2}\myp (1-\tilde{\nu} X_1)^{q}
\bigr]-\tilde{\nu}\,\EE\bigl[X_1^{n_1+1}\myn
X_2^{n_2} \myp (1-\tilde{\nu} X_1)^{q}\bigr]\Bigr\}
\end{aligned}\\[.3pc]
&=(N\tilde{\nu})^{n_1+n_2}\,\EE\bigl[X_1^{n_1}\myn X_2^{n_2}\myp
(1-\tilde{\nu} X_1)^{q+1}\bigr],
\end{align*}
which verifies \eqref{eq:comaring_moments} for $q+1$, and therefore
for all $q\ge1$. This completes the proof of
Theorem~\ref{thm:large_cycles_2}.
\end{proof}

\subsection{Poisson--Dirichlet distribution for the cycle order statistics}

Let us now consider the cycle lengths without the lexicographic
ordering, and arrange them in decreasing order.
\begin{definition} \label{def:longest_cycle} For a permutation
$\sigma\in\SN$, let $L^{(1)}=L^{(1)}(\sigma)$ be the length of the
longest cycle in $\sigma$, $L^{(2)}=L^{(2)}(\sigma)$ the length of
the second longest cycle in $\sigma$, etc.
\end{definition}

Let us first prove a suitable ``cut-off'' lemma for
lexicographically ordered cycles.
\begin{lemma}\label{lm:cut-off}
For any $\varepsilon>0$, we have
\begin{equation}\label{eq:cut-off}
\lim_{n\to\infty}\limsup_{N\to\infty}\,\Pns\!\left(\myn\bigcup_{j>n}
\bigl\{N^{-1} L_j>\varepsilon\bigr\}\right)=0.
\end{equation}
\end{lemma}

\begin{proof}[Proof\/ of\/ Lemma \textup{\ref{lm:cut-off}}]
Fix a $K\in\NN$ and note that, for all $N\ge K/\varepsilon$,
\begin{align}
\notag \Pns\!\left(\myn\bigcup_{j>n} \bigl\{N^{-1}
L_j>\varepsilon\bigr\}\right)&=\Pns\!\left(\myn\bigcup_{j>n}
\bigl\{N^{-1}
L_j\mypp\bfOne_{\{L_j>K\}}>\varepsilon\bigr\}\right)\\
\label{eq:>K} &\le \Pns\myn\left\{\frac{1}{N}\sum_{j>n}
L_j\mypp\bfOne_{\{L_j>K\}}>\varepsilon\right\}.
\end{align}
Furthermore, noting that
\begin{equation}\label{eq:sum-sum}
\sum_{j>n} L_j\mypp\bfOne_{\{L_j>K\}}=\sum_{j=1}^\infty
L_j\mypp\bfOne_{\{L_j>K\}}-\sum_{j=1}^{n}
L_j\mypp\bfOne_{\{L_j>K\}},
\end{equation}
by Theorems \ref{thm:fractional_infinite_random}
and~\ref{thm:large_cycles_2} we have, as $N\to\infty$,
\begin{align*}
\frac{1}{N}\sum_{j=1}^\infty
L_j\mypp\bfOne_{\{L_j>K\}}\stackrel{p}{\longrightarrow}
\tilde{\nu}_K,\qquad \frac{1}{N}\sum_{j=1}^{n}
L_j\mypp\bfOne_{\{L_j>K\}}&\stackrel{d}{\longrightarrow}
\tilde{\nu}\sum_{j=1}^{n} X_j\mypp.
\end{align*}
Returning to \eqref{eq:sum-sum}, this gives
\[
\sum_{j>n} L_j\mypp\bfOne_{\{L_j>K\}}\stackrel{d}{\longrightarrow}
\tilde{\nu}_K-\tilde{\nu}\sum_{j=1}^{n} X_j,\qquad N\to\infty.
\]
Hence, recalling that $\lim_{K\to\infty}\tilde{\nu}_K=\tilde{\nu}$
(see~\eqref{eq:def_v_near_spatial}), from \eqref{eq:>K} we obtain
\begin{align*}
\limsup_{N\to\infty}\,\Pns\!\left(\myn\bigcup_{j>n} \bigl\{N^{-1}
L_j>\varepsilon\bigr\}\right)&\le
\liminf_{K\to\infty}\,\PP\myn\left\{\tilde{\nu}\sum_{j=1}^{n} X_j\le
\tilde{\nu}_K-\varepsilon\right\}\\
&=\PP\myn\left\{\mynn\myn\sum_{j=1}^{n} X_j\le
1-\varepsilon/\tilde{\nu}\right\}.
\end{align*}
Finally, passing here to the limit as $n\to\infty$ and noting that,
by Lemma~\ref{lm:sum-Xn}, $\sum_{j=1}^{\infty} X_j =1$ (a.s.), we
arrive at~\eqref{eq:cut-off}, as claimed.
\end{proof}

The next theorem is our main result in this subsection. Recall that
the parameter $\theta^*\myn>0$ is involved in the
assumption~\eqref{eq:assum_derivative_g_Theta}.

\begin{theorem}\label{thm:large_cycles_1}
In the sense of convergence of finite-dimensional distributions,
\begin{equation*}
\frac{1}{N\tilde{\nu}}\,\bigl(L^{(1)}\myn,L^{(2)}\myn,\dots\bigr)
\stackrel{d}{\longrightarrow} \PD(\theta^*),\qquad N\to\infty,
\end{equation*}
where $\PD(\theta^*)$ denotes the Poisson--Dirichlet distribution
with parameter $\theta^*$.
\end{theorem}

\begin{proof}
By virtue of Lemma \ref{lm:large_cycles_1}, it suffices to show
that, for each $m\in\NN$\myp,
\begin{equation}\label{eq:poisson_diriclet}
\frac{1}{N\tilde{\nu}} \,\bigl(L^{(1)}\myn,\dots,L^{(m)}\bigr)
\stackrel{d}{\longrightarrow}
\bigl(X^{(1)}\myn,\dots,X^{(m)}\bigr),\qquad N\to\infty.
\end{equation}
Let us first verify \eqref{eq:poisson_diriclet} for $m=1$. Fix an
integer $n\ge 1$ and observe that, for any $x\in(0,1)$,
$$
\Pns\bigl\{(N\tilde{\nu})^{-1}L^{(1)}\myn>x\bigr\}\ge
\Pns\Bigl\{(N\tilde{\nu})^{-1}\max_{j\le n} L_j>x\Bigr\}.
$$
Hence, by Theorem~\ref{thm:large_cycles_2} and the continuous
mapping theorem, it follows that
\begin{align}
\notag
\liminf_{N\to\infty}\,\Pns\bigl\{(N\tilde{\nu})^{-1}L^{(1)}\myn>x\bigr\}
&\ge \lim_{n\to\infty}\PP\Bigl\{\max_{j\le n} X_j>x\Bigr\}\\
\label{eq:LB} &\ge \PP\{X^{(1)}\myn>x\},
\end{align}
because $\max_{j\le n} X_j\uparrow X^{(1)}$ as $n\uparrow\infty$. On
the other hand, we have an upper bound
$$
\Pns\bigl\{(N\tilde{\nu})^{-1}L^{(1)}\myn>x\bigr\}\le
\Pns\Bigl\{(N\tilde{\nu})^{-1}\max_{j\le n}
L_j>x\Bigr\}+\Pns\Bigl\{N^{-1}\max_{j>n} L_j>\tilde{\nu} x\Bigr\},
$$
and by Theorem~\ref{thm:large_cycles_2} and Lemma~\ref{lm:cut-off}
this yields (cf.~\eqref{eq:LB})
\begin{align}
\notag
\limsup_{N\to\infty}\,\Pns\bigl\{(N\tilde{\nu})^{-1}L^{(1)}\myn>x\bigr\}&\le\lim_{n\to\infty}
\limsup_{N\to\infty}\,\Pns\Bigl\{(N\tilde{\nu})^{-1}\max_{j\le n}
L_j>x\Bigr\}\\
\notag &\le\lim_{n\to\infty} \PP\myp\Bigl\{\max_{j\le n}
X_j>x\Bigr\}\\
\label{eq:UB} &\le \PP\myp\bigl\{X^{(1)}\myn\ge x\bigr\}.
\end{align}
Combining \eqref{eq:LB} and \eqref{eq:UB} and assuming that
$x\in(0,1)$ is a point of continuity of the distribution of
$X^{(1)}$ (which is, in fact, automatically true owing to
Lemma~\ref{lm:large_cycles_1}), we obtain
$$
\lim_{N\to\infty}\Pns\bigl\{(N\tilde{\nu})^{-1}L^{(1)}\myn>x\bigr\}=\PP\{X^{(1)}\myn>x\},
$$
which proves \eqref{eq:poisson_diriclet} with $m=1$.

The general case $m\ge2$ is handled in a similar manner, by using
lower and upper estimates for the $m$-dimensional probability
$\Pns\bigl\{(N\tilde{\nu})^{-1}L^{(1)}\myn>x_1,\dots,
(N\tilde{\nu})^{-1}L^{(m)}\myn>x_m\bigr\}$ through the similar
probabilities for the order statistics of the truncated sample
$L_1,\dots,L_n$ \,($n\ge m$), where the ``discrepancy'' term due to
the contribution of the tail part $(L_{j},\,j>n)$ may be shown to be
asymptotically negligible as $n\to\infty$ by virtue of the cut-off
Lemma~\ref{lm:cut-off}.
\end{proof}

\begin{remark}\label{rm:thm5.11extension}
As already mentioned in Remark~\ref{rm:sing_g_kappa_and_theta}, the
requirement imposed in Section \ref{sec:5.1} that the asymptotic
expansion \eqref{eq:g_kappa_as=0} of $g_\kappa(z)$ holds with a
non-integer index $s>1$ may be extended to allow a
\emph{power-logarithmic} term
$b_{s}(1-z/R\myp)^{s}\log\myn(1-z/R\myp)$ (with \emph{any} $s>1$).
With the asymptotic formulas
\eqref{eq:assum_derivative_p_Upsilon-low},
\eqref{eq:assum_derivative_p_Upsilon} modified accordingly, the
proof of Lemma~\ref{lem:ell_1_ell_2_super} may be adapted as
appropriate, implying that Theorems \ref{thm:large_cycles_2} and
\ref{thm:large_cycles_1} remain true.
\end{remark}

\subsection{Case $\theta^*\myn = 0$}\label{sec:theta*=0}

Let us now turn to studying the asymptotic behaviour of cycles in
the case $\theta^*\myn = 0$ (see~\eqref{eq:g_theta_as}). More
precisely, throughout this subsection we suppose that, as in
Theorem~\ref{thm:aux_asypmtotic_2a}, the generating function
$g_\theta(z)$ is holomorphic in a suitable domain $\varDelta_0$ (see
Definition~\ref{def:Delta0}) and, moreover, is regular at point
$z=R$; in particular, the successive (modified) derivatives of
$g_\theta(z)$ have Taylor-type asymptotic expansions, as $z \to R$
\,($z\in \varDelta_0$),
\begin{equation*}
g_\theta^{\{n\}}\myn(z)\sim\left(\frac{z}{R}\right)^n\sum_{j=n}^\infty
\frac{(-1)^{j-n}\myp
g_\theta^{\{j\}}\myn(R\myp)}{(j-n)!}\,(1-z/R\myp)^{j-n}.
\end{equation*}
We also assume that the generating function $g_\kappa(z)$ admits the
asymptotic expansion~\eqref{eq:g_kappa_as=0} of
Theorem~\ref{thm:aux_asypmtotic_2a} (with a non-integer $s>1$),
which may be differentiated any number of times to yield a nested
family of expansions (with some $\delta_n>0$, \,$n\in\NN_0$),
\begin{align}\notag
g_\kappa^{\{n\}}\myn(z) = \left(\frac{z}{R}\right)^n
\Biggl\{&\sum_{n\le j<s} \frac{(-1)^{j-n}\myp
g_\kappa^{\{j\}}\myn(R\myp)}{(j-n)!}\left(1-\frac{z}{R}\right)^{j-n}\\
\label{eq:g_kappa_as=theta*=0}
&+\frac{\Gamma(-s+n)\,a_s}{\Gamma(-s)}\left(1-\frac{z}{R}\right)^{s-n}\Biggr\}
+O\mynn\left(\left(1-\frac{z}{R}\right)^{s-n+\delta_n}\right),
\end{align}
where the first sum is understood to vanish if $n>s$ (cf.\
\eqref{eq:assum_derivative_p_Upsilon-low},~\eqref{eq:assum_derivative_p_Upsilon}).

Let us now revisit the modified stick-breaking process underpinning
our argumentation in the case $\theta^*\myn>0$ in
Section~\ref{sec:stick}. Observe that if $\theta^*\myn\downarrow0$
then the beta distribution $\Beta(1,\theta^*)$ of the random
variables $B_n$ converges to Dirac delta measure
$\delta_1(\rmd{x})$, since for any $x\in[0,1)$
$$
\PP\{B_n>x\}=\int_x^1 \theta^* (1-u)^{\theta^*\myn-1}\,\rmd
u=(1-x)^{\theta^*}\to 1,\qquad \theta^*\to 0.
$$
It is easy to see that, under this limit, equations
\eqref{eq:xi-D-eta} and \eqref{eq:X} are greatly simplified to the
following. Let $(\xi_n)$ be a sequence of i.i.d.\ Bernoulli random
variables with success probability $\PP\{\xi_n=1\}=\tilde{\nu}>0$.
Let  $\tau_1:=\min\{n:\,\xi_n=1\}<\infty$ (a.s.)\ be the random time
until first success, with geometric distribution
\begin{equation}\label{eq:tau1=geometric}
\PP\{\tau_1=k\}=(1-\tilde{\nu})^{k-1}\myp\tilde{\nu},\qquad
k\in\NN\myp.
\end{equation}
Now, for $n\in\NN$ we set
\begin{equation}\label{eq:X/theta*=0}
X_n:=\bfOne_{\{\tau_1=n\}}=\left\{\begin{array}{ll}
0,&n\ne \tau_1,\\[.2pc]
1,&n=\tau_1.
\end{array}
\right.
\end{equation}

\begin{remark}
Formula \eqref{eq:X/theta*=0} shows that the modified stick-breaking
process $(X_n)$ described in Remark \ref{rm:stick} for
$\theta^*\myn>0$, reduces in the case $\theta^*\myn=0$ to removing
the entire breakable part $[0,\tilde{\nu}\myp]$ at once after
waiting time $\tau_1$.
\end{remark}

The following analogue of Lemma~\ref{lm:D} is formally obtained by
substituting $\theta^*\myn=0$; its proof is elementary by using the
definition \eqref{eq:X/theta*=0} and the distribution of $\tau_1$
(see~\eqref{eq:tau1=geometric}).
\begin{lemma}\label{lm:D2}
For any $n_1,n_2\in\NN$\myp,
\begin{gather*}
\EE\myp(X_1^{n_1})=\tilde{\nu},\\[.1pc]
\label{eq:ED11} \EE\myp\bigl[X_1^{n_1}\myn X_2^{n_2}\myp
(1-\tilde{\nu}
X_1)\bigr]=0,\\[.1pc]
\EE\myp\bigl[X_2^{n_2}\myp (1-\tilde{\nu}X_1)\bigr]
= (1-\tilde{\nu})\myp \tilde{\nu}\myp.
\end{gather*}
\end{lemma}

Next, we prove an analogue of Lemma~\ref{lem:ell_1_ell_2_super},
which formally looks as its particular case with $\theta^*\myn=0$
(cf.\ \eqref{eq:moments_ell_1}, \eqref{eq:moments_ell_1(n-ell_1)l_2}
and~\eqref{eq:moments_(n-ell_1)l_2}).

\begin{lemma}\label{lem:ell_1_ell_2_super_2}
For any $n_1,n_2\in\NN$\myp,
\begin{gather}\label{eq:moments_ell_1_1}
\lim_{N\to\infty} \frac{1}{(N\tilde\nu)^{n_1}}\,
\Ens (L_1^{n_1}) = \tilde{\nu},\\
\label{eq:moments_ell_1(n-ell_1)l_2_2} \lim_{N\to\infty}
\frac{1}{(N\tilde\nu)^{n_1+n_2}}\, \Ens\myn\bigl[L_1^{n_1}\myn
L_2^{n_2}\myp(1-L_1/N)\bigr] = 0,\\
\label{eq:moments_(n-ell_1)l_2_3} \lim_{N\to\infty}
\frac{1}{(N\tilde\nu)^{n_2}} \,\Ens\myn\bigl[L_2^{n_2}\myp (1-L_1/N)
\bigr] =(1-\tilde{\nu})\myp\tilde{\nu}\myp.
\end{gather}
\end{lemma}

\begin{proof}
We use similar argumentation as in the proof of
Lemma~\ref{lem:ell_1_ell_2_super}, but now exploiting
Theorem~\ref{thm:aux_asypmtotic_2a}. First of all, according to
\eqref{eq:moment_ell_1_with_gen} we have, for any $n_1\in\NN$\myp,
\begin{equation}\label{eq:moment_ell_1_with_gen-theta*=0}
\Ens\myn\bigl[(L_1-1)_{n_1}\bigr]=
\frac{1}{NH_N}\,[z^N]\!\left[\bigl(g_\theta^{\{n_1+1\}}
\myn(z)+Ng_\kappa^{\{n_1+1\}}\myn(z)\bigr)\,\rme^{\myp\GN(z)}\right].
\end{equation}
Applying Theorem~\ref{thm:aux_asypmtotic_2a}\myp(i) with
$f(z)=g_\theta^{\{n_1+1\}}\myn(z)$ and $\beta=\infty$ (see
\eqref{eq:th3.3=0(i)} and \eqref{eq:HN-2=0}) gives
\begin{equation}\label{eq:theta*=01}
\frac{1}{N H_N}\,[z^N]\bigl[g_\theta^{\{n+1\}}\myn(z)
\,\rme^{\myp\GN(z)}\bigr]\sim \frac{1}{N}\to 0,\qquad N\to\infty.
\end{equation}
On the other hand, on account of the asymptotic expansion
\eqref{eq:g_kappa_as=theta*=0} (with $n=n_1$), by
Theorem~\ref{thm:aux_asypmtotic_2a}\myp(ii) with
$f(z)=g_\kappa^{\{n_1+1\}}\myn(z)$ and $\beta=s-n_1-1<s-1$ we obtain
\begin{equation}\label{eq:kappa*1<stheta*=0}
\frac{1}{N H_N}\,[z^N]\bigl[Ng_\kappa^{\{n_1+1\}}\myn(z)
\,\rme^{\myp\GN(z)}\}\bigr]\sim (N\tilde{\nu})^{n_1}\myp\tilde{\nu},
\qquad N\to\infty,
\end{equation}
recalling that $\tilde{\nu} = 1-g_\kappa^{\{1\}}\myn(R\myp)$
(see~\eqref{eq:tilde-nu}). Hence, substituting \eqref{eq:theta*=01}
and \eqref{eq:kappa*1<stheta*=0} into
\eqref{eq:moment_ell_1_with_gen-theta*=0} yields
$\Ens\myn\bigl[(L_1-1)_{n_1}\bigr] \sim
(N\tilde{\nu})^{n_1}\myp\tilde{\nu}$, which
implies~\eqref{eq:moments_ell_1_1}.

We now turn to \eqref{eq:moments_ell_1(n-ell_1)l_2_2} and
\eqref{eq:moments_(n-ell_1)l_2_3}. Again considering the factorial
moments, by formula \eqref{eq:GGexpG} we have, for any integers
$n_1\ge0$, $n_2\ge 1$,
\begin{equation}\label{eq:GGexpG-theta*=0}
\Ens\bigl[(L_1-1)_{n_1} (L_2-1)_{n_2}\myp(1-L_1/N)\bigr]=
\frac{1}{N^2H_N}\,
[z^N]\bigl[\GN^{\{n_1+1\}}\myn(z)\,\GN^{\{n_2+1\}}\myn(z)
\,\rme^{\myp\GN(z)}\bigr],
\end{equation}
where the product $\GN^{\{n_1+1\}}\myn(z)\,\GN^{\{n_2+1\}}\myn(z)$
is expanded in \eqref{eq:Gexpansion}. Note that, like in
\eqref{eq:theta*=01},
\begin{equation}\label{eq:kappa*1<stheta*=011}
\frac{1}{N^2
H_N}\,[z^N]\bigl[g_\theta^{\{n_1+1\}}\myn(z)\,g_\theta^{\{n_2+1\}}\myn(z)
\,\rme^{\myp\GN(z)}\}\bigr]=O(N^{-2}), \qquad N\to\infty.
\end{equation}
Suppose that $n_1,n_2\ge 1$. Then, similarly to
\eqref{eq:kappa*1<stheta*=0}, we have as $N\to\infty$
\begin{align}
\label{eq:kappa*1<stheta*=0_2} \frac{1}{N^2
H_N}\,[z^N]\bigl[Ng_\kappa^{\{n_2+1\}}\myn(z)
\,g_\theta^{\{n_1+1\}}\myn(z)\,\rme^{\myp\GN(z)}\}\bigr]& =O(N^{n_2-1}),\\
\label{eq:kappa*1<stheta*=0_3} \frac{1}{N^2
H_N}\,[z^N]\bigl[Ng_\kappa^{\{n_1+1\}}\myn(z)\,g_\theta^{\{n_2+1\}}\myn(z)
\,\rme^{\myp\GN(z)}\}\bigr] &=O(N^{n_1-1}).
\end{align}
Furthermore, applying Theorem~\ref{thm:aux_asypmtotic_2a}\myp(ii)
with $f(z)=g_\kappa^{\{n_1+1\}}\myn(z)\,g_\kappa^{\{n_2+1\}}\myn(z)$
and
$$
\beta=\min\myp\{s-n_1-1,\,s-n_2-1\}<s-1,
$$
and noting that $\beta\ge s-n_1-n_2$ , we obtain
\begin{equation}\label{eq:kappa+3}
\frac{1}{N^2H_N}\,
[z^N]\bigl[N^2g_\kappa^{\{n_1+1\}}\myn(z)\,g_\kappa^{\{n_2+1\}}\myn(z)
\,\rme^{\myp\GN(z)}\bigr]=O(N^{s-\beta-1})=O\myp(N^{n_1+n_2-1}).
\end{equation}
Hence, collecting the estimates \eqref{eq:kappa*1<stheta*=011},
\eqref{eq:kappa*1<stheta*=0_2}, \eqref{eq:kappa*1<stheta*=0_3} and
\eqref{eq:kappa+3}, we obtain the
claim~\eqref{eq:moments_ell_1(n-ell_1)l_2_2}.

The same proof shows that the estimates
\eqref{eq:kappa*1<stheta*=011} and \eqref{eq:kappa*1<stheta*=0_2}
are valid with $n_1=0$; this is also true for
\eqref{eq:kappa*1<stheta*=0_3}, which can be proved using
Theorem~\ref{thm:aux_asypmtotic_2a}\myp(iii) with $\beta=s-1$.
Finally, applying Theorem~\ref{thm:aux_asypmtotic_2a}\myp(ii) with
$\beta=s-n_2-1<s-1$, the estimate \eqref{eq:kappa+3} is sharpened to
\begin{equation*}
\frac{1}{N^2H_N}
[z^N]\bigl[N^2g_\kappa^{\{1\}}\myn(z)\,g_\kappa^{\{n_2+1\}}\myn(z)
\,\rme^{\myp\GN(z)}\bigr]\sim (N\tilde{\nu})^{n_2}\mypp\tilde{\nu},
\end{equation*}
giving the leading term in the asymptotics of
\eqref{eq:GGexpG-theta*=0} (with $n_1=0$), which
proves~\eqref{eq:moments_(n-ell_1)l_2_3}.
\end{proof}

\begin{remark}
Lemmas~\ref{lm:D2} and~\ref{lem:ell_1_ell_2_super_2} can be extended
to the $m$-dimensional case, that is, with $X_1,\dots,X_m$ and
$L_1,\dots,L_m$, respectively (cf.\ Remark~\ref{rm:N-L1}).
\end{remark}

The next theorem is the corresponding version of
Theorem~\ref{thm:large_cycles_2}.
\begin{theorem}\label{th:law_ell_j_theta0}
For each $m\in\NN$\myp,
\begin{equation*}
\frac{1}{N\tilde{\nu}}\,(L_1,\dots,L_m)
\stackrel{d}{\longrightarrow} (X_1,\dots,X_m),\qquad N\to\infty,
\end{equation*}
where $X_j$'s are Bernoulli random variables defined
in~\eqref{eq:X/theta*=0}.
\end{theorem}

\begin{proof}
The proof is precisely the same as that of
Theorem~\ref{thm:large_cycles_2}, now using Lemmas \ref{lm:D2}
and~\ref{lem:ell_1_ell_2_super_2} in place of Lemmas \ref{lm:D} and
\ref{lem:ell_1_ell_2_super}, respectively.
\end{proof}

We finally obtain our main result about convergence of ordered
cycles in the case $\theta^*\myn=0$, which is in sharp contrast with
Theorem~\ref{thm:large_cycles_1}.
\begin{theorem}\label{thm:large_cycles_0}
In the sense of convergence of finite-dimensional distributions,
\begin{equation}\label{eq:->100}
\frac{1}{N\tilde{\nu}}\,\bigl(L^{(1)}\myn,L^{(2)}\myn,\dots\bigr)
\stackrel{d}{\longrightarrow} (1,0,0,\dots),\qquad N\to\infty.
\end{equation}
\end{theorem}

\begin{proof}
From the definition \eqref{eq:X/theta*=0} it is clear that
rearranging the sequence $(X_n)$ in decreasing order gives
$(1,0,0,\dots)$ (cf.\ the right-hand side of~\eqref{eq:->100}). The
rest of the proof proceeds exactly as in
Theorem~\ref{thm:large_cycles_1}, again using the general cut-off
Lemma~\ref{lm:cut-off}.
\end{proof}

\begin{corollary}\label{cor:LLN}
Weak convergence of the first component in \eqref{eq:->100} entails
the following law of large numbers for the longest cycle in the case
$\theta^*\myn=0$,
\begin{equation*}
\frac{\,L^{(1)}}{N\tilde{\nu}} \stackrel{p}{\longrightarrow}
1,\qquad N\to\infty.
\end{equation*}
\end{corollary}

\begin{remark}
The result of Theorem~\ref{thm:large_cycles_0} means that there is a
single giant cycle (of size about $N\tilde{\nu}$) emerging as
$N\to\infty$.
\end{remark}

\begin{remark}\label{rm:thm5.13extension}
The condition of regularity of $g_\theta(z)$ at $z=R$ imposed at the
beginning of Section \ref{sec:theta*=0} is not essential; in the
spirit of Remark~\ref{rm:sing_g_kappa_and_theta}, Lemmas \ref{lm:D2}
and~\ref{lem:ell_1_ell_2_super_2} (underlying the proof of
Theorem~\ref{th:law_ell_j_theta0}) may be extended to the case where
$g_\theta(z)$ has a power and possibly also a power-logarithmic
singularity, $\tilde{b}_{s_1}(1-z/R\myp)^{s_1}\log\myn(1-z/R\myp)$,
as long as $s_1\myn>0$.
Furthermore, again alluding to
Remark~\ref{rm:sing_g_kappa_and_theta} it is not hard to see that
all calculations can be adapted for the asymptotic expansion
\eqref{eq:g_kappa_as=0} of $g_\kappa(z)$ to include a
power-logarithmic term $b_s(1-z/R\myp)^s\log\myn(1-z/R\myp)$, in
which case the index $s>1$ is permitted to be integer.
\end{remark}

\begin{remark} Ercolani et al.\ \cite[Theorem~7.3]{ErJaUe14},
using a different method, have obtained the limit distribution of
fluctuations of $L^{(1)}$ about $N\tilde{\nu}$ (cf.\
Corollary~\ref{cor:LLN}) in the particular case of $\theta_j\equiv
0$ and for certain concrete sequences $(\kappa_j)$; for instance, if
$\kappa_j=j^{-s}$ and $1<s<2$ then $(L^{(1)}-N\tilde{\nu})/N^{1/s}$
converges weakly to a stable distribution with characteristic
exponent~$s$, whereas if $s>2$ then
$(L^{(1)}-N\tilde{\nu})/\sqrt{N}$ is asymptotically normal as
$N\to\infty$.
\end{remark}

\section{Comparison with the Spatial Model}\label{sec:comparison}
The aim of this section is to bridge the gap between the
surrogate-spatial ($\Pns$) and spatial ($\Psp$) models. First, in
Section \ref{sec:6.1} we will explore in some detail the
Euler--Maclaurin type approximation \eqref{eq:approx_weigths} of the
Riemann sums arising in the definition of the measure $\Psp$
(see~\eqref{eq:Pspatial_with_partition}). More specifically, the
opening general remarks based on the Poisson summation formula
(Section~\ref{sec:6.1.1}) are followed by an in-depth analysis of
three concrete examples including the basic Gaussian case (Sections
\ref{sec:6.1.2}\mypp--\mypp\ref{sec:6.1.4}), summarized in
Section~\ref{sec:heuristic_conclusions} by a heuristic discussion of
limitations of the surrogate-spatial model and its possible
improvements. In Section~\ref{sec:6.2} we will show how the concept
of the system ``density'' may be introduced and interpreted in our
model. Finally, in Section~\ref{sec:comparison_1} we will compare
the asymptotic results for the cycle statistics obtained in the
present paper and by Betz and Ueltschi~\cite{BeUe11a}.

\subsection{Asymptotics of the Riemann sums}\label{sec:6.1}
To justify the ansatz \eqref{eq:approx_weigths}, which has motivated
the surrogate-spatial model \eqref{eq:def_near_spatial}, we need to
look more carefully at the Riemann integral sums in
\eqref{eq:Pspatial_with_partition}. Rather than using
Euler--Maclaurin's summation formula as suggested in Section
\ref{sec:1.3}, we take advantage of the link \eqref{eq:eps} between
the functions $\rme^{-\varepsilon(\bfs)}$ and $\rme^{-V(\bfx)}$ and
deploy the powerful \emph{Poisson summation formula} (see, e.g.,
\cite[\S\myp3.12, p.~52]{Bru} or \cite[\S\mypp{}XIX.5,
p.~630]{Feller}) yielding the identity
\begin{equation}\label{eq:PSS}
\sum_{k\in\ZZ} \varphi(k/\lambda)=\lambda\sum_{\ell\in\ZZ}
f(\lambda\ell)\qquad (\lambda>0),
\end{equation}
where the functions $\varphi(s)$ and $f(x)$ are the reciprocal
Fourier transforms,
\begin{align*}
\varphi(s)&=\int_\RR \rme^{-2\pi\rmi
xs}\,f(x)\,\rmd{x},\qquad
s\in\RR,\\
f(x)&=\int_\RR \rme^{\myp 2\pi\rmi xs}\,\varphi(s)\,\rmd{s},\qquad
x\in\RR,
\end{align*}
such that $f(x)$ is a probability density and, moreover,
$\varphi(s)$ is absolutely integrable on $\RR$.

For simplicity, in higher dimensions we restrict ourselves to
\emph{isotropic} (radially symmetric) potentials,
\begin{equation*}
V(\bfx)=\sum_{i=1}^d V_0(x_i),\qquad \bfx=(x_1,\dots,x_d)\in\RR^d,
\end{equation*}
leading to the decomposition
\begin{equation}\label{eq:eps-isotropic}
\varepsilon(\bfs)=\sum_{i=1}^d \varepsilon_0(s_i),\qquad \qquad
\bfs=(s_1,\dots,s_d)\in\RR^d,
\end{equation}
with the function $\varepsilon_0(\cdot)$ defined by
\begin{equation}\label{eq:epsilon_0}
\rme^{-\varepsilon_0(s)}=\int_\RR \rme^{-2\pi\rmi \myp xs}\myp
f_0(x)\,\rmd{x}\myp,\qquad f_0(x):=\rme^{-V_0(x)}.
\end{equation}
Hence, the $d$-dimensional sum in \eqref{eq:Pspatial_with_partition}
(with $j=1,\dots,N$) is decomposable as a product,
\begin{equation}\label{eq:product-sum}
\sum_{\bfk\in\ZZ^d} \rme^{- j\myp\varepsilon(\bfk/L)}=\prod_{i=1}^d
\sum_{k_i\in\ZZ} \rme^{-
j\myp\varepsilon_0(k_i/L)}=\left(\sum_{k\in\ZZ} \rme^{-
j\myp\varepsilon_0(k/L)}\right)^d,
\end{equation}
and likewise
\begin{equation}\label{eq:product-int}
\int_{\RR^d} \rme^{-
j\myp\varepsilon(\bfs)}\,\rmd{\bfs}=\prod_{i=1}^d \int_{\RR} \rme^{-
j\myp\varepsilon_0(s_i)}\,\rmd{s_i}=\left(\int_{\RR} \rme^{-
j\myp\varepsilon_0(s)}\,\rmd{s}\right)^d.
\end{equation}

Relation \eqref{eq:epsilon_0} implies that the function
$\exp\{-j\myp\varepsilon_0(s)\}$ is the Fourier transform of the
$j$-fold convolution $f_0^{\star j}(x)=f_0\star\cdots\star f_0(x)$
\,($j\in\NN$). Furthermore, from the Fourier inversion formula for
$f_0^{\star j}(x)$ (with $x=0$) we get
\begin{equation}\label{eq:f*j(0)}
f_0^{\star j}(0)=\int_\RR \rme^{-j\myp \varepsilon_0(s)}\,\rmd{s}.
\end{equation}
Hence, the Poisson formula \eqref{eq:PSS} (with $\lambda=L$)
together with the product decompositions \eqref{eq:product-sum} and
\eqref{eq:product-int} yields the key representation
\begin{align}
\label{eq:PSS^d} \sum_{\bfk\in\ZZ^d}
\rme^{-j\myp\varepsilon(\bfk/L)}&=L^d\left(\sum_{\ell\in\ZZ}
f_0^{\star
j}(\ell L)\right)^d\\
\label{eq:PSS_j} &=\rho^{\mypp-1}N \int_{\RR^d} \rme^{-j\myp
\varepsilon(\bfs)}\,\rmd{\bfs} \cdot\left(1+\frac{\sum_{\ell\ne0}
f_0^{\star j}(\ell L)}{f_0^{\star j}(0)}\right)^d,
\end{align}
on account of the thermodynamic calibration $L^d=\rho^{\mypp-1}\myn
N$ (see Section~\ref{sec:1.2}). On comparison with the conjectural
approximation \eqref{eq:approx_weigths}, we see that the expression
\eqref{eq:PSS_j} may be quite useful in providing information about
the Riemann sums asymptotics.

Let us prove one general result in this direction. Recall that the
probability density $f_0(x)$ is always assumed to be symmetric,
$f_0(-x)= f_0(x)$ \,($x\in\RR$). A symmetric function is called
\emph{unimodal} if it is non-increasing for $x\ge 0$.

\begin{lemma}\label{lm:unimodal}
Assume that $f_0(x)$ is unimodal, then for each $j\in\NN$ and any
$L>0$
\begin{equation}\label{eq:<1/L}
\sum_{\ell\ne0} f_0^{\star j}(\ell L)\le
\frac{4}{L}\int_{L/2}^\infty f_0^{\star j}(x)\,\rmd{x}.
\end{equation}
\end{lemma}

\begin{proof} Observe that the convolutions $f_0^{\star j}(x)$ are also
symmetric and unimodal (see \cite[Theorem 2.5.2, p.~67]{IL}).
Hence, for any $\ell\ge 1$ and $y>0$
\begin{equation}\label{eq:rectangular}
f_0^{\star j}(\ell y)\le \int_{\ell - 1}^{\ell} f_0^{\star
j}(uy)\,\rmd{u}.
\end{equation}
On the other hand, inserting extra ``mid-terms'' into the sum in
\eqref{eq:<1/L} and then using the bound \eqref{eq:rectangular} with
$y=L/2$, we obtain
\begin{align*}
\sum_{\ell\ne 0} f_0^{\star j}(\ell L)=2\sum_{\ell=1}^\infty
f_0^{\star j}(\ell L)&\le 2\sum_{\ell=2}^\infty f_0^{\star j}(\ell L/2)\\
&\le 2\int_{1}^\infty \!f_0^{\star
j}(uL/2)\,\rmd{u}=\frac{4}{L}\int_{L/2}^\infty f_0^{\star
j}(x)\,\rmd{x},
\end{align*}
and the estimate \eqref{eq:<1/L} is proved.
\end{proof}

\begin{remark}
The bound \eqref{eq:<1/L}, together with the Poisson formula
\eqref{eq:PSS_j}, ensures that the Riemann sum on the left-hand side
of \eqref{eq:PSS_j} is finite for all $L>0$, and so the series
convergence condition \eqref{eq:sum-eps} is automatically satisfied.
\end{remark}

By Lemma \ref{lm:unimodal}, from \eqref{eq:f*j(0)} and
\eqref{eq:PSS_j} we get for each $j\in\NN$ the asymptotic
equivalence
\begin{equation*}
\sum_{\bfk\in\ZZ^d} \rme^{-j\myp\varepsilon(\bfk/L)}\sim
\rho^{\mypp-1}N \int_{\RR^d} \rme^{-j\myp
\varepsilon(\bfs)}\,\rmd{\bfs}\qquad (N\to\infty),
\end{equation*}
with the absolute error
\begin{equation}\label{eq:remainder}
\begin{aligned}
\Delta_N^{(j)}:={}&\sum_{\bfk\in\ZZ^d}
\rme^{-j\myp\varepsilon(\bfk/L)}- \rho^{\mypp-1} N
\int_{\RR^d} \rme^{-j\myp \varepsilon(\bfs)}\,\rmd{\bfs}\\[.3pc]
 \sim{}&\rho^{\mypp-1} N\myp d\,\bigl\{f_0^{\star
j}(0)\bigr\}^{d-1}\sum_{\ell\ne0} f_0^{\star j}(\ell
L)=o(N^{1-1/d}), \qquad N\to\infty.
\end{aligned}
\end{equation}
Obtaining more accurate asymptotics of $\Delta_N^{(j)}$ (in
particular, investigating if it converges to a \emph{constant} as
suggested by the term $\theta_j$ in \eqref{eq:approx_weigths}), as
well as treating the case of $j$ growing with $N\to\infty$, requires
more information about the tail of the density $f_0(x)$ and its
convolutions $f_0^{\star j}(x)$ as $x\to\infty$, which will also
translate into the behaviour of the function $\varepsilon_0(s)$ for
$s\approx 0$ needed for the asymptotics of the integral in
\eqref{eq:PSS_j}, according to the Laplace method.


\begin{remark}\label{rm:LD}
Note that the integral in \eqref{eq:<1/L} can be written as the
probability $\PP\{S_j\ge L/2\}$, where $S_j:=X_1+\dots+X_j$ and
$(X_i)$ are i.i.d.\ random variables each with density $f_0(x)$;
hence, one can use suitable results from the large deviations theory
(see, e.g., \cite{Nagaev}) to get further bounds on \eqref{eq:<1/L}.
We will use this idea below to treat the example in
Section~\ref{sec:6.1.3}.
\end{remark}

Rather then attempting to develop any further general results, we
will illustrate some typical asymptotic effects by considering a few
``exactly solvable'' examples classified according to the type of
the probability density $f_0(x)=\rme^{-V_0(x)}$: \,(i)
\,$\varepsilon_0(s)=s^2$ \,(\emph{Gaussian}); \,(ii)
\,$\varepsilon_0(s)=|s|^\gamma$, \,$0<\gamma<2$ \,(\emph{stable});
\,(iii) \,$f_0(x)=\mu_0\mypp\rme^{-|x|^\gamma}$, \,$0<\gamma<2$
\,(\emph{exponential-power}). In what follows, we write $a_N\asymp
b_N$ if $0<\liminf_{N\to\infty} b_N/a_N\le\limsup_{N\to\infty}
b_N/a_N<\infty$.

\subsubsection{Gaussian case}\label{sec:6.1.1}
Here we have $\varepsilon_0(s)=s^2$,
\,$f_0(x)=\sqrt{\pi}\,\rme^{-\pi^2x^2}$, and the convolutions
$f_0^{\star j}$ \myp($j\in\NN$) are easily found,
\begin{equation}\label{eq:norm_int}
f_0^{\star j}(x)=\sqrt{\frac{\pi}{j}}\;\rme^{-\pi^2 x^2
j^{-1}},\qquad x\in\RR.
\end{equation}
Hence, the representation \eqref{eq:PSS_j} specializes to
\begin{equation}\label{eq:norm_sum_1}
\sum_{\bfk\in\ZZ^d} \rme^{- j\mypp \|\bfk\|^2/L^{2}}=\rho^{\mypp-1}N
\int_{\RR^d} \rme^{-j\mypp \|\bfs\|^2}\mypp\rmd{\bfs}
\cdot\left(1+\sum_{\ell\ne0} \rme^{-\pi^2 \ell^2 L^{2}/j} \right)^d,
\end{equation}
where $N=\rho\myp L^d$ and (see \eqref{eq:f*j(0)}
and~\eqref{eq:norm_int})
\begin{equation}\label{eq:Gauss_int}
\int_{\RR^d} \rme^{-j\mypp
\|\bfs\|^2}\mypp\rmd{\bfs}=\left(\int_{\RR} \rme^{-j
s^2}\mypp\rmd{s}\right)^d=\bigl\{f_0^{\star
j}(0)\bigr\}^d=\left(\frac{\pi}{j}\right)^{d/2}.
\end{equation}

The analysis of the expression \eqref{eq:norm_sum_1} is
straightforward. Recall that the index $j$ ranges from $1$ to $N$
(see \eqref{eq:Pspatial_with_partition}). As long as $j\myp
L^{-2}=o(1)$ (which is always true in dimension $d=1$), from
\eqref{eq:remainder} and \eqref{eq:norm_int} we obtain
$$
\Delta_N^{(j)}=\left(\frac{L^2}{j}\right)^{d/2}
O\bigl(\rme^{-\const\cdot L^{2}/j}\bigr)=o(1),\qquad N\to\infty,
$$
which means that the approximation \eqref{eq:approx_weigths},
\eqref{eq:kappa_j} is valid in this range of $j$, with
$\theta_j\equiv 0$. However, if $j\asymp L^{2}$ (when $d\ge2$) then
the sum on the left-hand side of \eqref{eq:norm_sum_1} is of order
of a constant; more specifically, if $jL^{-2}\to c>0$ then
\begin{equation}\label{eq:sum->const}
\sum_{\bfk\in\ZZ^d} \rme^{- j\mypp
\|\bfk\|^2/L^{2}}=\left(\sum_{k\in\ZZ} \rme^{- (j/L^2) \mypp
k^2}\right)^d\to \left(\sum_{k\in\ZZ} \rme^{- c\myp
k^2}\right)^d,\qquad L\to\infty.
\end{equation}
Note that the integral part in \eqref{eq:approx_weigths},
\eqref{eq:kappa_j} contributes to the limit \eqref{eq:sum->const}
the amount
\begin{equation}\label{eq:int_contribution}
L^d \left(\int_{\RR} \rme^{-j\myp
s^2}\mypp\rmd{s}\right)^d=L^d\left(\frac{\pi}{j}\right)^{d/2}\to
\left(\frac{\pi}{c}\right)^{d/2},
\end{equation}
whilst the rest of it must come from additional \emph{positive}
constants (see \eqref{eq:norm_sum_1}), leading accordingly to the
coefficients $\theta_j>0$ in \eqref{eq:approx_weigths}.

Similarly, if $jL^{-2}\to +\infty$ (which is possible in dimensions
$d\ge 3$) then
\begin{equation}\label{eq:sum->1}
\sum_{\bfk\in\ZZ^d} \rme^{- j\mypp
\|\bfk\|^2/L^{2}}=\left(1+2\sum_{k=1}^\infty \rme^{- (j/L^{2})\mypp
k^2}\right)^d\to 1,\qquad L\to\infty.
\end{equation}
Note that here the integral contribution is asymptotically vanishing
(cf.\ \eqref{eq:int_contribution}), so in this range of $j$ we must
have $\theta_j\sim \rme^{-\alpha_j}$, according to
\eqref{eq:approx_weigths} and~\eqref{eq:sum->1}.

\subsubsection{Stable case}\label{sec:6.1.2}
Here $\varepsilon_0(s)=|s|^\gamma$ with $0<\gamma<2$, which implies
that the (symmetric) density $f_0(x)$ is \emph{stable} (see, e.g.,
\cite[Theorem 2.2.2, p.~43]{IL}) and therefore unimodal
\cite[Theorem 2.5.3, p.~67]{IL}. The particular case $\gamma=1$
corresponds to the Cauchy distribution,
$f_0(x)=\pi^{-1}(1+x^2)^{-1}$. The convolutions $f_0^{\star j}(x)$
are easily found by rescaling,
$$
f_0^{\star j}(x)=j^{\myp-1/\gamma}
f_0\bigl(j^{\myp-1/\gamma}x\bigr),\qquad x\in\RR.
$$
In particular, note that (cf.\ \eqref{eq:f*j(0)})
\begin{equation}\label{eq:int_gamma}
f_0^{\star j}(0)=j^{\myp-1/\gamma} f_0(0)=j^{\myp-1/\gamma}\int_\RR
\rme^{-|s|^\gamma}\mypp\rmd{s}=
\frac{2\mypp\Gamma(1+1/\gamma)}{j^{1/\gamma}}\,,\qquad j\in\NN\myp.
\end{equation}
Hence, the representation \eqref{eq:PSS_j} takes the form
\begin{equation*}
\sum_{\bfk\in\ZZ^d} \rme^{- j\myp\varepsilon(\bfk/L)}
=\rho^{\mypp-1}N \int_{\RR^d}
\rme^{-j\myp\varepsilon(\bfs)}\,\rmd{\bfs}
\cdot\left(1+\frac{1}{f_0(0)}\sum_{\ell\ne0}
f_0(j^{\myp-1/\gamma}\ell L)\right)^d,
\end{equation*}
where (see \eqref{eq:f*j(0)} and \eqref{eq:int_gamma})
\begin{equation}\label{eq:stable-j}
\int_{\RR^d}
\rme^{-j\myp\varepsilon(\bfs)}\,\rmd{\bfs}=\left(\int_{\RR} \rme^{-j
\myp|s|^\gamma}\mypp\rmd{s}\right)^d=\bigl\{f_0^{\star
j}(0)\bigr\}^d=\bigl\{2\mypp\Gamma(1+1/\gamma)\bigr\}^{d}
\,j^{-d/\gamma}\myp.
\end{equation}

Furthermore, the tail asymptotics of the stable density $f_0(x)$ are
given by (see, e.g., \cite[Theorem 2.4.1, p.~54, for $0<\gamma<1$
and Theorem 2.4.2, p.~55, for $1<\gamma<2$]{IL})
$$
f_0(x)\sim \frac{1}{\pi
|x|^{1+\gamma}}\,\Gamma(1+\gamma)\,\sin\frac{\pi\gamma}{2}\mypp,\qquad
x\to\infty.
$$
(The case $\gamma=1$ is automatic in view of the explicit form of
the Cauchy density $f_0(x)$, as mentioned above.) Therefore, in the
range $j=o(L^{\gamma})$ the error term \eqref{eq:remainder} is
estimated as
\begin{equation}\label{eq:asymp_Delta}
\Delta_N^{(j)}\asymp N\myp j^{-d/\gamma} \sum_{\ell=1}^\infty
\frac{j^{\myp1+1/\gamma}}{\ell^{\myp1+\gamma}L^{1+\gamma}}\asymp
\left(\frac{L}{j^{1/\gamma}}\right)^{d-1-\gamma}, \qquad N\to\infty.
\end{equation}
Since $L\myp j^{\myp-1/\gamma} \to\infty$, the right-hand side of
\eqref{eq:asymp_Delta} tends to zero only if $d<1+\gamma$, which is
always true for $d=1$ but false for $d\ge3$ (where in fact
$\Delta_N^{(j)}\to\infty$); for $d=2$ we have $\Delta_N^{(j)}=o(1)$
if $\gamma>1$, while $\Delta_N^{(j)}\asymp 1$ if $\gamma=1$ and
$\Delta_N^{(j)}\to\infty$ if $0<\gamma<1$. More precisely, if
$j\asymp L^{\gamma-\epsilon}$ ($0<\epsilon\le \gamma$) then
$\Delta_N^{(j)}\asymp L^{\myp\epsilon\myp(d-1-\gamma)/\gamma}$,
which identifies the scale of $\Delta_N^{(j)}$ in the ``moderate''
range from $j=1$ (with $\Delta_N^{(1)}\asymp L^{d-1-\gamma}$ up to
$j\asymp L^{\gamma}$, when \eqref{eq:asymp_Delta} is formally
reduced to $\Delta_N^{(j)}\asymp 1$.

On the other hand, if $jL^{-\gamma}\to\infty$ then, using the
product formula \eqref{eq:product-sum} and the expression
$\varepsilon_0(s)=|s|^\gamma$, we see directly that
\begin{equation}\label{eq:sum->1'}
\sum_{\bfk\in\ZZ^d} \rme^{-
j\myp\varepsilon(\bfk/L)}=\left(1+2\sum_{k=1}^\infty \rme^{-
(j/L^\gamma)\mypp k^\gamma }\right)^d \to 1,\qquad L\to\infty.
\end{equation}

\subsubsection{Exponential-power case}\label{sec:6.1.3}
In this example, the density is specified as
$f_0(x)=\mu_0\mypp\rme^{-|x|^\gamma}$ \,($x\in\RR$), with
$0<\gamma<2$ and the normalization constant (cf.\
\eqref{eq:int_gamma})
\begin{equation}\label{eq:C0}
\mu_0\equiv \mu_0(\gamma)=\left(2\int_0^\infty
\!\rme^{-x^{\gamma}}\myp\rmd{x}
\right)^{-1}=\frac{1}{2\mypp\Gamma(1+1/\gamma)}\mypp.
\end{equation}
By the duality, from Section \ref{sec:6.1.2} we see that the Fourier
transform of $f_0(x)$ is given by $\varphi_0(s)=
\tilde{f}(s)/\tilde{f}(0)\ge0$, where $\tilde{f}(\cdot)$ is the
(stable) density with Fourier transform $\rme^{-|s|^\gamma}$.

\begin{remark}
The case $\gamma=2$, which corresponds to a Gaussian density (see
Section~\ref{sec:6.1.1}), is easily included in the analysis below.
\end{remark}

To estimate the error term \eqref{eq:remainder}, let us use Lemma
\ref{lm:unimodal} together with Remark \ref{rm:LD}, giving
\begin{equation}\label{eq:Delta<P}
\Delta_N^{(j)}\le L^d \left(\int_{\RR} \rme^{-j\myp
\varepsilon_0(s)}\,\rmd{s}\right)^d \left\{\left(1+\frac{\PP\{S_j\ge
L/2\}}{(L/4) \mypp f_0^{\star j}(0)}\right)^d-1\right\}.
\end{equation}
We have to distinguish two cases, (i) $1\le\gamma\le2$ and (ii)
$0<\gamma<1$.

\smallskip
(i) Consider the absolute moments of order $r\ge0$
(cf.~\eqref{eq:C0})
\begin{equation}\label{eq:mu_r}
\mu_r:=\int_\RR |x|^r f_0(x)\,\rmd{x}=2\myp \mu_0\int_0^\infty \mynn
x^r \rme^{-x^\gamma}\myp\rmd{x}=
\frac{\Gamma(1+(r+1)/\gamma)}{(r+1)\,\Gamma(1+1/\gamma)}\,.
\end{equation}
We will need a simple lemma about the gamma function.

\begin{lemma}\label{lm:Gamma(1+)}
For any $\gamma\in[1,2]$ and all integers $r\ge2$, the following
inequality holds
\begin{equation}\label{eq:Gamma_r}
\Gamma\mynn\left(1+\frac{r+1}{\gamma}\right)\le
\frac{(r+1)!}{6}\;\Gamma\mynn\left(1+\frac{3}{\gamma}\right)\!.
\end{equation}
\end{lemma}

\begin{proof} We argue by induction in $r$. For $r=2$ the claim
\eqref{eq:Gamma_r} is obvious (with the equality sign). Now, assume
that \eqref{eq:Gamma_r} holds for some $r\ge2$. Note that the gamma
function $\Gamma(t)$ is convex on $(0,\infty)$, because
$$
\Gamma''(t)=\int_0^\infty\! (\log x)^2 \mypp
x^{t-1}\,\rme^{-x}\,\rmd{x}>0,\qquad t>0.
$$
Since $\Gamma(1)=\Gamma(2)=1$, the convexity implies that
$\Gamma(t)$ is monotone increasing on $[2,\infty)$. On the other
hand, it is easy to check that if $\gamma\in[1,2]$ and $r\ge2$ then
\begin{equation}\label{eq:2<}
2\le \frac{r+2}{\gamma}\le 1+\frac{r+1}{\gamma}\mypp.
\end{equation}
Hence, by the monotonicity (using \eqref{eq:2<}) and the induction
hypothesis we obtain
\begin{align*}
\Gamma\mynn\left(1+\frac{r+2}{\gamma}\right)=\frac{r+2}{\gamma}\,\Gamma\mynn\left(\frac{r+2}{\gamma}\right)
&\le (r+2)\,\Gamma\mynn\left(1+\frac{r+1}{\gamma}\right)\\
&\le \frac{(r+2)!}{6}\;\Gamma\mynn\left(1+\frac{3}{\gamma}\right),
\end{align*}
which proves \eqref{eq:Gamma_r} for $r+1$. Thus, the lemma is valid
for all integer $r\ge2$.
\end{proof}

In view of the expressions \eqref{eq:mu_r}, the inequality
\eqref{eq:Gamma_r} yields the following estimate on the growth of
successive moments in the case $1\le \gamma\le 2$,
$$
\mu_r\le \tfrac12\, \mu_2\,r!\,,\qquad r\ge 2.
$$
Thus, one can apply Bernstein's inequality (in its enhanced modern
form, see \cite[Eq.\;(7), p.~38]{Bennett}), which gives for all
$j\in\NN$
\begin{equation}\label{eq:LD>1-B}
\PP\{S_j\ge L/2\}\le \exp\!\left(-\frac{L^2/4}{2j\myp\mu_2
+L}\right),\qquad L>0.
\end{equation}
Noting that $2j\myp\mu_2+L\le (2\myp\mu_2+1)\myp\max\{j,L\}$, from
\eqref{eq:LD>1-B} we easily deduce a more convenient estimate
\begin{equation}\label{eq:LD>1}
\PP\{S_j\ge L/2\}\le \exp\!\left(-\frac{L^2/j}{4\myp(2\myp\mu_2
+1)}\right)+\exp\!\left(-\frac{L}{4\myp(2\myp\mu_2 +1)}\right).
\end{equation}

\smallskip
(ii) If $0<\gamma<1$, we utilize a different suitable bound from the
large deviations theory (see \cite[Eq.\;(2.32), p.~764]{Nagaev})
yielding
\begin{equation}\label{eq:LD<1}
\PP\{S_j\ge L/2\}\le c \left(\exp\!\left(-\frac{L^2}{80\myp
j}\right)+j\,\PP\{X_1\ge L/4\}\right),\qquad L>0,
\end{equation}
where $c>0$ is a constant depending only on $\gamma$. Integrating by
parts it is easy to find
$$
\PP\{X_1\ge y\}=\mu_0\int_y^\infty\mynn \rme^{-u^\gamma}\mypp
\rmd{u}\sim
\frac{\mu_0}{\gamma}\,y^{1-\gamma}\rme^{-y^\gamma},\qquad
y\to+\infty.
$$
Hence, for $L\to\infty$ the bound \eqref{eq:LD<1} becomes
\begin{equation}\label{eq:j/L^2-gamma<1}
\PP\{S_j\ge L/2\}=O(1)\mypp\exp\!\left(-\frac{L^2}{80\myp
j}\right)+O\myp(jL^{1-\gamma})\myp\exp\!\left(-\frac{L^{\gamma}}{4^{\gamma}}\right).
\end{equation}

\smallskip
Returning to \eqref{eq:Delta<P}, from the estimates \eqref{eq:LD>1}
and \eqref{eq:j/L^2-gamma<1} we get, for any fixed $j\in\NN$\myp,
\begin{equation*}
\begin{array}{ll}\Delta_N^{(j)}= O(L^{d-1})
\exp\myn(-c_1\myn L),\qquad&1\le \gamma\le 2,\\[.4pc]
\Delta_N^{(j)}=O(L^{d-\gamma})\exp\myn(-c_2 L^{\gamma}),&0<\gamma<1,
\end{array}
\end{equation*}
with some constants $c_1,\mypp c_2>0$. In particular,
$\Delta_N^{(j)}$ is (exponentially) small as $L\to\infty$.

For $j\to\infty$, the asymptotics of $f_0^{\star j}(0)$ represented
as the integral \eqref{eq:f*j(0)} can be found using the Laplace
method \cite[Ch.\;4]{Bru}. Specifically, $\varepsilon_0(0)=0$ is the
unique minimum of $\varepsilon_0(s)=-\log \varphi_0(s)$; noting that
$\varphi'_0(0)=0$, \,$\varphi''_0(0)=-4\myp\pi^2\mu_2$ (see
\eqref{eq:mu_r}), we have $\varepsilon_0''(0)=4\myp\pi^2\mu_2$ and
hence
\begin{equation}\label{eq:f^j-Laplace}
f_0^{\star j}(0)=\int_\RR \rme^{-j\myp
\varepsilon_0(s)}\mypp\rmd{s}\sim
\sqrt{\frac{2\myp\pi}{j\mypp\varepsilon_0''(0)}}
=\frac{1}{\sqrt{\vphantom{j^j}\mypp2\myp\pi \mu_2 \myp
j\,}}\,,\qquad j\to\infty.
\end{equation}
Consequently, the estimates \eqref{eq:Delta<P}, \eqref{eq:LD>1} and
\eqref{eq:j/L^2-gamma<1} give
\begin{equation}\label{eq:j/L^2**}
\begin{array}{ll}
\displaystyle \Delta_N^{(j)}=O\bigl((Lj^{-1/2})^{d-1}\bigr)
\Bigl\{\exp\myn(-\tilde{c}_1 L^2j^{-1}) +
\exp\myn(-\tilde{c}_2 L\myp)\Bigr\}, &1\le \gamma\le 2,\\[.9pc]
\displaystyle \Delta_N^{(j)}=O\bigl((Lj^{-1/2})^{d-1}\bigr)
\Bigl\{\exp\myn(-\tilde{c}_3 L^2j^{-1})+
jL^{1-\gamma}\exp\myn(-\tilde{c}_4 L^\gamma)\Bigr\},\quad
&0<\gamma<1,
\end{array}
\end{equation}
where $\tilde{c}_i>0$ are some constants. Again, it is easy to see
from \eqref{eq:j/L^2**} that in all cases $\Delta_N^{(j)}=o(1)$ as
$L\to\infty$, provided that $L^2/j\to\infty$.

Consider now the opposite case where the index $j$ grows as $L^2$ or
faster. If $j\asymp L^2$ then the Poisson summation formula
\eqref{eq:PSS^d}, by virtue of Lemma \ref{lm:unimodal} and formula
\eqref{eq:f^j-Laplace}, yields
\begin{align}
\notag 1\le \sum_{\bfk\in\ZZ^d}^\infty
\rme^{-j\myp\varepsilon(\bfk/L)}&\le L^d\left(f_0^{\star j}(0) +
\frac{4}{L}\int_{0}^\infty\! f_0^{\star
j}(x)\,\rmd{x}\right)^d\\[-.3pc]
\notag
&=\left(L\int_\RR\rme^{-j\myp\varepsilon_0(s)}\mypp\rmd{s}
+ 2\right)^d\\[.2pc]
\label{eq:sum=O(1)} &=\bigl\{O(L\myp j^{-1/2})+2\bigr\}^d=O(1),
\end{align}
that is, $\Delta_N^{(j)}\asymp 1$ as $L\to\infty$.

In the remaining case where $j/L^2\to\infty$, observe that the
function $\varepsilon_0(s)$ is \emph{strictly increasing} in the
right neighbourhood of $s=0$, and moreover
$\varepsilon_0(s)\to+\infty$ as $s\to\infty$. This implies that
$0<\varepsilon_0(1/L)\le \varepsilon_0(k/L)$ for all $k\ge1$ (at
least for $L$ large enough). Hence,
\begin{align}
\notag
\sum_{k=1}^\infty
\rme^{-j\myp\varepsilon_0(k/L)}&=\sum_{k=1}^\infty
\rme^{-(j-L^2)\mypp\varepsilon_0(k/L)}\,\rme^{-L^2\myp\varepsilon_0(k/L)}\\
\label{eq:sum->1''}
&\le
\rme^{-(j-L^2)\mypp\varepsilon_0(1/L)}\sum_{k=1}^\infty\rme^{-L^2\myp\varepsilon_0(k/L)}
=\rme^{-(j-L^2)\mypp\varepsilon_0(1/L)} \cdot O(1),
\end{align}
where the $O(1)$-term appears according to the estimate
\eqref{eq:sum=O(1)} with $j=L^2$. Furthermore, using the expansion
$\varepsilon_0(s)=4\myp\pi^2 s^2+O(s^4)$ as $s\to0$ and recalling
that $j/L^2\to+\infty$, we see that
$(j-L^2)\,\varepsilon_0(1/L)\asymp jL^{-2}\to+\infty$. Hence, the
right-hand side of \eqref{eq:sum->1''} tends to zero and therefore
we get (cf.\ \eqref{eq:sum->1} and \eqref{eq:sum->1'})
\begin{equation*}
\sum_{\bfk\in\ZZ^d} \rme^{-
j\myp\varepsilon(\bfk/L)}=\left(1+2\sum_{k=1}^\infty
\rme^{-j\myp\varepsilon_0(k/L)}\right)^d \to 1,\qquad L\to\infty.
\end{equation*}

\subsubsection{Behaviour of the function\/ $\varepsilon(\bfs)$ at the origin}
\label{sec:6.1.4}

Betz and Ueltschi \cite[p.\:1176]{BeUe11a} work under the condition
that, with some $a>0$, $\delta>0$ and $0<\eta<d$,
\begin{equation}\label{eq:eps-near0}
\varepsilon(\bfs)\ge a\mypp\|\bfs\|^\eta,\qquad \|\bfs\|\le \delta,
\end{equation}
which guarantees that the critical density
\eqref{eq:Riemann-integral} is finite, $\rho_{\rm c}<\infty$.

For the examples considered in Sections \ref{sec:6.1.1} and
\ref{sec:6.1.3} we have $\varepsilon(\bfs)\sim \const\,\|\bfs\|^2$
as $\bfs\to\bfzero$, and so the condition \eqref{eq:eps-near0} is
fulfilled (with $\eta=2$) in dimensions $d\ge 3$.

For the ``stable'' example in Section \ref{sec:6.1.2}, the function
$\varepsilon(\bfs)=\sum_{i=1}^d |s_i|^\gamma$ ($0<\gamma<2$) is
comparable with $\|\bfs\|^\gamma$ due to the well-known fact that
any two norms in $\RR^d$ are equivalent; more explicitly, this
follows from the elementary inequalities (see \cite[Theorem 16,
p.~26, and Theorem 19, p.~28]{HLP})
$$
\|\bfs\|^\gamma \le \sum_{i=1}^d |s_i|^\gamma \le d^{\mypp
1-\gamma/2}\, \|\bfs\|^\gamma,\qquad \bfs=(s_1,\dots,s_d)\in\RR^d.
$$
Thus, here we have $\eta=\gamma<d$ unless $d=1$, $\gamma\ge1$.

Under the condition \eqref{eq:eps-near0} it is easy to justify the
universal behaviour of the Riemann sums in
\eqref{eq:Pspatial_with_partition} for large indices $j$ (even
without the condition of radial symmetry,
see~\eqref{eq:eps-isotropic}).

\begin{lemma}\label{lm:Sum->1}
Assume that \eqref{eq:eps-near0} is satisfied, and suppose that
$jL^{-\eta}\to\infty$. Then
\begin{equation}\label{eq:->1}
\sum_{\bfk\in\ZZ^d}\rme^{-j\myp\varepsilon(\bfk/L)}\to 1,\qquad
L\to\infty.
\end{equation}
\end{lemma}

\begin{proof} Note that $j\le N\asymp L^d$ and, since $\eta<d$, the range
of $j$ considered in the lemma is non-empty. Recalling that
$\varepsilon(\bfzero)=0$, for the proof of \eqref{eq:->1} it
suffices to show that the sum over $\bfk\ne\bfzero$ is
asymptotically small. With $\delta>0$ as in the condition
\eqref{eq:eps-near0} we have
\begin{equation}\label{eq:<sum+sum}
0\le\sum_{\bfk\ne\bfzero}\rme^{-j\myp\varepsilon(\bfk/L)}\le
\sum_{0<\|\bfk\|\le \delta L}\rme^{-j\myp
a\myp\|\bfk/L\|^\eta}+\sum_{\|\bfk\|>\delta
L}\rme^{-j\myp\varepsilon(\bfk/L)}.
\end{equation}

Using \eqref{eq:eps-near0} and the condition $j/L^\eta\to+\infty$,
we obtain by dominated convergence
\begin{equation}\label{eq:sum1}
\sum_{0<\|\bfk\|\le \delta L}\rme^{-j\myp
a\myp\|\bfk/L\|^\eta}\le\left(\sum_{0<k\le \delta L}\rme^{-a\myp
(j/L^\eta)\,k^\eta}\right)^d\to 0,\qquad L\to\infty,
\end{equation}
since $j\ge a L^\eta$ (for $L$ large enough) and
$\sum_{k=1}^\infty\rme^{-a\myp k^\eta}<\infty$.

To estimate the second sum in \eqref{eq:<sum+sum}, we can assume
that $\varepsilon(\bfs)\ge c_0>0$ for $\|\bfs\|>\delta$, hence,
owing to the bound \eqref{eq:int-eps},
\begin{align*}
\sum_{\|\bfk\|>\delta L}\rme^{-j\myp\varepsilon(\bfk/L)}&\le
\rme^{-(j-1)\myp c_0} \sum_{\|\bfk\|>\delta
L}\rme^{-\varepsilon(\bfk/L)}\\
&\le \rme^{-c_0 L^\eta} L^d\sum_{\|\bfk\|>\delta
L}\rme^{-\varepsilon(\bfk/L)}\,L^{-d}\\
&\sim \rme^{-c_0 L^\eta} L^d
\int_{\|\bfs\|>\delta}\rme^{-\varepsilon(\bfs)}\,\rmd{\bfs}\\[.4pc]
&=O\bigl(\rme^{-c_0 L^\eta} L^d\bigr)\to0,\qquad L\to\infty.
\end{align*}

By the estimates \eqref{eq:sum1} and \eqref{eq:sum1} the right-hand
side of \eqref{eq:<sum+sum} vanishes as $L\to\infty$, which
completes the proof.
\end{proof}

\subsubsection{Some heuristic conclusions}\label{sec:heuristic_conclusions}
Empirical evidence provided by the examples in Sections
\ref{sec:6.1.1}\mypp--\mypp\ref{sec:6.1.3} suggests that the
approximation picture is qualitatively universal in the class of
probability densities $f_0(x)$ with fast decaying tails (more
precisely, Gaussian as in Section~\ref{sec:6.1.1} or
exponential-power as in Section~\ref{sec:6.1.3}). Namely, here the
error $\Delta_N^{(j)}$ is asymptotically small if $j$ is fixed or
growing slower than $L^2\asymp N^{2/d}$; in the transition zone
$j\asymp L^{2}$, a bounded correction $\Delta_N^{(j)}\asymp 1$
emerges (comparable with the contribution of the ``main'' integral
term $N\kappa_j$ as defined in \eqref{eq:kappa_j}, cf.\
\eqref{eq:int_contribution} and \eqref{eq:f^j-Laplace}), whereas
with a faster growth of $j$ (possible in dimensions $d\ge3$) this is
transformed into the flat asymptotics $\Delta_N^{(j)}=1+o(1)$ (in
accordance with Lemma~\ref{lm:Sum->1}), but now with a polynomially
small error arising entirely due to the integral contribution,
$N\kappa_j\asymp (Lj^{-1/2})^d$ (see~\eqref{eq:f^j-Laplace}).

For polynomially decaying distribution tails as exemplified in
Section \ref{sec:6.1.2}, the situation is more complex: here, the
range $j=o(L^\gamma)$ produces an extended scale of the power
asymptotics $\Delta_N^{(j)}\asymp (L\myp
j^{-1/\gamma})^{d-1-\gamma}$, which is unbounded in sharp contrast
with the ansatz \eqref{eq:approx_weigths} unless $d=1$ or $d=2$,
$\gamma\ge1$. In the transition zone $j\asymp L^\gamma$ this is
naturally transformed into $\Delta_N^{(j)}\asymp 1$, and
furthermore, if $L\myp j^{-1/\gamma}=o(1)$ then we have the
universal (distribution-free) asymptotics $\Delta_N{(j)}=1+o(1)$,
similarly to the exponential tails (and again in line with
Lemma~\ref{lm:Sum->1}).

It should be clear from this discussion that the na\"ive use of any
specific surrogate-spatial model $\Pns$ as a proxy to the spatial
model $\Psp$ cannot be correct in the entire range of the cycle
lengths $j=1,\dots,N$. For instance, in the simplest Gaussian case,
taking $\theta_j\equiv 0$ works well for moderate values of $j$
(i.e., asymptotically smaller than $N^{2/d}$) but fails above this
threshold; on the other hand, choosing $\theta_j\equiv
\rme^{-\alpha_j}$ is adequate for cycles of size $j$ above $N^{2/d}$
but would incorrectly enhance the weighting of shorter cycles.

One might attempt to achieve a better approximation to $\Psp$ by
choosing the coefficients $\theta_j$ in formula
\eqref{eq:approx_weigths} so as to emulate the different asymptotics
of the correction term $\Delta_N^{(j)}$ (in particular, allowing
$\theta_j$ to depend on $L$). For instance, noting that the Gaussian
case is essentially characterized in terms of the natural
\emph{order parameter} $\eta_{j,L}:=L\myp j^{-1/2}$ (see
Section~\ref{sec:6.1.1}), the following phenomenological formula may
be suggested,
\begin{equation}\label{eq:Delta_modified}
\rme^{\myp\alpha_j}\myp\theta_{j,L}\propto \Theta_{\myn j,L}^{
d-1}\exp\myn(-\Theta_{\myn j,L}^{2}),\qquad j\in\NN\myp,
\end{equation}
where
\begin{equation}\label{eq:Theta}
\Theta_{\myn j,L}:=\frac{1}{1-\rme^{-1/\eta_{j\myn, L}}}\mypp,\qquad
j\in\NN\myp.
\end{equation}
Similarly, in the stable case (see Section~\ref{sec:6.1.2}) a
plausible approximation is given by
\begin{equation}\label{eq:Delta_modified-stable}
\rme^{\myp\alpha_j}\myp\theta_{j,L}\propto \Theta_{\myn
j,L}^{d-1-\gamma},\qquad j\in\NN\myp,
\end{equation}
where the parameter $\eta_{j,L}$ in \eqref{eq:Theta} is now
re-defined as $\eta_{j,L}:=L\myp j^{-1/\gamma}$ (see
\eqref{eq:asymp_Delta}).

The corresponding generating function $g_\theta(t)$ (see
\eqref{eq:def_g_theta_and_p_tau}) for the coefficients
\eqref{eq:Delta_modified} or \eqref{eq:Delta_modified-stable} may be
too complicated to deal with, but if we opt to ignore the
transitional details in the narrow zone $j\asymp L^2$ or $j\asymp
L^\gamma$, respectively, then we get much simpler heuristic formulas
\begin{equation*}
\rme^{\myp\alpha_j}\myp \theta_{j,L}\propto \left\{\begin{array}{ll}
\displaystyle 0,&
j\le L^2,\\[.3pc]
\displaystyle 1,& j>L^2,
\end{array}\right.\quad\ \text{and}\qquad
\rme^{\myp\alpha_j}\myp \theta_{j,L}\propto \left\{\begin{array}{cl}
\displaystyle (Lj^{-1/\gamma})^{d-1-\gamma},&
j\le L^\gamma,\\[.3pc]
\displaystyle 1,\qquad& j\ge L^\gamma.
\end{array}\right.
\end{equation*}
We intend to study such modifications of the surrogate-spatial model
in another paper.

\subsection{Density dependence in the surrogate-spatial model}\label{sec:6.2}

\subsubsection{Introducing an analogue of the particle density}
Although the surrogate-spatial model \eqref{eq:def_near_spatial},
\eqref{eq:HN} is defined with no reference to any underlying spatial
structure, an analogue of the density $\rho$ (cf.\
Section~\ref{sec:1.2}) can be incorporated in the system using the
expression \eqref{eq:kappa_j}, which provides a heuristic link
between the surrogate-spatial and spatial models. Namely, by analogy
with formula \eqref{eq:kappa_j}, let us write the coefficients
$\kappa_j$ in the form
\begin{equation}\label{eq:kappa*}
\kappa_j=\tilde{\rho}^{\mypp -1}\myp\check{\kappa}_j,\qquad
j\in\NN\myp,
\end{equation}
where the parameter $\tilde{\rho}>0$ is interpreted as ``density''
and the constants
\begin{equation}\label{eq:0kappa_j}
\check{\kappa}_j=\rme^{-\alpha_j} \int_{\RR^d} \rme^{-j
\myp\varepsilon(\bfs)}\,\rmd{\bfs},\qquad j\in\NN\myp,
\end{equation}
are treated as the baseline (density-free) coefficients that define
a specific subclass of the models~\eqref{eq:def_near_spatial}. For
the corresponding generating function this gives
\begin{equation}\label{eq:g-rho}
g_\kappa(z)=\tilde{\rho}^{\mypp -1}\sum_{j=1}^\infty
\frac{\check{\kappa}_j}{j}\,z^j=:\tilde{\rho}^{\mypp
-1}\mypp\check{g}_\kappa(z),
\end{equation}
hence
\begin{equation}\label{eq:g'-rho}
g_\kappa^{\{1\}}\myn(z)=\tilde{\rho}^{\mypp -1}\sum_{j=1}^\infty
\check{\kappa}_j\myp z^j=\tilde{\rho}^{\mypp
-1}\mypp\check{g}_\kappa^{\{1\}}\myn(z).
\end{equation}

\subsubsection{Critical density}\label{sec:crit_density}
At the singularity point $z=R$\myp, formula \eqref{eq:g'-rho}
specializes to
\begin{equation}\label{eq:g'-rhoR}
g_\kappa^{\{1\}}\myn(R\myp)=\tilde{\rho}^{\mypp -1}\sum_{j=1}^\infty
\check{\kappa}_j\myp R^j=\tilde{\rho}^{\mypp
-1}\myp\check{g}_\kappa^{\{1\}}\myn(R\myp).
\end{equation}
According to Definition \ref{def:sub-sup} (see also
Section~\ref{sec:cycles}), the critical case is determined by the
condition $g_\kappa^{\{1\}}\myn(R\myp)=1$; therefore,
\eqref{eq:0kappa_j} and \eqref{eq:g'-rhoR} imply that the critical
density is given by
\begin{equation}\label{eq:crit_density_surrogate}
\tilde{\rho}_{\myp \mathrm{c}}=\sum_{j=1}^\infty
\check{\kappa}_j\myp R^j=\sum_{j=1}^\infty R^j\,\rme^{-\alpha_j}
\int_{\RR^d} \rme^{-j \myp\varepsilon(\bfs)}\,\rmd{\bfs}.
\end{equation}
This is consistent with the sub- and supercritical regimes as
introduced in Definition \ref{def:sub-sup}:
$$
g_\kappa^{\{1\}}\myn(R\myp)>1\quad \Leftrightarrow\quad
\tilde{\rho}<\tilde{\rho}_{\myp \mathrm{c}}\mypp.
$$

Hence, we can express the expected fraction of points in infinite
cycles (see \eqref{eq:tilde-nu}) as
\begin{equation*}
\tilde\nu=\left\{\begin{array}{cl}
0,&\quad \tilde{\rho}\le \tilde{\rho}_{\myp \mathrm{c}}\myp,\\[.3pc]
\displaystyle 1- \frac{ \tilde{\rho}_{\myp \mathrm{c}}}{
\tilde{\rho}},&\quad \tilde{\rho}> \tilde{\rho}_{\myp
\mathrm{c}}\myp,
\end{array}
\right.
\end{equation*}
which exactly reproduces the formula \eqref{eq:rho-critical} for the
spatial model.

Under natural assumptions on the coefficients $(\alpha_j)$, the
expression \eqref{eq:crit_density_surrogate} recovers the formula
for the critical density $\rho_{\rm c}$ obtained in
\cite[Eq.\,(2.9), p.\:1177]{BeUe11a}.
\begin{lemma}\label{lm:r-cr<infty} Suppose that the coefficients
$\alpha_j$ satisfy the bounds
\begin{equation}\label{eq:e^alpha}
c_1 j^{-\delta}\le \rme^{-\alpha_j}\le c_2,\qquad j\in\NN\myp,
\end{equation}
with some positive constants $\delta$, $c_1$ and $c_2$. Then $R=1$
and formula \eqref{eq:crit_density_surrogate} is reduced to
\begin{equation*}
\tilde{\rho}_{\myp \mathrm{c}}=\sum_{j=1}^\infty \rme^{-\alpha_j}
\int_{\RR^d} \rme^{-j \myp\varepsilon(\bfs)}\,\rmd{\bfs}.
\end{equation*}
\end{lemma}

\begin{proof} Using the upper bound in \eqref{eq:e^alpha}, for any real
$r\in(0,R)$ we have
\begin{equation}\label{eq:Li_1}
\tilde{\rho}\, g_\kappa(r)\le c_2\int_{\RR^d}\sum_{j=1}^\infty
\frac{\rme^{-j \myp\varepsilon(\bfs)}}{j}\,r^j\,\rmd{\bfs}=
-c_2\int_{\RR^d} \log\myn\bigl(1-r\myp
\rme^{-\myp\varepsilon(\bfs)}\bigr) \,\rmd{\bfs}.
\end{equation}
Note that the right-hand side of \eqref{eq:Li_1} is finite due to
the bound \eqref{eq:int-eps},
\begin{align*}
\int_{\RR^d}\sum_{j=1}^\infty \frac{\rme^{-j
\myp\varepsilon(\bfs)}}{j}\,r^j\,\rmd{\bfs}&\le \int_{\RR^d} \rme^{-
\myp\varepsilon(\bfs)}\,\rmd{\bfs} \ \sum_{j=1}^\infty
\frac{r^j}{j}<\infty.
\end{align*}
Since $\varepsilon(\bfzero)=0$, the right-hand side of
\eqref{eq:Li_1} has singularity as $r\uparrow 1$. Therefore, by
Prings\-heim's Theorem (see Lemma~\ref{lm:Pringsheim}) the estimate
\eqref{eq:Li_1} implies that $R\ge 1$.

Similarly, by the lower bound in \eqref{eq:e^alpha} we get
\begin{equation}\label{eq:Li_1+delta}
\tilde{\rho}\, g_\kappa(r)\ge c_1\int_{\RR^d}\sum_{j=1}^\infty
\frac{\rme^{-j
\myp\varepsilon(\bfs)}}{j^{1+\delta}}\,r^j\,\rmd{\bfs}=c_1\int_{\RR^d}
\Li_{1+\delta}\bigl(r\myp\rme^{-\myp\varepsilon(\bfs)}\bigr)\,\rmd{\bfs}.
\end{equation}
By the known asymptotics of polylogarithm (see
Lemma~\ref{lm:polylog-sing}) the right-hand side of
\eqref{eq:Li_1+delta} has singularity at $r=1$ and it follows that
$R\le 1$. Thus $R=1$ and the lemma is proven.
\end{proof}

\begin{remark}
Assumption \eqref{eq:e^alpha} covers the cases considered in
\cite{BeUe11a} (see Section \ref{sec:comparison_1} below).
\end{remark}

\subsubsection{Total number of cycles}
Ansatz \eqref{eq:kappa*} also enables us to investigate the
$\tilde{\rho}$-dependence of the asymptotic statistics of cycles.
For instance, it is easy to check that the total number of cycles,
$T_N$, stochastically decreases with the growth of the density
$\tilde{\rho}$, as one would expect. Indeed, according to Corollary
\ref{cor:T-LLN-sub} and formula \eqref{eq:r*}, $T_N/N$ converges to
$g_\kappa(r_1)$ if $g_\kappa^{\{1\}}\myn(R\myp)\ge1$ \,(i.e.,
$\tilde{\rho}\le\tilde{\rho}_{\myp\mathrm{c}}$) or $g_\kappa(R\myp)$
if $g_\kappa^{\{1\}}\myn(R\myp)\le 1$ \,(i.e., $\tilde{\rho}\ge
\tilde{\rho}_{\myp\mathrm{c}}$), and to verify the claim it suffices
to show that the $\tilde{\rho}$-derivative of the limit is negative.

Differentiating with respect to $\tilde{\rho}$ and using the
representation \eqref{eq:g-rho}, we readily obtain
\begin{equation*}
\frac{\partial g_\kappa(R\myp)}{\partial\tilde{\rho}}=
-\tilde{\rho}^{\mypp-2}\myp\check{g}_\kappa(R\myp)=-\tilde{\rho}^{\mypp-1}g_\kappa(R\myp)<0.
\end{equation*}
Similarly,
\begin{equation}\label{eq:d/drho-sub}
\frac{\partial
g_\kappa(r_1)}{\partial\tilde{\rho}}=-\tilde{\rho}^{\mypp-1}g_\kappa(r_1)+g'_\kappa(r_1)\,
\frac{\partial\myp r_1}{\partial\tilde{\rho}}\mypp.
\end{equation}
On the other hand, differentiation of the equation
$g_\kappa^{\{1\}}\myn(r_1)=1$, rewritten for convenience as $r_1
\myp \check{g}^{\myp\prime}_\kappa(r_1)=\tilde{\rho}$, gives
$$
\frac{\partial\myp r_1}{\partial\tilde{\rho}}\,\tilde{\rho}\,
g'_\kappa(r_1)+r_1\myp \tilde{\rho}\,
g''_\kappa(r_1)\,\frac{\partial\myp r_1}{\partial\tilde{\rho}}=1,
$$
whence we find
\begin{equation}\label{eq:dr1}
\frac{\partial\myp
r_1}{\partial\tilde{\rho}}=\frac{\tilde{\rho}^{\mypp-1}}{g'_\kappa(r_1)+r_1
 g''_\kappa(r_1)}=\frac{\tilde{\rho}^{\mypp-1}\myp r_1}{1
 + g_\kappa^{\{2\}}\myn(r_1)}>0.
\end{equation}
Hence, returning to \eqref{eq:d/drho-sub} and again using the
identity $r_1\myp g'_\kappa(r_1)=g_\kappa^{\{1\}}\myn(r_1)\equiv 1$,
we get
\begin{align*}
\frac{\partial g_\kappa(r_1)}{\partial\tilde{\rho}}
&=-\tilde{\rho}^{\mypp-1}\!\left(g_\kappa(r_1)-\frac{1} {1+
 g_\kappa^{\{2\}}\myn(r_1)}\right)<0,
\end{align*}
where the inequality follows from Lemma \ref{lm:g<g}.

\subsubsection{Cycle counts}
Let us now investigate the asymptotic trend of the individual cycle
counts $C_j$ (for each $j\in\NN$) with the growth of the density
$\tilde{\rho}$. Assuming that all $\kappa_j>0$, by Theorem
\ref{thm:cycle_counts_normalized} and formula \eqref{eq:r*} we know
that $C_j/N$ converges to $\kappa_j\mypp r_1^j/j$ \,(for
$\tilde{\rho}\le\tilde{\rho}_{\myp\mathrm{c}}$) or $\kappa_j R^j/j$
\,(for $\tilde{\rho}\ge\tilde{\rho}_{\myp\mathrm{c}}$).

First, consider the supercritical domain,
$\tilde{\rho}>\tilde{\rho}_{\myp\mathrm{c}}$. Using the
representation \eqref{eq:kappa*} we obtain
\begin{align*}
\frac{\partial\myp(\kappa_j
R^j)}{\partial\tilde{\rho}}&=-\tilde{\rho}^{\mypp-2}\myp\check{\kappa}_j
R^j=-\tilde{\rho}^{\mypp-1}\kappa_j R^j<0,\qquad j\in\NN\myp,
\end{align*}
which means that the asymptotic proportion of cycles of any finite
length has the tendency to decrease with the growth of
$\tilde{\rho}$ (whilst the infinite cycle stays infinite).

In the subcritical domain
($\tilde{\rho}>\tilde{\rho}_{\myp\mathrm{c}}$), again invoking
\eqref{eq:kappa*} and also using formula \eqref{eq:dr1} for the
derivative $\partial\myp r_1/\partial\tilde{\rho}$, we get
\begin{align*}
\frac{\partial\myp(\kappa_j\mypp
r_1^j)}{\partial\tilde{\rho}}&=-\tilde{\rho}^{\mypp-2}\check{\kappa}_j
\mypp r_1^j+\kappa_j\myp j\, r_1^{j-1}\mypp\frac{\partial\myp
r_1}{\partial\tilde{\rho}} =-\tilde{\rho}^{\mypp-1}\kappa_j\mypp
r_1^j\left(1-\frac{j}{1+g_\kappa^{\{2\}}\myn(r_1)}\right).
\end{align*}
Thus, with the growth of the density $\tilde{\rho}$, as long as
$\tilde{\rho}<\tilde{\rho}_{\myp\mathrm{c}}$, the limiting
proportions of short cycles (with lengths
$j<1+g_\kappa^{\{2\}}\myn(r_1)$) decrease whereas those of longer
cycles (with lengths $j>1+g_\kappa^{\{2\}}\myn(r_1)$) increase.

Note, however, that the threshold $1+g_\kappa^{\{2\}}\myn(r_1)$
varies itself, and it is natural to expect that it is
\emph{increasing} with $\tilde{\rho}$, which is corroborated
heuristically by the limiting case $\tilde{\rho}\uparrow
\tilde{\rho}_{\myp\mathrm{c}}$, with $r_1\uparrow R$ and
$g_\kappa^{\{2\}}\myn(r_1)\uparrow
g_\kappa^{\{2\}}\myn(R\myp)=\max_{0\le r\le R}
\,g_\kappa^{\{2\}}\myn(r)$. More precisely, observing that
$$
\sum_{j=1}^\infty \check{\kappa}_j\mypp
r_1^j=\tilde{\rho}\,g_\kappa^{\{1\}}\myn(r_1)=\tilde{\rho}
$$
and
\begin{equation}\label{eq:frac:kappa/kappa}
1+g_\kappa^{\{2\}}\myn(r_1)=g_\kappa^{\{1\}}\myn(r_1)+
g_\kappa^{\{2\}}\myn(r_1)=\tilde{\rho}^{\mypp-1}\sum_{j=1}^\infty
j\myp\check{\kappa}_j\mypp r_1^j=\frac{\sum_{j=1}^\infty
j\myp\check{\kappa}_j\mypp r_1^j}{\sum_{j=1}^\infty
\check{\kappa}_j\mypp r_1^j}\mypp,
\end{equation}
we differentiate the right-hand side of \eqref{eq:frac:kappa/kappa}
to obtain
\begin{align*}
\frac{\partial\mypp
\bigl(1+g_\kappa^{\{2\}}\myn(r_1)\bigr)}{\partial\tilde{\rho}}
&=\tilde{\rho}^{\mypp -2}\myp r_1^{-1}\mypp \frac{\partial\myp
r_1}{\partial\tilde{\rho}}\cdot\left\{\sum_{j=1}^\infty
j^2\myp\check{\kappa}_j\mypp r_1^j\,\sum_{j=1}^\infty
\check{\kappa}_j \mypp r_1^j-\left(\sum_{j=1}^\infty
j\myp\check{\kappa}_j\mypp r_1^j\right)^2\right\} \ge0,
\end{align*}
according to \eqref{eq:dr1} and the Cauchy--Schwarz inequality (cf.\
Lemma~\ref{lm:g<g}).

\subsection{Comparison of the asymptotic results for long
cycles}\label{sec:comparison_1}

\subsubsection{Choosing a suitable surrogate-spatial model}
As was stressed in Section \ref{sec:heuristic_conclusions}, the
surrogate-spatial model $\Pns$ defined by
\eqref{eq:def_near_spatial}\mypp--\mypp\eqref{eq:HN} cannot
approximate correctly the spatial model $\Psp$
\eqref{eq:Pspatial_with_partition} in the entire range of the cycle
lengths $j=1,\dots,N$. However, if we focus on the asymptotics of
\emph{long cycles} only (i.e., with lengths $j\asymp N$), then the
discussion in Section \ref{sec:6.1} suggests the following choice of
the coefficients in the surrogate-spatial model,
\begin{equation}\label{eq:kappa-theta-ansatz}
\kappa_j=\rme^{-\alpha_j} \kappa^* j^{-s},\qquad
\theta_j=\rme^{-\alpha_j},\qquad j\in\NN\myp,
\end{equation}
with some index $s>0$. Here we suppress the dependence on the
density $\tilde{\rho}$ (cf.\ \eqref{eq:kappa*},
\eqref{eq:0kappa_j}), which is not essential for the comparison. The
expression \eqref{eq:kappa-theta-ansatz} for $\kappa_j$ bears on the
asymptotics of the integral \eqref{eq:product-int} as $j\to\infty$,
exemplified by the Gaussian and the exponential-power cases (both
with $s=d/2$, see \eqref{eq:Gauss_int} and \eqref{eq:f^j-Laplace},
respectively) and the stable case (with $s=d/\gamma$, see
\eqref{eq:stable-j}). The expression for $\theta_j$ in
\eqref{eq:kappa-theta-ansatz} picks up on the universal behaviour of
the correction term to the integral approximation of the Riemann sum
in \eqref{eq:Pspatial_with_partition} (for $j$ large enough, see
Lemma~\ref{lm:Sum->1}).


\begin{remark} It would be interesting to compare the measures
$\Psp$ and $\Pns$ with regard to the asymptotic statistics of
\emph{short cycles} (say, with fixed lengths $j=1,2,\dots$ as in
Theorem~\ref{thm:cycle_counts_normalized}). According to the
discussion in Section \ref{sec:6.1}, a better match between the two
models may be expected when the expression for $\theta_j$ in
\eqref{eq:kappa-theta-ansatz} is replaced by $\theta_j\equiv0$. The
asymptotics of the total number of cycles $T_N$ are also of
significant interest, especially in the critical case (cf.\
Theorem~\ref{thm:asymptotic_Kon}). However, such information is
currently not available under the spatial measure~$\Psp$
(see~\cite{BeUe11a} and further references therein).
\end{remark}

\subsubsection{Convergence to the Poisson--Dirichlet distribution}

For the coefficients $\alpha_j$ entering the definition of the
spatial measure $\Psp$ (see \eqref{eq:Pspatial_with_partition}),
Betz and Ueltschi \cite[p.\:1176]{BeUe11a} have considered
\emph{inter alia} the following two classes,
\begin{align}\label{eq:(i)}
\text{(i)}&\quad \lim_{j\to\infty}\alpha_j =
\alpha>0, \quad \sum_j
|\alpha_j-\alpha|<\infty;\\[-.2pc]
\label{eq:(ii)} \text{(ii)}&\quad \lim_{j\to\infty}\alpha_j =
\alpha\le 0,\quad \sum_j \frac{|\alpha_j-\alpha|}{j}<\infty.
\end{align}
With either of these assumptions, they prove for $\Psp$
\cite[Theorem 2.1\myp(b), p.\:1177]{BeUe11a} that in the
supercritical regime (i.e., $\rho>\rho_{\rm c}$), the ordered cycle
lengths $L^{(1)}\myn,L^{(2)}\myn,\dots$ (see
Definition~\ref{def:longest_cycle}) converge to the
Poisson--Dirichlet distribution with parameter $\rme^{-\alpha}$,
\begin{equation}\label{eq:PD_spatial}
\frac{1}{N \nu}\,(L^{(1)}\myn,L^{(2)}\myn,\dots)
\stackrel{d}{\longrightarrow} \PD(\rme^{-\alpha}),\qquad
N,L\to\infty.
\end{equation}

This resonates well with our Theorem~\ref{thm:large_cycles_1}.
Indeed, let the coefficients $\alpha_j$ have the form
\begin{equation}\label{eq:alpha_j=}
\alpha_j=\alpha-\log\myn(1+\xi(j)),\qquad j\in\NN\myp,
\end{equation}
where the function $\xi(z)$ satisfies the analyticity conditions of
Section \ref{sec:pert_polylog}, together with the estimate
$\xi(z)=O(z^{-\epsilon})$ \,($\epsilon>0$). The simplest example is
$\xi(j)=j^{-\epsilon}$, leading to
$\alpha-\alpha_j=\log\myn(1+j^{-\epsilon}) \sim j^{-\epsilon} \to0$
as $j\to\infty$. Let us stress that, as opposed to
\eqref{eq:(i)}\mypp--\mypp\eqref{eq:(ii)}, the sign of $\alpha$ in
\eqref{eq:alpha_j=} is not important, and also that the difference
$\alpha_j-\alpha$ satisfies the series convergence condition
\eqref{eq:(ii)}, but not necessarily \eqref{eq:(i)} (which only
holds for $\epsilon>1$).

With \eqref{eq:alpha_j=}, the coefficients
\eqref{eq:kappa-theta-ansatz} take the form
\begin{equation}\label{eq:kappa-theta-ansatz*}
\kappa_j=\rme^{-\alpha} \kappa^*\mypp
\frac{1+\xi(j)}{j^{s}}\myp,\qquad
\theta_j=\rme^{-\alpha}(1+\xi(j))\mypp,\qquad j\in\NN\myp.
\end{equation}
Let $s>1$, which ensures the existence of the supercritical regime
(see~\eqref{eq:xi-series-derivarive<1}), and suppose first that $s$
is non-integer, $q<s<q+1$ ($q\in\NN$); without loss of generality
(by reducing $\epsilon>0$ if necessary) we can assume that
$s+\epsilon<q+1$.
Then, using Lemma \ref{lm:pert_polylog} (more precisely, its part
(a) for $g_\kappa(z)$ and part (b) with $q=0$ and
$\kappa^*\myn=\rme^{-\alpha}$ for $g_\theta(z)$), it is not hard to
see that the generating functions $g_\kappa(z)$, $g_\theta(z)$
satisfy all the conditions of Theorem~\ref{thm:aux_asypmtotic_2},
including the asymptotic formulas \eqref{eq:g_theta_as} and
\eqref{eq:g_kappa_as} with $\theta^*\myn=\rme^{-\alpha}>0$. Hence,
Theorem~\ref{thm:large_cycles_1} may be applied, thus replicating
the convergence \eqref{eq:PD_spatial} for the surrogate-spatial
measure $\Pns$ (of course, with $\tilde\nu$ in place of~$\nu$).

The case of integer $s=q>1$ may be handled similarly, using the
suitable versions of Theorems~\ref{thm:aux_asypmtotic_2} and
\ref{thm:large_cycles_1} as indicated in
Remark~\ref{rm:thm5.11extension}.

\subsubsection{Emergence of a giant cycle} The third and last
specific class of the coefficients $\alpha_j$ considered by Betz and
Ueltschi \cite[p.\:1177]{BeUe11a} is given by
\begin{equation}\label{eq:(iii)}
\alpha_j = \gamma_0 \log j, \qquad j\in\NN\myp,
\end{equation}
where $\gamma_0>0$. Then it is proven \cite[Theorem 2.2\myp(b),
p.\:1178]{BeUe11a} that, under the supercritical spatial measure
$\Psp$\myp, there is asymptotically a single giant cycle,
\begin{equation}\label{eq:giant_cycle_spatial}
\frac{1}{N\nu}\, L^{(1)} \stackrel{d}{\longrightarrow} 1,\qquad
N,L\to\infty.
\end{equation}

This is directly analogous to the convergence \eqref{eq:->100} in
the statement of Theorem~\ref{thm:large_cycles_0}. Indeed,
substituting the expression \eqref{eq:(iii)} into
\eqref{eq:kappa-theta-ansatz} we get
\begin{equation}\label{eq:kappa-theta-ansatz**}
\kappa_j=\frac{\kappa^*}{j^{s+\gamma_0}}\mypp,\qquad
\theta_j=\frac{1}{j^{\gamma_0}}\mypp,\qquad j\in\NN\myp,
\end{equation}
so that the corresponding generating functions are given by
\begin{equation}\label{eq:kappa-theta-iii}
g_\kappa(z)=\kappa^*\Li_{s+\gamma_0+1}(z),\qquad
g_\theta(z)=\Li_{\gamma_0+1}(z).
\end{equation}
Although Theorem~\ref{thm:large_cycles_0} is not immediately
applicable here (because the function $g_\theta(z)$ has singularity
at $z=1$), the result \eqref{eq:->100} is valid for
\eqref{eq:kappa-theta-iii} in view of
Remark~\ref{rm:thm5.13extension}.

Moreover, we can go further and generalize the narrow class
\eqref{eq:(iii)} to
$$
\alpha_j=\gamma_0\log j-\log\myn(1+\xi(j)),\qquad j\in\NN\myp,
$$
thus replacing \eqref{eq:kappa-theta-ansatz**} by
(cf.~\eqref{eq:kappa-theta-ansatz*})
\begin{equation*}
\kappa_j=\kappa^*\mypp \frac{1+\xi(j)}{j^{s+\gamma_0}}\mypp,\qquad
\theta_j=\frac{1+\xi(j)}{j^{\gamma_0}}\mypp,\qquad j\in\NN\myp.
\end{equation*}
Then, using Lemmas \ref{lm:polylog-extension} and
\ref{lm:pert_polylog}, and deploying
Remark~\ref{rm:thm5.13extension}, we see that the convergence
\eqref{eq:->100} holds true.

\subsection{Summary of the comparison}\label{sec:6.4}

To wrap up the discussion in Section \ref{sec:comparison}, we have
demonstrated that, under natural conditions on the coefficients
$\kappa_j$ and $\theta_j$, the surrogate-spatial model successfully
reproduces the main features of the spatial model, including the
formulas for the critical density and the limiting fraction of
points in infinite clusters, as well as the asymptotic convergence
of the descending cycle lengths either to the Poisson--Dirichlet
distribution or to the degenerate distribution (with a single giant
cycle), depending on the asymptotic behaviour of the modulating
coefficients $\rme^{-\alpha_j}$ (i.e., convergent \emph{vs.}\
divergent $\alpha_j$'s); in our terms, this is translated into the
distinction between the type of singularity of the generating
function $g_\theta(z)$ (i.e., purely logarithmic or
power-logarithmic, respectively). Overall, our analysis shows that
the surrogate-spatial model, being of significant interest in its
own right, proves to be a flexible and efficient approximation of
the spatial model, providing at the same time a much greater
analytical tractability thus making it a useful exploratory tool.

\subsection*{Acknowledgements}
Partial support from SFB\,701 and ZiF (Bielefeld) is gratefully
acknowledged. L.V.\,B.\ was also partially supported by a Leverhulme
Research Fellowship and by the Hausdorff Research Institute for
Mathematics (Bonn). The authors would like to thank Joseph Najnudel,
Ashkan Nikeghbali and Anatoly Vershik for stimulating discussions on
the topic of random permutations. D.\,Z.\ is also grateful to Volker
Betz and Daniel Ueltschi for enlightening comments on the model of
spatial permutations.



\begin{thebibliography}{99}

\bibitem{ABT03}
Arratia, R., Barbour, A.D.\ and Tavar\'e, S. \textit{Logarithmic
Combinatorial Structures: a Probabilistic Approach}. EMS Monographs
in Mathematics. European Mathematical Society, Z\"urich, 2003.
\MR{2032426}

\bibitem{Bennett}
Bennett, G. Probability inequalities for the sum of independent
random variables. J.~Amer.\ Statist.\ Assoc.\ {\bf 57} (1962),
33--45.


\bibitem{BeUe09}
Betz, V.\ and Ueltschi, D. Spatial random permutations and infinite
cycles. \textit{Comm.\ Math.\ Phys.} {\bf 285} (2009), 469--501.
\MR{2461985}

\bibitem{BeUe10}
Betz, V. and Ueltschi, D. Critical temperature of dilute Bose gases.
\textit{Phys.\ Rev.\ A} {\bf 81} (2010), 023611.

\bibitem{BeUe11} Betz, V.\ and Ueltschi, D. Spatial random
permutations with small cycle weights. \textit{Probab.\ Theory
Related Fields} {\bf 149} (2011), 191--222. \MR{2773029}

\bibitem{BeUe11a} Betz, V.\ and Ueltschi, D. Spatial
random permutations and Poisson--Dirichlet law of cycle lengths.
\textit{Electron.\ J.\ Probab.} {\bf 16} (2011),
1173--1192. \MR{2820074}

\bibitem{BeUeVe11}
Betz, V., Ueltschi, D.\ and Velenik, Y. Random permutations with
cycle weights. \textit{Ann.\ Appl.\ Probab.} {\bf 21} (2011),
312--331. \MR{2759204}

\bibitem{BRR}
Bhattacharya, R.\,N.\ and Ranga Rao, R. \textit{Normal Approximation
and Asymptotic Expansions}, corrected printing. Robert E.~Krieger
Publ.\ Co.,
Malabar, FL, 1986. \MR{0855460}



\bibitem{Bochner}
Bochner, S.\ and Chandrasekharan, K. \textit{Fourier Transforms}.
Annals of Mathematics Studies, vol.\,19. Princeton University Press,
Princeton, NJ; Oxford University Press, London, 1949. \MR{0031582}

\bibitem{Bru} de Bruijn, N.G. \textit{Asymptotic Methods in
Analysis}, 2nd ed. Bibliotheca Mathematica, vol.\,IV. North-Holland,
Amsterdam; Noordhoff, Groningen, 1961. \MR{0177247}


\bibitem{ErJaUe14}
Ercolani, N.M., Jansen, A.\ and Ueltschi, D. Random partitions in
statistical mechanics. Preprint (2014),
\url{http://arxiv.org/abs/1401.1442} \,(last accessed 17.02.2014).

\bibitem{ErUe12}
Ercolani, N.M.\ and Ueltschi, D. Cycle structure of random
permutations with cycle weights. \textit{Random Structures
Algorithms} {\bf 44} (2014), 109--133. \MR{3143592}

\bibitem{Ewens}
Ewens, W.J. The sampling theory of selectively neutral alleles.
\textit{Theoret.\ Population Biology} {\bf 3} (1972),
87--112. \MR{0325177}

\bibitem{Feller}
Feller, W. \textit{An Introduction to Probability Theory and Its
Applications, Vol.~II}, 2nd ed. Wiley Series in Probability and
Mathematical Statistics.
Wiley, New York, 1971. \MR{0270403}


\bibitem{FlSe09}
Flajolet, P.\ and Sedgewick, R. \textit{Analytic Combinatorics}.
Cambridge University Press, New York, 2009. \MR{2483235}

\bibitem{GR}
Gradshteyn, I.S.\ and Ryzhik,  I.M. \textit{Table of Integrals,
Series, and Products}, 7th ed.
Elsevier/Academic Press, Amsterdam, 2007. \MR{2360010}


\bibitem{HLP}
Hardy, G.H., Littlewood, J.E.\ and P\'olya, G.
\textit{Inequalities}, 2nd ed. Cambridge Mathematical Library. At
the University Press, Cambridge, 1952. \MR{0046395}

\bibitem{IL}
Ibragimov, I.A.\ and Linnik, Yu.V. \textit{Independent and
Stationary Sequences of Random Variables}.
Wolters-Noordhoff, Groningen, 1971. \MR{0322926}


\bibitem{IsZa}
Ishwaran, H.\ and Zarepour, M. Exact and approximate sum
representations for the Dirichlet process. \textit{Canad.\ J.\
Statist.} {\bf 30} (2002),
269--283. \MR{1926065}

\bibitem{Kingman1975} Kingman, J.F.C. Random discrete
distributions.
\textit{J.~Roy.\ Statist.\ Soc.\ Ser.~B}
{\bf 37} (1975),
1--22. \MR{0368264}

\bibitem{Kingman}
Kingman, J.F.C. The population structure associated with the Ewens
sampling formula. \textit{Theoret.\ Population Biology} {\bf 11}
(1977), 274--283. \MR{0682238}


\bibitem{Kingman-Poisson}
Kingman, J.F.C. \textit{Poisson Processes}. Oxford Studies in
Probability, vol.~3. Clarendon Press, Oxford University Press,
Oxford, 1993. \MR{1207584}

\bibitem{Lewin} Lewin, L. \textit{Polylogarithms and Associated
Functions}. North-Holland, New York, 1981. \MR{0618278}



\bibitem{Mac95}
Macdonald, I.G. \textit{Symmetric Functions and Hall Polynomials},
2nd ed. Oxford Mathematical Monographs.
Oxford University Press, New York, 1995. \MR{1354144}

\bibitem{Man02}
Manstavi\v{c}ius, E. Mappings on decomposable combinatorial
structures: Analytic approach. \textit{Combin.\ Probab.\ Comput.}\
{\bf 11} (2002), 61--78. \MR{1888183}



\bibitem{MaNiZe12} Maples, K., Nikeghbali, A.\ and Zeindler, D.
On the number of cycles in a random permutation. \textit{Electron.\
Commun.\ Probab.} {\bf 17} (2012), no.~20, 13~pp. \MR{2943103}

\bibitem{Nagaev}
Nagaev, S.V. Large deviations of sums of independent random
variables. \textit{Ann.\ Probab.} {\bf 7} (1979), 745--789.
\MR{0542129}

\bibitem{NiZe13}
Nikeghbali, A.\ and Zeindler, D. The generalized weighted
probability measure on the symmetric group and the asymptotic
behavior of the cycles. \textit{Ann.\ Inst.\ H.~Poincar\'e Probab.\
Statist.} {\bf 49} (2013),
961--981. \MR{3127909}





\bibitem{Polya}
P\'olya, G. Kombinatorische Anzahlbestimmungen f\"ur Gruppen,
Graphen und chemische Verbindungen. (German) \textit{Acta Math.}
{\bf 68} (1937), 145--254; English transl.\ in: P\'olya, G.\ and
Read, R.C. \textit{Combinatorial Enumeration of Groups, Graphs, and
Chemical Compounds}.
Springer, New York, 1987, pp.\ 1--95. \MR{0884155}


\bibitem{Tavare}
Tavar\'e, S. The birth process with immigration, and the
genealogical structure of large populations. \textit{J.~Math.\
Biol.} {\bf 25} (1987), 161--168. \MR{0896431}

\bibitem{VeSh77}
Vershik, A.M.\ and Shmidt, A.A. Limit measures arising in the
asymptotic theory of symmetric groups.~I. (Russian) \textit{Teor.\
Veroyatnost.\ i Primenen.}\ {\bf 22} (1977),
72--88; English transl.\ in: \textit{Theory Probab.\ Appl.}\ {\bf
22} (1977),
70--85. \MR{0448476}

\bibitem{Watterson}
Watterson, G.A. The stationary distribution of the infinitely-many
neutral alleles diffusion model. \textit{J.~Appl.\ Probab.} {\bf 13}
(1976), 639--651. \MR{0504014}

\end{thebibliography}

\end{document}